\documentclass[10pt, reqno, draft]{amsart}

\usepackage{amsmath}
\usepackage{amssymb}
\usepackage{amsthm}
\usepackage{mathrsfs}	
\usepackage{bm}
\usepackage{comment}
\usepackage{color}
\usepackage{mathtools}
\usepackage{listings}
\usepackage{subfiles}

\mathtoolsset{showonlyrefs=true}

\allowdisplaybreaks[3]
\theoremstyle{plain} 
\newtheorem{theorem}{Theorem}[section]
\newtheorem{lemma}{Lemma}[section]
\newtheorem{corollary}{Corollary}[section]

\newtheorem{proposition}{Proposition}[section]

\newtheorem*{conjecture*}{Conjecture}
\newtheorem*{theorem*}{Theorem}

\theoremstyle{plain}

\theoremstyle{remark}
\newtheorem{remark}{Remark}[section]

\theoremstyle{definition}
\newtheorem*{assumption*}{Assumption}

\newtheorem*{notation*}{Notation}
\newtheorem*{acknowledgment*}{Acknowledgments}

\numberwithin{equation}{section}

\newcommand\swapcommand[2]{%
\let\swaptemp#1
\let#1#2
\let#2\swaptemp
}

\let\sl\l
\renewcommand\l{%
	\leavevmode
  \ifmmode
    \left
  \else
    \sl
  \fi}

\newcommand\set[2]{%
\left\{ #1 \; : \; #2 \right\}
}

\newcommand\norm[2][1]{
\| #2 \|_{#1}}

\newcommand\EXP[1]{
  \mathbb{E}\left[ #1 \right]
}

\swapcommand{\SS}{\S}
\renewcommand{\S}{\mathscr{S}}

\newcommand{\CC}{\mathbb{C}}
\newcommand{\RR}{\mathbb{R}}
\newcommand{\QQ}{\mathbb{Q}}
\newcommand{\ZZ}{\mathbb{Z}}

\newcommand{\PP}{\mathbb{P}}
\newcommand{\e}{\varepsilon}
\newcommand{\s}{\sigma}

\newcommand{\lam}{\lambda}
\newcommand{\Lam}{\Lambda}

\newcommand{\us}{\underset}

\newcommand{\qqquad}{\qquad \qquad \qquad}
\newcommand{\qqqquad}{\qquad \qquad \qquad \qquad}

\newcommand{\G}{\Gamma}

\newcommand{\T}{\mathscr{T}}

\newcommand{\Sc}{\mathcal{S}}

\newcommand{\ol}{\overline}

\newcommand{\meas}{\operatorname{meas}}

\newcommand{\sgn}{\operatorname{sgn}}

\renewcommand{\a}{\alpha}
\renewcommand{\b}{\beta}
\renewcommand{\r}{\right}

\renewcommand{\Re}{\operatorname{Re}}
\renewcommand{\Im}{\operatorname{Im}}
\renewcommand{\epsilon}{\varepsilon}

\renewcommand{\P}{\Phi}

\title[Joint value distribution of $L$-functions]{Joint value distribution of $L$-functions \\on the critical line}
\author[S. INOUE and J. Li]{Sh\={o}ta Inoue and Junxian Li}

\address[S. Inoue]{Graduate School of Mathematics, Nagoya University,
Furocho, Chikusaku, Nagoya 464-8602, Japan}
\email{m16006w@math.nagoya-u.ac.jp}
\address[J. Li]{Max Planck Institute for Mathematics,    
Vivatsgasse 7,
53111 Bonn,
Germany }
\email{jli135@mpim-bonn.mpg.de}



\begin{document}

\maketitle

\begin{abstract}
	In this paper, we discuss the joint value distribution of $L$-functions in a suitable class.
	We obtain joint large deviations results in the central limit theorem for these $L$-functions and some mean value theorems, which give evidence that different $L$-functions are ``statistically independent".
\end{abstract}

\section{\textbf {Introduction}}

The value distribution of $L$-functions plays an important role in analytic number theory.  A celebrated theorem of Selberg \cite{SCR} states that $\log \zeta(\frac{1}{2}+it)$ is approximately Gaussian distributed, namely,
\begin{align}\label{SCLT}
	\lim_{T\rightarrow \infty}\frac{1}{T}\meas\set{t \in [T, 2T]}{\frac{\log |\zeta(\tfrac{1}{2}+it)|}{\sqrt{\tfrac{1}{2}\log\log T}} > V}
	=\int_{V}^\infty e^{-\frac{u^2}{2}}\frac{du}{\sqrt{2\pi}}
\end{align}
for any fixed $V$. Throughout this paper, $\meas(\cdot)$ stands for the one-dimensional Lebesgue measure.
Selberg also stated that primitive $L$-functions in the Selberg class are ``statistically independent"
without precise description for the independence (see \cite[p.7]{S1992}).  Later,
Bombieri and Hejhal \cite{BH1995} indeed showed that normalized values of $L$-functions satisfying certain assumptions
behave like independent Gaussian variables, precisely,
\begin{align}\label{BHJVD}
	 & \lim_{T\rightarrow \infty}\frac{1}{T}\meas\set{t \in [T, 2T]}{
		\frac{\log |L_j(\frac{1}{2}+it)|}{\sqrt{\frac{n_{j}}{2}\log\log T}} > V_{j} \text{ for } j = 1, \dots, r}
	=\prod_{j=1}^r\int_{V_{j}}^\infty e^{-\frac{u^2}{2}}\frac{du}{\sqrt{2\pi}}
\end{align}
for any fixed $V_{1}, \dots, V_{r} \in \RR$ under certain conditions of $\{L_j\}_{j=1}^r$, 
where $n_{j}$ are certain positive integers determined by $L_{j}$.
It is natural to ask if \eqref{SCLT} and \eqref{BHJVD} still hold for a larger range of $V$.
In fact, if estimates similar to \eqref{SCLT} holds for a larger range of $V$,
then one could expect the growth of moments of $\zeta(\frac{1}{2} +it)$ as
\begin{align*}
	\int_T^{2T}|\zeta(\tfrac{1}{2}+it)|^{2k} dt
	\asymp_{k} T (\log T)^{k^2}.
\end{align*}
This idea has been used by Soundararajan \cite{SM2009}, where he proved the upper bound for all $k \geq 0$ conditionally on the Riemann Hypothesis, up to a factor of $(\log T)^\epsilon$. However, the asymptotic \eqref{SCLT} should not hold for arbitrarily large $V$, as observed by Radziwi\l\l\mbox{} \cite{Ra2011}, where he proved that \eqref{SCLT} holds for $V=o((\log\log T)^{1/10-\epsilon})$ and conjectured that \eqref{SCLT} holds for $V=o(\sqrt{\log\log T})$.

In this paper, we derivate large deviations results for joint value distribution of $L$-functions and give some further evidence of the ``statistical independence" of $L$-functions in a general class, including the Riemann zeta function, Dirichlet $L$-functions, and $L$-functions attached to holomorphic and Maa\ss \   cusp forms. We obtain the expected upper bound \eqref{BHJVD} for
$V = o\l((\log\log T)^{1/6}\r)$ (see Therorem  \ref{Main_Thm_LD_JVD}).
We also obtain weaker bounds when $V\ll \sqrt{\log\log T}$ (see Theorem \ref{GJu} and Theorem \ref{RH_LD_JVD}). As an application,  we have the following result.
\begin{theorem}\label{ThmDirichlet}
	Let $\{\chi_j\}_{j=1}^r$ be distinct primitive Dirichlet characters (where the case $\chi_{i} = 1$ for some $1 \leq i \leq r$ is allowed).
	Then there exists some positive constant $B$ depending on $\{\chi_j\}_{j=1}^r$ such that for any fixed sufficiently small positive real number $k$,
	\begin{align}
		\label{EMVTRD}
		\int_{T}^{2T}\l( \min_{1\leq j \leq r}   |L(\tfrac{1}{2}+it, \chi_j)| \r)^{2k}dt
		\ll T(\log{T})^{k^2/r +  Bk^{3}},
	\end{align}
	\begin{align} 
    \label{EMVTRD2}
		\int_{T}^{2T}\l( \max_{1\leq j \leq r}   |L(\tfrac{1}{2}+it, \chi_j)| \r)^{-2k}dt
		\gg T(\log{T})^{k^2/r - Bk^{3}},
	\end{align}
	\begin{align}
		T(\log{T})^{k^2/r +  Bk^{3}}
		\ll \int_{T}^{2T}\exp\l( 2k\min_{1\leq j \leq r} \Im\log L(\tfrac{1}{2}+it, \chi_j)\r)dt
		\ll T(\log{T})^{k^2/r +  Bk^{3}},
	\end{align}
	and
	\begin{align}
		T(\log{T})^{k^2/r - B k^{3}}
		\ll \int_{T}^{2T}\exp\l( -2k\max_{1\leq j \leq r} \Im\log L(\tfrac{1}{2}+it, \chi_j)\r)dt
		\ll T(\log{T})^{k^2/r + B k^{3}}.
	\end{align}
	The above implicit constants depend on $\{\chi_j\}_{j=1}^r$.
	If we assume the Riemann Hypothesis for $L(s, \chi_j)$ for all $j$, then we have for any $k > 0$ and $\e > 0$,
	\begin{align}
		\int_{T}^{2T}\l( \min_{1\leq j \leq r}|L(\tfrac{1}{2}+it, \chi_j)| \r)^{2k}dt
		\ll T(\log{T})^{k^2/r + \e},
	\end{align}
	\begin{align}
		\int_{T}^{2T}\l( \max_{1\leq j \leq r}   |L(\tfrac{1}{2}+it, \chi_j)| \r)^{-2k}dt
		\gg T(\log{T})^{k^2/r - \e},
	\end{align}
	\begin{align} \label{EMVTRDRH3}
		T(\log{T})^{k^2/r - \e}
		\ll \int_{T}^{2T}\exp\l( 2k\min_{1\leq j \leq r} \Im\log L(\tfrac{1}{2}+it, \chi_j)\r)dt
		\ll T(\log{T})^{k^2/r + \e},
	\end{align}
	and
	\begin{align} \label{EMVTRDRH4}
		T(\log{T})^{k^2/r - \e}
		\ll \int_{T}^{2T}\exp\l( -2k\max_{1\leq j \leq r} \Im\log L(\tfrac{1}{2}+it, \chi_j)\r)dt
		\ll T(\log{T})^{k^2/r + \e}.
	\end{align}
	Here, the above implicit constants depend on $\{\chi_j\}_{j=1}^r$, $k$, and $\e$.
\end{theorem}
\begin{remark}
Theorem \ref{ThmDirichlet} is a special case for Theorem \ref{New_MVT} and Theorem \ref{New_MVT_RH}, which are applicable to more general $L$-functions such as $L$-functions associated with $GL(2)$ cusp forms. 
\end{remark}
\begin{remark}
	It is now known that for $0\leq k \leq 2$
	\begin{align*}
		\int_{T}^{2T}|\zeta(\tfrac{1}{2}+it)|^{2k}dt
		\asymp_{k} T(\log{T})^{k^2},
	\end{align*}
	by the works of Heap-Radziwi\l\l-Soundararajan \cite{HRS2019} and Heap-Soundararajan \cite{HeapSound},
	so the unconditional result \eqref{EMVTRD} is nontrivial when $k$ is sufficiently small and $0 < k < B^{-1}/2$, 
  if one of the $L$-functions is chosen to be the Riemann zeta-function and $r\geq 2$.
  In particular, when $k$ is sufficiently small and $0 < k < B^{-1}/2$, we obtain the relation
  \begin{align*}
    \int_{T}^{2T}\l( \min\l\{ |\zeta(\tfrac{1}{2} + it)|, |L(\tfrac{1}{2} + it, \chi)| \r\} \r)^{2k}dt
    = o\l( \int_{T}^{2T}|\zeta(\tfrac{1}{2} + it)|^{2k}dt \r) \quad \text{as $T \rightarrow + \infty$}.
  \end{align*}
  This relation also holds for any $k > 0$ if the Riemann Hypothesis for $\zeta(s)$ and $L(s, \chi)$ is true.
\end{remark}
\begin{remark}
	The conditional upper bound for the $2k$-th moment for $\Im \log \zeta(\frac{1}{2}+it)$ was recently proved by Najnudel \cite{Naj} 
	using a different argument, but the lower bound remained unproved.
  Estimates \eqref{EMVTRDRH3}, \eqref{EMVTRDRH4} give the lower bound and also recover the original Najnudel's result.
\end{remark}
\begin{remark}
  Gonek \cite{G1989} showed a lower bound of the negative moment of the Riemann zeta-function under the Riemann Hypothesis.
  On the other hand, there were no unconditional results for the lower bound.
  Estimate \eqref{EMVTRD2} gives an unconditional result for the negative moment for sufficiently small $k$.
\end{remark}

\subsection{Proof sketch and outline of the paper}
Our proof starts with approximate formulae for $L$-functions that combines the formula of Selberg \cite{SCR} and the hybrid formula of Gonek-Hughes-Keating \cite{GHK2007} developed for the study of $\zeta(s)$. These formulae essentially contain a Dirichlet polynomial involving primes and explicit expressions involving zeros of these $L$-functions. After examining the contribution from the zeros, the joint value distribution of $L$-functions can be reduce to the joint value distribution of Dirichlet polynomials. To make clear of the scope of our methods,  we work within a general class of $L$-functions satisfying suitable assumptions. These assumptions can be verified for Dirichlet $L$-functions and $L$-functions associated with $GL(2)$ cusp forms. Note that we do not assume the Ramanujan conjecture, but Hypothesis H of Rudnick-Sarnak \cite{RS1996}. The weaker assumption makes our proof different in several aspects (see Section \ref{appL} and \ref{dpoly}).

The rest of the paper is organized as follows: In Section \ref{General} we give the definition of the class of $L$-functions and state our results for these $L$-functions and their corresponding Dirichlet polynomials. 
In Section \ref{appL} we give the approximate formulas for $L$-functions in our class.
In Section \ref{dpoly} we give the proofs of results for Dirichlet polynomials.
In Section \ref{uncond} and Section \ref{cond} we give the proof for results for $L$-functions
without and with the Riemann Hypothesis respectively. In Section \ref{finalremarks}, we give some remarks on other applications of our method.

\section*{acknowledgement}
The authors would like to thank Winston Heap for his comments and remarks on earlier version of the paper.
The authors also thank Kenta Endo and Masahiro Mine for their useful comments.
The first author is supported by Grant-in-Aid for JSPS Research Fellow (Grant Number: 19J11223). The second author thanks the Max Planck Institute for the financial support when attending the conference on zeta functions in
Centre International de Rencontres Math\'ematiques, where the project was started.

\section{\textbf{Definitions and Results} }\label{General}

We consider the class $\Sc^\dagger$ consisting of $L$-functions $F(s)$
that satisfy the following conditions (S1)--(S5).
\begin{enumerate}
	\setlength{\parskip}{2mm}

	\item [(S1)] (Dirichlet series) $F(s)$ has a  Dirichlet series ${F(s) = \sum_{n = 1}^{\infty}\frac{a_{F}(n)}{n^{s}}}$, 
        which converges absolutely for $\s > 1$.

	\item[(S2)] (Analytic continuation) \quad
	      There exists $m_F \in \ZZ_{\geq 0}$ such that
	      $(s - 1)^{m_F}F(s)$ is entire of finite order.

	\item[(S3)](Functional equation) \quad
	      $F(s)$ satisfies the following functional equation
	      \begin{align*}
		      \P_{F}(s) = \omega_F \ol{\P_{F}}(1 - s),
	      \end{align*}
	      where $\P_{F}(s) = \gamma(s)F(s)$ and
	      $
		      \gamma(s)= Q^{s}\prod_{\ell = 1}^{k}\G(\lam_{\ell}s + \mu_{\ell}),
	      $ with
	      $\lam_{\ell} > 0$, $Q > 0$, $\Re \mu_{\ell} \geq -\frac{1}{2}$,
	      and $|\omega_F |= 1$. Here we use the notation $ \ol{\P_{F}}(s)=\ol{\P_F(\ol s)}$.
        The quantity $d_{F}=2\sum_{\ell = 1}^{k}\lam_{\ell}$ is often called the degree of $F$.


	\item[(S4)] (Euler product)
	      $F(s)$ can be written as
	      \begin{align*}
		      F(s) = \prod_{p}F_{p}(s),
		      \quad
		      F_{p}(s) = \exp\l( \sum_{\ell = 1}^{\infty}\frac{b_{F}\l( p^{\ell} \r)}{p^{\ell s}} \r),
	      \end{align*}
	      where
	      $b_{F}(n) = 0$ unless $n = p^{\ell}$ with $\ell \in \ZZ_{\geq 1}$, and
	      $b_{F}(n) \ll n^{\vartheta_{F}}$ for some $\vartheta_{F} < \frac{1}{2}$.


	\item[(S5)](Hypothesis H)
	      For any $\ell \geq 2$,
	      \begin{align}
		      \label{HHRS}
		      \sum_{p}\frac{|b_{F}(p^{\ell}) \log{p^{\ell}}|^2}{p^{\ell}} < +\infty.
	      \end{align}
\end{enumerate}
This class of $L$-functions is different from the Selberg class in that we do not assume the Ramanujan Conjecture,
which says 
\begin{itemize}
\item [(S3')] $F(s)$ has the same functional equation as in (S3)
      with the condition $\Re \mu_{\ell} \geq 0$ for $1\leq \ell \leq k$.
\item[(S5')] 
	      For every $\e > 0$, the inequality
	      $a_{F}(n) \ll_{F} n^{\e}$ holds.
\end{itemize}
The set of $L$-functions satisfying (S1), (S2), (S3'), (S4), and (S5') is called the Selberg class $\Sc$. It is believed that automorphic $L$-functions of $GL(n)$ are in the Selberg class $\Sc$, but this hasn't been proved due to the lack of the Ramanujan Conjecture. The weaker conditions in $\mathcal S^{\dagger}$ allow us to include more automorphic $L$-functions of $GL(n)$ unconditionally.
The properties (S1)-(S4) for automorphic $L$-functions of $GL(n)$ as well as 
their Rankin-Selberg convolutions can be found in \cite[Section 2]{RS1996}. 
Hypothesis H have been proved for $1\leq n\leq 4$ 
( The case for $n = 1$ is trivial and the case for $n=2$ follows from the work of Kim-Sarnak \cite{KimSarnak}. 
See Rudnick-Sarnak \cite[Proposition 2.4]{RS1996} for $n= 3$ and Kim \cite{Kim06} for $n=4$). 


To study the joint value distribution of $L$-functions in $\mathcal{S}^\dagger$, we need the following assumption $\mathscr A$.
\begin{assumption*}[$\mathscr{A}$]
  An $r$-tuple of Dirichlet series $\bm{F} = (F_{1}, \dots, F_{r})$
	and an $r$-tuple of the numbers $\bm{\theta} = (\theta_{1}, \dots, \theta_{r}) \in \RR^{r}$
	satisfy $\mathscr{A}$ if and only if
	$\bm{F}$, $\bm{\theta}$ satisfy the following properties.
	\begin{itemize}
		\item[(A1)] (Selberg Orthonormality Conjecture) \quad
		      For any $F_j$, we have
		      \begin{align}	\label{SNC}
			      \sum_{p \leq x}\frac{|a_{F_{j}}(p)|^2}{p}
			      = n_{F_{j}} \log{\log{x}} + O_{F_{j}}(1), \quad x\rightarrow \infty.
					      \end{align}
		      for some constant $n_{F_{j}}>0$.
		      For any pair $F_{i} \not= F_{j}$,
          \begin{align}
            \label{SOC2}
			      \sum_{p \leq x}\frac{a_{F_{i}}(p) \ol{a_{F_{j}}(p)}}{p}
			      = O_{F_{i}, F_{j}}(1), \quad x\rightarrow \infty.
		      \end{align}
		\item[(A2)] For every $i$,  there is at most one $j\not=i$ such that $F_{i} = F_{j}$ and in this case $|\theta_{i} - \theta_{j}| \equiv \pi/2\pmod{ \pi }$.
		\item[(A3)](Zero Density Estimate)
		      For every $F_{j}$, there exists a positive constant $\kappa_{F_{j}}$ such that, uniformly for any $T \geq 3$ and $1/2 \leq \s \leq 1$,
		      \begin{align}	\label{ZDC1}
			      N_{F_{j}}(\s, T) \ll_{F_{j}} T^{1 - \kappa_{F_{j}}(\s - 1/2)}\log{T},
		      \end{align}
		      where $N_{F}(\s, T)$ is the number of nontrivial zeros $\rho_{F} = \b_{F} + i\gamma_{F}$ of $F \in \Sc^{\dagger}$
		      with $\b_{F} \geq \s$ and $0 \leq \gamma_{F} \leq T$.
	\end{itemize}
\end{assumption*}

\begin{remark}
	The Selberg Orthonormality Conjecture has been proved
	for  $L$-functions associated with cuspidal automorphic representations of $GL(n)$  unconditionally for $n\leq 4$ by Liu-Ye in \cite{LiuYe},
	(See \cite{AS, Liu05}) and in full generality in \cite{LiuYe06}.
\end{remark}

\begin{remark}
	It is natural to assume (A2). This allows us to  consider the joint distribution of $\log |F(s)|$ and $\Im\log F(s)$.
	It can be seen that $\Re e^{-i\theta_1}\log F(s)$ and $\Re e^{-i\theta_2}\log F(s)$ can not be independent 
  when $|\theta_1-\theta_2| \not\equiv \frac{\pi}{2} \pmod{\pi}$.
\end{remark}
\begin{remark} We require a strong zero density estimate (the exponent of $\log T$ is $1$) close to the line $\sigma=1/2$. 
	Results of the shape \eqref{ZDC1} have been obtained for the Riemann zeta-function by Selberg \cite{SCR}  and for Dirichlet $L$-functions by Fujii \cite{F1974}.
	For $GL(2)$ $L$-functions,  (A3) has been established by Luo \cite{Luo95} for holomorphic cusp forms
	of the full modular group (see \cite[Section 7]{FZ} for other congruence subgroups). 
  The proof should also be applicable for Maa\ss\  forms (see the remark in \cite[p 141]{Luo95}, see also a weaker estimate in \cite{SSmaass}).
	If we assume the Riemann Hypothesis for $F_j$, then \eqref{ZDC1} holds for any $\kappa_{F_j}>0$.
\end{remark}

 Before we state our theorems, we need some notation.
Let $r$ be a fixed positive integer. For $\bm V=(V_1, \dots, V_r)\in \mathbb{R}^r$, 
$\bm \theta=(\theta_1, \dots, \theta_r) \in \RR^{r}$, and $\bm F=(F_1, \dots, F_r)\in (\mathcal{S}^\dagger)^r$ 
satisfying assumption $\mathscr A$, we define
\begin{align}
	\S(T, \bm{V}; \bm{F}, \bm{\theta})
	:= \set{t \in [T, 2T]}
	{\frac{\Re e^{-i\theta_{j}}\log{F_{j}(\frac{1}{2} + it)}}{
			\sqrt{\frac{n_{F_{j}}}{2}\log{\log{T}}}} > V_j \text{ for } j = 1, \dots, r},
\end{align}
where the constants $n_{F_j}$ are defined in \eqref{SNC}. Let $\a_{\bm{F}} := \min\{ 2r, \frac{1 - 2\vartheta_{\bm{F}}}{2\vartheta_{\bm{F}}} \}$, where $\vartheta_{\bm{F}}=\max_{1\leq j\leq r}\vartheta_{F_{j}}$ as defined in (S4).
Here $\a_{\bm{F}} = 2r$ if $\vartheta_{\bm{F}} = 0$.
We denote $\norm[]{\bm{z}}=\max_{1 \leq j \leq r}|z_{j}|$.
Throughout this paper, we write $\log_{3}x$ for $\log{\log{\log}}x$.

\subsection{Results for large deviations}

The following theorem extends the result of Bombieri and Hejhal \cite{BH1995},
where we show \eqref{BHJVD} holds for a larger range of $V$.
\begin{theorem}	\label{Main_Thm_LD_JVD}
	Let $\bm{\theta} = (\theta_{1}, \dots, \theta_{r}) \in \RR^{r}$ and
	$\bm{F} = (F_1, \dots, F_{r}) \in (\Sc^{\dagger} )^{r}$ satisfy assumption $\mathscr{A}$.
	Let $A \geq 1$ be a fixed constant.
	For any large $T$ and any $\bm{V} = (V_{1}, \dots, V_{r}) \in \RR^{r}$
	with $\norm[]{\bm{V}} \leq A(\log{\log{T}})^{1/10}$, we have
	\begin{align}	\label{Main_Thm_LD_JVD1}
		 & \frac{1}{T}\meas(\S(T, \bm{V}; \bm{F}, \bm{\theta}))                                                               \\
		 & =\l( 1 + O_{\bm{F}, A}\l(\frac{(\norm[]{\bm{V}}^{4} + (\log_{3}{T})^2)(\norm[]{\bm{V}} + 1)}{\sqrt{\log{\log{T}}}}
			+ \frac{\prod_{k = 1}^{r}(1 + |V_{k}|)}{(\log{\log{T}})^{\a_{\bm{F}} + \frac{1}{2}}}\r)\r)
		\prod_{j = 1}^{r}\int_{V_j}^{\infty} e^{-\frac{u^2}{2}}\frac{du}{\sqrt{2\pi}}.
	\end{align}
	Moreover, if $\bm \theta \in  [-\frac{\pi}{2}, \frac{\pi}{2}]^r$ and $\|\bm V\| \leq A (\log{\log{T}})^{1/6}$
	we have
	\begin{align}	\label{Main_Thm_LD_JVD2}
		 & \frac{1}{T}\meas(\S(T, \bm{V}; \bm{F}, \bm{\theta}))                                                                   \\
		 & \leq\l ( 1 + O_{\bm{F}, A}\l(\frac{(\norm[]{\bm{V}}^{2} + (\log_{3}{T})^2)(\norm[]{\bm{V}} + 1)}{\sqrt{\log{\log{T}}}}
			+ \frac{\prod_{k = 1}^{r}(1 + |V_{k}|)}{(\log{\log{T}})^{\a_{\bm{F}} + \frac{1}{2}}}\r)\r)
		\prod_{j = 1}^{r}\int_{V_j}^{\infty} e^{-\frac{u^2}{2}}\frac{du}{\sqrt{2\pi}},
	\end{align}
	and if $\bm \theta \in [\frac{\pi}{2}, \frac{3\pi}{2}]^r$, $\|\bm V\| \leq  a_1(\log{\log{T}})^{1/6}$, and
	$\prod_{j = 1}^{r} (1 + |V_{j}|) \leq a_1(\log{\log{T}})^{\a_{\bm{F}} + \frac{1}{2}}$ with $a_1= a_1(\bm{F}) > 0$ small enough, we have
	\begin{align}	\label{Main_Thm_LD_JVD3}
		 & \frac{1}{T}\meas(\S(T, \bm{V}; \bm{F}, \bm{\theta}))                                                                 \\
		 & \geq \l ( 1 - O_{\bm{F}}\l(\frac{(\norm[]{\bm{V}}^{2} + (\log_{3}{T})^2)(\norm[]{\bm{V}} + 1)}{\sqrt{\log{\log{T}}}}
			+ \frac{\prod_{k = 1}^{r}(1 + |V_{k}|)}{(\log{\log{T}})^{\a_{\bm{F}} + \frac{1}{2}}}\r)\r)
		\prod_{j = 1}^{r}\int_{V_j}^{\infty} e^{-\frac{u^2}{2}}\frac{du}{\sqrt{2\pi}}.
	\end{align}
\end{theorem}
\begin{remark}
	For $r=1$, $F_1=\zeta$, and $\theta_{1} = 0$, the asymptotic for $\|\bm V\|\ll (\log\log T)^{10}$ in \eqref{Main_Thm_LD_JVD1}
	was obtained by Radziwi\l\l\mbox{} \cite{Ra2011}, and the bound for $\|\bm V\|\ll (\log\log T)^{1/6}$ in \eqref{Main_Thm_LD_JVD2}
	was obtained by Inoue \cite{II2019}.
\end{remark}

It is reasonable to conjecture that the asymptotic in \eqref{Main_Thm_LD_JVD1} holds
for $\|\bm V\|=o(\sqrt{\log\log T})$ as speculated in \cite{Ra2011} for $\zeta(s)$. If we are only concerned with upper and lower bounds, we could extend the range of $\|\bm V\|$ further, which yields applications to moments of $L$-functions. 

We have the following unconditional results which extends the range of $\|\bm V\|$ in  Theorem \ref{Main_Thm_LD_JVD}. 

\begin{theorem}	\label{GJu}
	Let
	$\bm{F} = (F_1, \dots, F_{r}) \in (\Sc^{\dagger})^{r}$
	and $\bm{\theta} = (\theta_{1}, \dots, \theta_{r}) \in [-\tfrac{\pi}{2}, \tfrac{3\pi}{2}]^r$ satisfy assumption $\mathscr{A}$.
	Let $T$ be large.
	There exists some positive constant $a_{2} = a_{2}(\bm{F})$ such that
	if $\bm{\theta} \in [-\tfrac{\pi}{2}, \tfrac{\pi}{2}]^r$, we have
	\begin{align*}
		 & \frac{1}{T}\meas(\S(T, \bm{V}; \bm{F}, \bm{\theta}))                                                                     \\
		 & \ll_{\bm{F}} \l\{\l(\prod_{j = 1}^{r}\frac{1}{1 + V_{j}}\r) + \frac{1}{(\log{\log{T}})^{\a_{\bm{F}} + \frac{1}{2}}} \r\}
		\exp\l( -\frac{V_{1}^2 + \cdots + V_{r}^2}{2}
		+ O_{\bm{F}}\l( \frac{\norm[]{\bm{V}}^{3}}{\sqrt{\log{\log{T}}}} \r) \r)
	\end{align*}
	for any $\bm{V} = (V_{1}, \dots, V_{r}) \in (\RR_{\geq 0})^{r}$
	satisfying $\norm[]{\bm{V}} \leq a_{2}(1 + V_{m}^{1/2})(\log{\log{T}})^{1/4}$ with $V_{m} = \min_{1 \leq j \leq r}V_{j}$,
	and if $\bm{\theta} \in \l[ \frac{\pi}{2}, \frac{3\pi}{2} \r]^{r}$, we have
	\begin{align*}
		\frac{1}{T}\meas(\S(T, \bm{V}; \bm{F}, \bm{\theta}))
		\gg_{\bm{F}} \l(\prod_{j = 1}^{r}\frac{1}{1 + V_{j}}\r)
		\exp\l( -\frac{V_{1}^2 + \cdots + V_{r}^2}{2}
		+ O_{\bm{F}}\l( \frac{\norm[]{\bm{V}}^{3}}{\sqrt{\log{\log{T}}}} \r) \r)
	\end{align*}
	for $\norm[]{\bm{V}} \leq a_{2}(1 + V_{m}^{1/2})(\log{\log{T}})^{1/4}$ with
	$\prod_{j = 1}^{r}(1 + V_{j}) \leq a_{2}(\log{\log{T}})^{\a_{\bm{F}} + \frac{1}{2}}$.
\end{theorem}

Substituting
$
	\bm{V} = \l(\frac{V}{\sqrt{\frac{n_{F_{1}}}{2}\log{\log{T}}}}, \dots, \frac{V}{\sqrt{\frac{n_{F_{r}}}{2}\log{\log{T}}}}\r)
$ to Theorem \ref{GJu}, we obtain the following corollary.

\begin{corollary} \label{NNGJu}
	Let
	$\bm{F} = (F_1, \dots, F_{r}) \in (\Sc^{\dagger})^{r}$
	and $\bm{\theta} = (\theta_{1}, \dots, \theta_{r}) \in [-\tfrac{\pi}{2}, \tfrac{3\pi}{2}]^r$ satisfy assumption $\mathscr{A}$.
	Set $h_{\bm{F}} = n_{F_{1}}^{-1} + \cdots + n_{F_{r}}^{-1}$.
	There exists a small constant $a_{3} = a_{3}(\bm{F})$ such that
	if $\bm{\theta} \in \l[ -\frac{\pi}{2}, \frac{\pi}{2} \r]^{r}$, we have
	\begin{align} \label{NNGJu1}
		 & \frac{1}{T}\meas\set{t \in [T, 2T]}{\min_{1 \leq j \leq r}\Re e^{-i\theta_{j}} \log{F_{j}(\tfrac{1}{2}+it)} > V} \\
		 & \ll_{\bm{F}}
		\l(\frac{1}{(1 + V / \sqrt{\log{\log{T}}})^{r}} + \frac{1}{(\log{\log{T}})^{\a_{\bm{F}} + \frac{1}{2}}} \r)
		\exp\l( -h_{\bm{F}}\frac{V^2}{\log{\log{T}}}
		\l( 1 + O_{\bm{F}}\l( \frac{V}{\log{\log{T}}} \r) \r) \r)
	\end{align}
	for any $0 \leq V \leq a_{3}\log{\log{T}}$, and if $\bm{\theta} \in \l[ \frac{\pi}{2}, \frac{3\pi}{2} \r]^{r}$, we have
	\begin{align} \label{NNGJu2}
		 & \frac{1}{T}\meas\set{t \in [T, 2T]}{\min_{1 \leq j \leq r}\Re e^{-i\theta_{j}} \log{F_{j}(\tfrac{1}{2}+it)} > V} \\
		 & \gg_{\bm{F}} \frac{1}{(1 + V / \sqrt{\log{\log{T}}})^{r}}
		\exp\l( -h_{\bm{F}}\frac{V^2}{\log{\log{T}}}
		\l( 1 + O_{\bm{F}}\l( \frac{V}{\log{\log{T}}} \r) \r) \r)
	\end{align}
	for any $0 \leq V \leq a_{3}\min\{\log{\log{T}}, (\log \log T)^{\frac{\a_{\bm{F}}}{r} + \frac{1}{2} + \frac{1}{2r}}\}$.
\end{corollary}

\begin{remark}
	When $r=1$, $F_1=\zeta$, and $\theta_{1} = 0$, Jutila \cite{Ju1983}, using bounds on moments of $\zeta(s)$, has proved
	\begin{align}
		 & \meas\set{t \in [T, 2T]}{ \log |\zeta(\tfrac{1}{2} + it) |> V}\ll T
		\exp\l( -\frac{V^2}{\log{\log{T}}}\l(1 + O\l(\frac{V}{\log{\log{T}}}\r) \r) \r)
	\end{align}
	uniformly for $0\leq V\leq \log \log T$.  More recently, Heap and Soundararajan\cite[Corollary 2]{HeapSound} obtained 
	\begin{align}
		 & \meas\set{t \in [T, 2T]}{ \log |\zeta(\tfrac{1}{2} + it) |> V}\ll T
		\exp\l( -\frac{V^2}{\log{\log{T}}}+ O\l(\frac{V\log_3T}{\sqrt{\log \log T}}\r) \r)
	\end{align}
	for $\sqrt{\log \log T}\log_3T\leq V\leq (2-o(1))\log\log T$
	using sharp bounds on moments of $\zeta(s)$. 
	Our Theorem \ref{GJu} slightly improves Jutila's bound by a factor of $\frac{\sqrt{\log \log T}}{V}$ when $\sqrt{\log\log T}\leq V\leq a_3\log\log T$ for some small constant $a_3$, but is weaker than \cite{HeapSound} when 	$\sqrt{\log \log T}\log_3T\leq V\leq (2-o(1))\log\log T$.
	\end{remark}

%
%
%


If we assume the Riemann Hypothesis for the corresponding $L$-functions, we can extend the range of $\|\bm V\|$ in Theorem \ref{GJu} even further. 
\begin{theorem}\label{RH_LD_JVD}
	Let $\bm{F} = (F_{1}, \dots, F_{r}) \in (\Sc^{\dagger} )^{r}$ and $\bm{\theta} = (\theta_{1}, \dots, \theta_{r}) \in \RR^{r}$
	satisfy assumption $\mathscr{A}$ and assume that the Riemann Hypothesis is true for $F_{1}, \dots, F_{r}$. Let $\bm{V} = (V_{1}, \dots, V_{r}) \in (\RR_{\geq 3})^{r}$ and denote $V_{m} = \min_{1 \leq j \leq r}V_{j}$.  Then there exists $a_4=a_4(\bm F)$ such that the following holds.
		If $\bm{\theta} \in [-\tfrac{\pi}{2}, \tfrac{\pi}{2}]^{r}$ and $\|\bm{V}\| \leq a_{4}V_{m}^{1/2}(\log{\log{T}})^{1/4}(\log_{3}{T})^{1/2}$, we have
	\begin{align}
		\label{RH_LD_JVD1u}
		 & \frac{1}{T}\meas(\S(T, \bm{V}; \bm{F}, \bm{\theta}))                                                 \\
		 & \ll_{\bm{F}}\l(\frac{1}{V_{1} \cdots V_{r}} + \frac{1}{(\log{\log{T}})^{\a_{\bm{F}}+\frac{1}{2}}}\r)
		\exp\l( -\frac{V_{1}^2 + \cdots + V_{r}^2}{2}
		+ O_{\bm{F}}\l( \frac{\norm[]{\bm{V}}^{3}}{\sqrt{\log{\log{T}}} \log{\norm[]{\bm{V}}}} \r) \r).
	\end{align}
	If $\bm{\theta} \in \l[ \frac{\pi}{2}, \frac{3\pi}{2} \r]^{r}$, $\|\bm{V}\| \leq a_{4}V_{m}^{1/2}(\log{\log{T}})^{1/4}(\log_{3}{T})^{1/2}$
 and
	$\prod_{j = 1}^{r}V_{j} \leq a_{4}(\log{\log{T}})^{\a_{\bm{F}} + \frac{1}{2}}$,  we have
	\begin{align}
		\label{RH_LD_JVD1l}
		 & \frac{1}{T}\meas(\S(T, \bm{V}; \bm{F}, \bm{\theta})) \\
		 & \gg_{\bm{F}} \frac{1}{V_{1} \cdots V_{r}}
		\exp\l( -\frac{V_{1}^2 + \cdots + V_{r}^2}{2}
		- O_{\bm{F}}\l( \frac{\norm[]{\bm{V}}^{3}}{\sqrt{\log{\log{T}}} \log{\norm[]{\bm{V}}}} \r) \r).
	\end{align}
	Moreover, there exist some positive constants $a_{5} = a_{5}(\bm{F})$
	such that for any $\bm{V} \in (\RR_{\geq 3})^{r}$ with $\norm[]{\bm{V}} \geq \sqrt{\log\log{{T}}}$ and $\bm{\theta}  \in \l[ -\frac{\pi}{2}, \frac{\pi}{2} \r]^r$,
	\begin{align}
		\label{RH_LD_JVD2}
		\frac{1}{T}\meas(\S(T, \bm{V}; \bm{F}, \bm{\theta}))
		\ll_{\bm{F}}
		\exp\l( -a_{5}\norm[]{\bm{V}}^2 \r) +  \exp\l( -a_{5} \norm[]{\bm{V}}\sqrt{\log{\log{T}}} \log{\norm[]{\bm{V}}} \r).
	\end{align}
\end{theorem}

With $r=1$, Theorem \ref{RH_LD_JVD} slightly improves the bound in \cite[Proposition 4.1]{MT} in the range of the following corollary.

\begin{corollary}	\label{CBSP}
	Let $F \in \Sc^{\dagger} $, and assume the Riemann Hypothesis for $F$.
	Let $A \geq 1$, $\theta \in \l[ -\frac{\pi}{2}, \frac{\pi}{2} \r]$.
	Then, for  any real number $V$ with
	$\sqrt{\log{\log{T}}} \leq V \leq A (\log{\log{T}})^{{2/3}} (\log_{3}{T})^{1/3}$, we have
	\begin{align*}
		\frac{1}{T}\meas\set{t \in [T, 2T]}{\Re e^{-i\theta} \log{F(\tfrac{1}{2}+it)|} > V}
		\ll_{A, F} \frac{\sqrt{\log{\log{T}}}}{V}\exp\l( -\frac{V^2}{n_{F}\log{\log{T}}} \r).
	\end{align*}
	as $T\rightarrow \infty$.
\end{corollary}

\subsection{Moments of $L$-functions}
We apply the large deviation results to obtain bounds for moments of $L$-functions without and with the Riemann Hypothesis of the corresponding $L$-functions. 
The following theorem is a consequence of Theorem \ref{GJu} without the Riemann Hypothesis. 
\begin{theorem}	\label{New_MVT}
	Let	$\bm{F} = (F_{1}, \dots, F_{r}) \in (\Sc^{\dagger})^{r}$ and $\bm{\theta}=(\theta_1, \dots, \theta_r) \in \RR^{r}$ satisfy assumption $\mathscr{A}$.
	Set $h_{\bm{F}} = n_{F_{1}}^{-1} + \cdots + n_{F_{r}}^{-1}$. Then there exist some positive constants $a_{6} = a_{6}(\bm{F})$
	and $B=B(\bm{F})$ such that
	for any fixed $0 < k \leq a_{6}$, 
	\begin{align} \label{New_MVT1}
		\int_{T}^{2T}\exp\l(2k\min_{1 \leq j \leq r}\Re e^{-\theta_{j}} \log{F_{j}(\tfrac{1}{2}+it)}\r)dt
		\ll_{\bm F} T (\log{T})^{k^2 / h_{\bm{F}} + Bk^{3}}
	\end{align}
	when $\bm{\theta}  \in \l[ -\frac{\pi}{2}, \frac{\pi}{2} \r]^r$,
	and if $\vartheta_{\bm{F}} \leq \frac{1}{r + 1}$, we have
	\begin{align} \label{New_MVT2}
		\int_{T}^{2T}\exp\l(2k\min_{1 \leq j \leq r}\Re e^{-i\theta_{j}} \log{F_{j}(\tfrac{1}{2}+it)}\r)dt
		\gg _{\bm F} T (\log{T})^{k^2 / h_{\bm{F}} - Bk^{3}}
	\end{align}
	when $\bm{\theta}  \in \l[ \frac{\pi}{2}, \frac{3\pi}{2} \r]^r$.
	Here, the above implicit constants depend only on $\bm{F}$.
	In particular, if $\vartheta_{\bm{F}} \leq \frac{1}{r + 1}$, it holds that, for any $0 < k \leq a_{6}$,
	\begin{align*}
		\int_{T}^{2T}\l(\min_{1 \leq j \leq r}|F_{j}(\tfrac{1}{2}+it)|\r)^{2k}dt
		\ll_{\bm F}  T (\log{T})^{k^2 / h_{\bm{F}} + Bk^{3}}
	\end{align*}
	\begin{align*}
		\int_{T}^{2T}\l(\max_{1 \leq j \leq r}|F_{j}(\tfrac{1}{2}+it)|\r)^{-2k}dt
		\gg_{\bm F}  T (\log{T})^{k^2 / h_{\bm{F}} - Bk^{3}}
	\end{align*}
	\begin{align*}
		T (\log{T})^{k^2 / h_{\bm{F}} - Bk^{3}}
		\ll_{\bm F}  \int_{T}^{2T}\exp\l(2k\min_{1 \leq j \leq r}\Im \log F_{j}(\tfrac{1}{2}+it)\r)dt
		\ll_{\bm F} T (\log{T})^{k^2 / h_{\bm{F}} + Bk^{3}}
	\end{align*}
	and
	\begin{align*}
		T (\log{T})^{k^2 / h_{\bm{F}} - Bk^{3}}
		\ll_{ \bm{F}} \int_{T}^{2T}\exp\l(-2k\max_{1 \leq j \leq r}\Im \log F_{j}(\tfrac{1}{2}+it)\r)dt
		\ll_{ \bm{F}} T (\log{T})^{k^2 / h_{\bm{F}} + Bk^{3}}.
	\end{align*}
\end{theorem}

\begin{remark}
	The constant $a_6$ could be computed explicitly by keeping track of the constants in the proof. It depends on the size of $a_{F_i}(p)$, 
  the constants ($n_{F{j}}, O_{F_i}(1), O_{F_i, F_j}(1)$) in condition (A1), the degree $d_{F_{j}} = 2\sum_{\ell = 1}^{k}\lam_{\ell}$,
  the constant $\kappa_{F_j}$, and the implicit constant in \eqref{ZDC1}. We also need $a_6 \ll_{\bm F}\frac{1}{r}$.
\end{remark}
\begin{remark}
We could also allow $k\rightarrow 0$ as $T\rightarrow \infty$. In fact, as shown in the proof, the dependency of the implicit constant in $k$ is of size $\frac{1 + k\sqrt{\log\log T}}{1+(k\sqrt{\log\log T})^r}$. 
\end{remark}
The restriction on $k$ in Theorem \ref{New_MVT} is due to the fact that the range of $\| \bm V\|$ in Theorem \ref{GJu} is only a small multiple of $\log\log T$. If we assume the Riemann Hypothesis for the corresponding $L$-functions, we can apply Theorem \ref{RH_LD_JVD} to establish moment results for all $k>0$.

\begin{theorem}	\label{New_MVT_RH}
	Let
	$\bm{F} = (F_{1}, \dots, F_{r}) \in (\Sc^{\dagger} )^{r}$
	and $\bm{\theta} = (\theta_{1}, \dots, \theta) \in [-\tfrac{\pi}{2}, \tfrac{3\pi}{2}]^r$ satisfy assumption $\mathscr{A}$,
	and assume that the Riemann Hypothesis is true for $F_{1}, \dots, F_{r}$.
	Let $T$ be large, and put $\e(T) = (\log_{3}{T})^{-1}$.
	Then there exists some positive constant $B = B(\bm{F})$ such that for any $k > 0$,
	if $\bm{\theta} \in \l[ -\frac{\pi}{2}, \frac{\pi}{2} \r]^r$, we have
	\begin{align} \label{New_MVT_RH1}
		\int_{T}^{2T}\exp\l(2k\min_{1 \leq j \leq r}\Re e^{-\theta_{j}} \log{F_{j}(\tfrac{1}{2}+it)}\r)dt
		\ll_{k, \bm{F}} T(\log{T})^{k^{2}/h_{\bm{F}} + Bk^{3}\e(T)}
	\end{align}
	and if $\bm{\theta} \in \l[ \frac{\pi}{2}, \frac{3\pi}{2} \r]^r$ and $\vartheta_{\bm{F}} < \frac{1}{r + 1}$, we have
	\begin{align} \label{New_MVT_RH2}
		\int_{T}^{2T}\exp\l(2k\min_{1 \leq j \leq r}\Re e^{-\theta_{j}} \log{F_{j}(\tfrac{1}{2}+it)}\r)dt
		\gg _{k, \bm{F}} T + T(\log{T})^{k^{2}/h_{\bm{F}} - Bk^{3}\e(T)}
	\end{align}
	In particular, if $\vartheta_{\bm{F}} < \frac{1}{r + 1}$, it holds that, for any $k > 0$, $\e > 0$,
	\begin{align*}
		\int_{T}^{2T}\l(\min_{1 \leq j \leq r}|F_{j}(\tfrac{1}{2}+it)|\r)^{2k}dt
		\ll_{\e, k, \bm{F}} T(\log{T})^{k^{2}/h_{\bm{F}} + \e},
	\end{align*}
	\begin{align*}
		\int_{T}^{2T}\l(\max_{1 \leq j \leq r}|F_{j}(\tfrac{1}{2}+it)|\r)^{-2k}dt
		\gg_{\e, k, \bm{F}} T(\log{T})^{k^{2}/h_{\bm{F}} - \e},
	\end{align*}
	\begin{align*}
		T(\log{T})^{k^{2}/h_{\bm{F}} - \e}
		\ll_{\e, k, \bm{F}}\int_{T}^{2T} \exp\l(2k\min_{1 \leq j \leq r}\Im \log F_{j}(\tfrac{1}{2}+it)\r)dt
		\ll_{\e, k, \bm{F}} T(\log{T})^{k^{2}/h_{\bm{F}} + \e},
	\end{align*}
	and
	\begin{align*}
		T(\log{T})^{k^{2}/h_{\bm{F}} - \e}
		\ll_{\e, k, \bm{F}}\int_{T}^{2T} \exp\l(-2k\max_{1 \leq j \leq r}\Im \log F_{j}(\tfrac{1}{2}+it)\r)dt
		\ll_{\e, k, \bm{F}} T(\log{T})^{k^{2}/h_{\bm{F}} + \e}.
	\end{align*}
\end{theorem}

\begin{remark}
	One would expect that for any $k>0$, it holds that
	\begin{align}
		\int_T^{2T}\exp\l(2k\min_{1 \leq j \leq r}\Re e^{-i\theta_{j}} \log F_{j}(\tfrac{1}{2}+it)\r)dt
		\asymp_{k} T \frac{(\log T)^{k^2 / h_{\bm{F}}}}{(\log{\log{T}})^{(r-1)/2}}.
	\end{align}
	It remains an interesting question to see if the bounds in Theorem \ref{New_MVT_RH} can be made sharp using techniques from Harper \cite{H2013} under the corresponding Riemann Hypothesis.
\end{remark}

\begin{remark}
	An $L$-function is called primitive if it cannot be factored into $L$-functions of smaller degree. It is conjectured that
	\begin{align}
		\int_0^T | F(\tfrac{1}{2}+it)|^{2k} dt \sim C(F, k)T(\log T)^{k^2}
	\end{align}
	for some constant $C(F, k)$ as $T\rightarrow \infty$, see \cite{Conrey05,Heap13}. 
  It is expected that the values of distinct primitive $L$-functions are uncorrelated, which leads to the conjecture
	\begin{align}
		\int_0^T |F_1(\tfrac{1}{2}+it)|^{2k_1} \cdots  |F_r(\tfrac{1}{2}+it)|^{2k_r} dt
		\sim C(\bm F, \bm k) T(\log T)^{k_1^2 + \dots + k_r^2}
	\end{align}
	for some constant $C(\bm F, \bm k)$ as $T\rightarrow \infty$ if $F_i\not=F_j$ for $i\not=j$.
	This has be established for the product of two Dirichlet $L$-functions for $k_1=k_2=1$ (see \cite{Heap13, M1970, Top} and
	for some degree two $L$-functions when $k=1$ and $r=1$ (see \cite{Good, Zhang1, Zhang2}).
	For higher degree $L$-functions and higher values of $k$, obtaining the asymptotic formula seems to be beyond the scope of
	current techniques.
	An upper bound of this kind has been established by Milinovich and Turnage-Butterbaugh \cite{MT} for automorphic $L$-functions of $GL(n)$
	under the Riemann Hypothesis for these $L$-functions. Modifying our approach, we can recover their results (see Section \ref{finalremarks}) and obtain some weaker unconditional results for sufficiently small $k_i$'s when $n\leq 2$.
	Our results give some further evidence that distinct primitive $L$-functions are `` statistically independent".
\end{remark}
\begin{remark}
Our method as well as the works of Bombieri-Hejhal \cite{BH1995} and Selberg \cite{SCR} requires a strong zero density estimate for $L$-functions.
Unfortunately, the estimate has not been proved yet for many $L$-functions.
There have been other methods 
to prove Selberg's central limit theorem without the strong zero density estimate such as Laurin\v{c}ikas \cite{L1987} 
and Radziwi\l\l-Soundararajan \cite{RS2017}.
However, their methods  do not apply to the central limit theorems for $\Im \log \zeta(s)$. 
Nevertheless, Hsu-Wong \cite{HW2020}  proved a joint central limit theorem (for fixed $V_{j}$) 
for Dirichlet $L$-functions and certain $GL(2)$ $L$-functions twisted by Dirichlet characters by using the method in \cite{RS2017}. 
The fact that the coefficients of these $L$-functions satisfy $|\chi(n)| \leq 1$ plays an important role in their proofs, which make it not applicable to more general $L$-functions.
\end{remark}

\subsection{Results for Dirichlet polynomials}
To prove the above theorems, we consider the Dirichlet polynomials associated with $F$.
We need some notation before stating our results. 
Let $\Lam_{F}(n)$ be the von Mangoldt function associated with $F$ defined by $\Lam_{F}(n) = b_{F}(n) \log{n}$.
Let $\bm{x} = (x_{1}, \dots, x_{r}) \in \RR^{r}$, $\bm{z} = (z_{1}, \dots, z_{r}) \in \CC^{r}$,
and $\bm{F} = (F_{1}, \dots, F_{r})$ be an $r$-tuple of Dirichlet series, 
and let $\bm{\theta} = (\theta_{1}, \dots, \theta_{r}) \in \RR^r$.
When $\bm{F}$ satisfies (S4), we define
\begin{align}	\label{def_D_P}
	P_{F}(s, X)
	:= \sum_{p^{\ell} \leq X}\frac{b_{F}(p^{\ell})}{p^{\ell s}}
	= \sum_{2 \leq n \leq X}\frac{\Lam_{F}(n)}{n^{s}\log{n}},
\end{align}
\begin{align}	\label{def_var}
	\s_{F}(X)
	:= \sqrt{\frac{1}{2}\sum_{p \leq X} \sum_{\ell = 1}^{\infty}\frac{|b_{F}(p^{\ell})|^2}{p^{\ell}}},
\end{align}
\begin{align} \label{def_tau_F}
	\tau_{j_{1}, j_{2}}(X) = \tau_{j_{1}, j_{2}}(X; \bm{F}, \bm{\theta})
	:= \frac{1}{2}\sum_{p \leq X}\sum_{\ell = 1}^{\infty}
	\frac{\Re e^{-i\theta_{j_{1}}}b_{F_{j_{1}}}(p^{\ell})\ol{e^{-i\theta_{j_{2}}} b_{F_{j_{2}}}(p^{\ell})}}{p^{\ell}},
\end{align}
\begin{align}	\label{def_K_F}
	K_{\bm{F}, \bm{\theta}}(p, \bm{z})
	:= \sum_{1 \leq j_{1}, j_{2} \leq r}z_{j_{1}} z_{j_{2}}
	\sum_{\ell = 1}^{\infty}\frac{\Re e^{-i\theta_{j_{1}}}b_{F_{j_{1}}}(p^{\ell})\ol{e^{-i\theta_{j_{2}}} b_{F_{j_{2}}}(p^{\ell})}}{p^{\ell}},
\end{align}
and
\begin{align}\label{S_X def}
	\S_{X}(T, \bm{V}; \bm{F}, \bm{\theta})
	:= \set{t \in [T, 2T]}
	{\frac{\Re e^{-i\theta_{j}} P_{F_{j}}(\tfrac{1}{2}+it, X)}  {\s_{F_{j}}(X)}
		> V_{j} \text{ for } j = 1, \dots, r}.
\end{align}
Let $\{ \mathcal{X}(p) \}_{p \in \mathcal{P}}$ be a sequence of independent random variables
on a probability space $(\Omega, \mathscr{A}, \PP)$
with uniformly distributed on the unit circle in $\CC$, where $\mathcal{P}$ is the set of prime numbers.
Denote
\begin{align} \label{def_RP}
	P_{F}(\s, \mathcal{X}, X)
	:= \sum_{p \leq X}\sum_{\ell = 1}^{\infty}\frac{b_{F}(p^{\ell}) \mathcal{X}(p)^{\ell}}{p^{\ell \s}},
\end{align}
\begin{align} \label{def_RDP_MGF}
	M_{p, \s}(\bm{z}) = M_{p, \s}(\bm{z}; \bm{F}, \bm{\theta})
	:= \EXP{\exp\l(\sum_{j = 1}^{r}z_{j} \Re e^{-i\theta_{j}}
		\sum_{\ell = 1}^{\infty}\frac{b_{F_{j}}(p^{\ell}) \mathcal{X}(p)^{\ell}}{p^{\ell \s}}\r)},
\end{align}
where $\EXP{}$ is the expectation. 
Finally, when $\bm{F}$ satisfies (S4) and (S5), we define
\begin{align} \label{def_Xi}
	\Xi_{X}(\bm{x})
	= \Xi_{X}(\bm{x}; \bm{F}, \bm{\theta})
	:= \exp\l( \sum_{1 \leq j_{1} < j_{2} \leq r}x_{j_{1}} x_{j_{2}}\tau_{j_{1}, j_{2}}(X) \r)
	\prod_{p}\frac{M_{p, \frac{1}{2}}(\bm{x})}{\exp\l( K_{\bm{F}, \bm{\theta}}(p, \bm{x})/4 \r)}.
\end{align}
The convergence of the infinite product of \eqref{def_Xi} is shown in Lemma \ref{Prop_Psi_FPP}.

We have the following the joint large deviations results for Dirichlet polynomials.

\begin{proposition}	\label{Main_Prop_JVD}
	Let $\bm{F} = (F_{1}, \dots, F_{r})$ be an $r$-tuple of Dirichlet series and $\bm{\theta} \in \RR^{r}$
	satisfy (S4), (S5), (A1), and (A2).
	Let $T$, $X$ be large numbers with $X^{(\log{\log{X}})^{4(r + 1)}} \leq T$.
	Then there exists some positive constant $a_{7} = a_{7}(\bm{F})$ such that
	for $\bm{V} = (V_{1}, \dots, V_{r}) \in \RR^{r}$ with $|V_{j}| \leq a_{7} \s_{F_{j}}(X)$,
	\begin{align}	\label{MPJVD1}
		 & \frac{1}{T}\meas(\S_{X}(T, \bm{V}; \bm{F}, \bm{\theta}))                                                   \\
		 & = \l( 1 + O_{\bm{F}}\l( \frac{\prod_{k = 1}^{r}(1 + |V_{k}|)}{(\log{\log{X}})^{\a_{\bm{F}} + \frac{1}{2}}}
			+ \frac{1 + \norm[]{\bm{V}}^{2}}{\log{\log{X}}}\r) \r)\prod_{j = 1}^{r}
		\int_{V_{j}}^{\infty}e^{-u^2/2}\frac{du}{\sqrt{2\pi}}.
	\end{align}
\end{proposition}

We could improve the range $V_{j}$ in Propositions \ref{Main_Prop_JVD} with a weaker error term.

\begin{proposition}	\label{Main_Prop_JVD3}
	Let $\bm{F} = (F_{1}, \dots, F_{r})$ be an $r$-tuple of Dirichlet series and $\bm{\theta} \in \RR^{r}$
	satisfy (S4), (S5), (A1), and (A2).
	Let $T$, $X$ be large numbers satisfying $X^{(\log{\log{X}})^{4(r + 1)}} \leq T$.
	Then for any $\bm{V} = (V_{1}, \dots, V_{r}) \in (\RR_{\geq 0})^{r}$ with $\| \bm{V} \| \leq (\log{\log{X}})^{2r}$, we have
	\begin{align} \label{MPJVD2}
		 & \frac{1}{T}\meas(\S_{X}(T, \bm{V}; \bm{F}, \bm{\theta}))
		= (1 + E) \times \Xi_{X}\l(\tfrac{V_{1}}{\s_{F_{1}}(X)}, \dots, \tfrac{V_{r}}{\s_{F_{r}}(X)}\r)
		\prod_{j = 1}^{r}\int_{V_{j}}^{\infty}e^{-u^2/2}\frac{du}{\sqrt{2\pi}},
	\end{align}
	where $E$ satisfies
	\begin{align} \label{ERMPJVD}
		E \ll_{\bm{F}}
		\exp\l( C \l(\frac{\| \bm{V} \|}{\sqrt{\log{\log{X}}}}\r)^{\frac{2 - 2\vartheta_{\bm{F}}}{1 - 2\vartheta_{\bm{F}}}} \r)
		\l\{\frac{\prod_{k = 1}^{r}(1 + V_{k})}{(\log{\log{X}})^{\a_{\bm{F}} + \frac{1}{2}}}
		+ \frac{1}{\sqrt{\log{\log{X}}}}\r\}.
	\end{align}
\end{proposition}

\begin{remark}
	In contrast to Proposition \ref{Main_Prop_JVD}, we allow $V_j$ to be of size $C\sqrt{\log\log X}$ for arbitrarily large $C$,
	which is important in the proof of Theorem \ref{New_MVT_RH}.
	We can prove an estimate similar to \eqref{MPJVD2} for larger $V_{j}$,
	where we need to change the value of $X$ suitably in this case.
	However, our main purpose is to prove Theorems \ref{Main_Thm_LD_JVD}, \ref{GJu}, and \ref{RH_LD_JVD},
	and the case of larger $V_{j}$ is not required in their proofs.
	For this reason, we give only the case $\norm[]{\bm{V}} \leq (\log{\log{X}})^{2r}$ for simplicity.
\end{remark}

Applying Proposition \ref{Main_Prop_JVD3}, we can prove the following theorem.

\begin{theorem} \label{GMDP}
	Let $\bm{F} = (F_{1}, \dots, F_{r})$ be an $r$-tuple of Dirichlet series, and $\bm{\theta} \in \RR^{r}$
	satisfy (S4), (S5), (A1), and (A2). Denote $\vartheta_{\bm{F}}^{*}=\min_{1\leq j\leq r}\vartheta_{F_j}$.
	Then for any fixed $k > 0$
there exist positive constants $c_1=c_1(k, \bm F)$, $c_2=c_2(k, \bm F)$ and $X_{0} = X_{0}(\bm{F}, k)$ such that
	for any large $T$ and
	$X_{0} \leq X \leq \max\{ c_{1}(\log{T} \log{\log{T}})^{\frac{2}{1 + 2\vartheta_{\bm{F}}^{*}}}, 
  c_{2}\frac{(\log{T})^2 \log{\log{T}}}{\log_{3}{T}} \}$,
	we have
	\begin{align*}
		 & \frac{1}{T}\int_{T}^{2T}\exp\l(2k\min_{1 \leq j \leq r}\Re e^{-i\theta_{j}} P_{F_{j}}(\tfrac{1}{2} + it, X)\r)dt \\
		 & = \frac{\exp\l( k^2 H_{\bm{F}}(X) \r) \prod_{j = 1}^{r}\s_{F_{j}}(X)}{\l( \sqrt{2\pi} k H_{\bm{F}}(X) \r)^{r-1}
		\sqrt{\frac{1}{2}H_{\bm{F}}(X)}}
		\Xi_{X}\l( \frac{k H_{\bm{F}}(X)}{\s_{F_{1}}(X)^2}, \dots, \frac{k H_{\bm{F}}(X)}{\s_{F_{r}}(X)^2} \r)\l( 1 + E \r),
	\end{align*}
	where $H_{\bm{F}}(X) = 2\l(\sum_{j = 1}^{r}\s_{F_{j}}(X)^{-2}\r)^{-1}$, and
	\begin{align*}
		E
		\ll_{\bm{F}, k}
		\frac{1}{(\log{\log{X}})^{\a_{\bm{F}} - \frac{r - 1}{2}}}
		+ \frac{\log_{3}{X}}{\sqrt{\log{\log{X}}}}.
	\end{align*}
\end{theorem}

\begin{remark}
When $r=1$, these type of results can be obtained via mean value theorems of Dirichlet series (see Gonek-Hughes-Keating \cite[Theorem 2]{GHK2007} for $\zeta(s)$ and Heap \cite[Theorem 2]{Heap13} for Dedekind zeta-functions associated with Galois extension number fields over $\QQ$). Our method has the advantage to handle more than one $L$-functions, though the error term is weaker. 
\end{remark}
\begin{remark}
The exact dependency of $c_1, c_2$ on $\bm F$ and $k$ are determined explicitly in the proof. 
When $\vartheta_{\bm{F}}^{*} = 0$, the upper bound becomes $X \leq c_1(\log T\log\log T)^2$ where $c_1 \ll_{\bm F} \frac{1}{k^2}$.
\end{remark}

The asymptotic behavior can be further evaluated under some additional assumptions.
\begin{corollary}\label{jointDirichlet}
We use the same notation and assumptions as in Theorem \ref{GMDP}. We further assume $\vartheta_{\bm{F}} < \frac{1}{r + 1}$ and the strong Selberg Orthonormality Conjecture, that is, 
\begin{align} \label{SSOC}
	\sum_{n \leq X}\frac{a_{F_{i}}(p) \ol{a_{F_{j}}(p)}}{p}
	= \delta_{F_{i}, F_{j}} n_{F_{j}}\log{\log{X}} + e_{i, j} + o(1), \quad X \rightarrow + \infty
\end{align}
for some constant $e_{i, j}$. 
Here, $\delta_{F_{i}, F_{j}}$ is the Kronecker delta function, that is, $\delta_{F_{i}, F_{j}} = 1$ if $F_{i} = F_{j}$, 
and $\delta_{F_{i}, F_{j}} = 0$ otherwise. 
Then, as $X \rightarrow + \infty$,
\begin{align} \label{GGHK}
  \hspace{-3mm}
  \frac{1}{T}\int_{T}^{2T}
  \exp\l( 2k \min_{1 \leq j \leq r}\Re e^{-i\theta_{j}} \sum_{2 \leq n \leq X}\frac{\Lam_{F_{j}}(n)}{n^{1/2+it} \log{n}} \r) dt
	\sim C(\bm{F}, k, \bm{\theta})\frac{(\log{X})^{k^2 / h_{\bm{F}}}}{(\log{\log{X}})^{(r - 1)/2}}
\end{align}
for some positive constant $C(\bm{F}, k, \bm{\theta})$.
\end{corollary}

\begin{remark}
 When $r = 1$, $F_{1} = \zeta$, and $\theta_{1} = 0 \text{ or } \pi $ Corollary \ref{jointDirichlet} recovers a result of Heap \cite[Proposition 1]{He2021} for real moments of partial Euler product with $X<(\frac{1}{16k^2}-\epsilon )(\log T\log\log T)^2$.
\end{remark}

\section{\textbf{Approximate formulas for $L$-functions}}\label{appL}

\subsection{Approximate formulas for $L$-functions}

In this section, we give an approximate formula for $\log{F(s)}$.
Here, we choose the branch of $\log{F(\s + it)}$ as follows.
If $t$ is equal to neither imaginary parts of zeros nor poles of $F$, then we choose the branch by the integral
$\log{F(\s + it)} = \int_{\infty + it}^{\s + it}\frac{F'}{F}(z)dz$.
If $t \not= 0$ is equal to an imaginary part of a zero or a pole of $F$,
then we take $\log{F(\s + it)} = \lim_{\e \downarrow 0}\log{F(\s + i(t - \sgn(t)\e))}$.
Here, $\sgn$ is the signum function.
If there exists a pole or a zero of $F$ such that the imaginary part of zero, then we take $\log{F(\s)} = \lim_{\e \downarrow 0} \log{F(\s - \e)}$.
Next we introduce some notation.

\begin{notation*}

	Let $H \geq 1$ be a real parameter.
	The function $f: \RR \rightarrow [0, +\infty)$ is mass one and supported on $[0, 1]$, and further $f$ is a
	$C^{1}([0, 1])$-function, or $f$ belongs to $C^{d-2}(\RR)$ and is a $C^{d}([0, 1])$-function for some $d \geq 2$.
	For such $f$'s, we define the number $D(f)$, and functions $u_{f, H}$, $v_{f, H}$ by
	\begin{align}	\label{def_D}
		D(f) = \max\set{ d \in \ZZ_{\geq 1} \cup \{+\infty\}}{\text{$f$ is a $C^{d}([0, 1])$-function}},
	\end{align}
	$u_{f, H}(x) = Hf(H\log(x/e))/x$, and
	\begin{gather}
		\label{def_v}
		v_{f, H}(y) = \int_{y}^{\infty}u_{f, H}(x)dx.
	\end{gather}
	Further, the function $U_{0}(z)$ is defined by
	\begin{gather}
		\label{def_U}
		U_{0}(z) = \int_{0}^{\infty}u_{f, H}(x)E_{1}(z\log{x})dx
	\end{gather}
	for $\Im(z) \not= 0$.
	Here, $E_{1}(z) = E_{1}(x+iy)$ is the exponential integral defined by
	\begin{align*}
		E_{1}(z)
		:= \int_{x+iy}^{+\infty+iy}\frac{e^{-w}}{w} dw
		= \int_{z}^{\infty}\frac{e^{-w}}{w} dw.
	\end{align*}
	When $\Im(z) = 0$, then $U_{0}(x) = \lim_{\e \uparrow 0}U_{0}(x + i\e)$.

	Let $X \geq 3$ be a real parameter.
	Let $\rho_{F} = \b_{F} + i\gamma_{F}$ be a nontrivial zero of $F$ with $\b_F$, $\gamma_F$ real numbers.
	Here, nontrivial zeros refer to the zeros of $\Phi_F(s)$ that do not come from the Gamma factors $\gamma(s)$.
	We also define $\s_{X, t}(F)$ for $F \not= 1$ and $w_{X}(y)$ by
	\begin{gather}
		\label{def_s_X_F}
		\s_{X, t}(F)
		= \frac{1}{2} + 2\max_{|t - \gamma_{F}| \leq \frac{X^{3|\b_{F} - 1/2|}}{\log{X}}}
		\l\{ \b_{F} - \frac{1}{2}, \frac{2}{\log{X}} \r\},\\
		\label{def_w_X}
		w_{X}(y) = \l\{
		\begin{array}{cl}
			1                                                         & \text{if\, $1 \leq y \leq X$,}       \\[2mm]
			\frac{(\log(X^3/y))^2 - 2(\log(X^2 / y))^2}{2(\log{X})^2} & \text{if\, $X \leq y \leq X^2$,}     \\[2mm]
			\frac{(\log(X^{3}/y))^2}{2(\log{X})^2}                    & \text{if\, $X^{2} \leq y \leq X^3$.}
		\end{array}
		\r.
	\end{gather}
\end{notation*}

\begin{remark}
  The number $\s_{X, t}(F)$ is well-defined as a finite value for every $X \geq 3$, $t \in \RR$, $F \in \Sc^{\dagger} \setminus \{1\}$.
  Actually, by the same method as Conrey-Ghosh \cite[Theorem 2]{CG1993}, we see that $d_{F} \geq 1$ or $F = 1$ for any $F \in \Sc^{\dagger}$.
  We also find that there are infinitely many zeros of $F \in \Sc^{\dagger} \setminus \{1\}$ in the region $0 \leq \s \leq 1$
  by the standard argument in the Riemann-von Mangoldt formula, and that $F$ does not have zeros for $\s > 1$ from (S1) and (S5).
  Hence, we have $\frac{1}{2} + \frac{4}{\log{X}} \leq \s_{X, t}(F) \leq \frac{3}{2}$. 
\end{remark}

Then we have the following theorem, which is a generalization of \cite[Theorem 1]{II2019} in the case when $F$ is the Riemann zeta-function $\zeta(s)$.

\begin{theorem}	\label{Main_F_S}
	Assume $F \in \Sc^{\dagger}$. Let $ D(f)\geq 2$, and $H$, $X$ be real parameters with $H \geq 1$, $X \geq 3$.
	Then, for any $\s \geq 1/2$, $t \geq 14$, we have
	\begin{align}
		\label{EQ_F_S}
		 & \log{F(s)} =                                                                                      \\
		 & \sum_{2 \leq n \leq X^{1+1/H}}\frac{\Lam_{F}(n)v_{f, H}\l( e^{\log{n}/\log{X}} \r)}{n^{s}\log{n}}
		+\sum_{|s - \rho_{F}| \leq \frac{1}{\log{X}}}\log((s - \rho_{F})\log{X}) + R_{F}(s, X, H),
	\end{align}
	where the error term $R_{F}(s, X, H)$ satisfies 	\begin{align}
		\label{ESR_S2}
		 & R_{F}(s, X, H)
		\ll_{f} m_{F}\frac{X^{2(1-\s)} + X^{1-\s}}{t\log{X}}                        \\
		 & +H^{3}(\s_{X, t}(F) - 1/2)(X^{2(\s_{X, t}(F) - \s)} + X^{\s_{X, t}(F) - \s})
		\l( \Bigg| \sum_{n \leq X^3}\frac{\Lam_{F}(n)w_{X}(n)}{n^{\s_{X, t}(F) + it}} \Bigg| + d_{F}\log{t} \r)
	\end{align}
	for $|t| \geq t_{0}(F)$ with $t_{0}(F)$ a sufficiently large constant depending on $F$.
\end{theorem}

\begin{remark}
	Note that in the above theorem we choose the branch of $\log{(s - \rho_{F})}$ as follows.
	If $t \not= \gamma_{F}$, then $-\pi < \arg{(s - \rho_{F})} < \pi$,
	and if $t = \gamma_{F}$, then $\arg{(s - \rho_{F})} = \lim_{\e \uparrow 0}\arg{(\s - \b_{F} + i\e)}$.
\end{remark}

\begin{remark}
	Theorem \ref{Main_F_S} is a modification and generalization of the hybrid formula by Gonek, Hughes, and Keating \cite{GHK2007} to apply
	the method of Selberg-Tsang \cite{KTDT}, where a different formula \cite[Lemma 5.4]{KTDT} was used.
	With our formula, we can find the sign of the contribution from zeros to close $s$ by the form
	$
		\sum_{|s - \rho_{F}|\leq \frac{1}{\log{X}}}\log((s - \rho_{F})\log{X})
	$.
	This fact plays an important role in the proof of the theorems in Section \ref{General} in a fashion similar to the work of Soundararajan \cite{SM2009}. 
\end{remark}

Theorem \ref{Main_F_S} can be obtained in the same method as the proof of \cite[Theorem 1]{II2019},
where we need the following proposition instead of  \cite[Proposition 1]{II2019}.

\begin{proposition}	\label{GSEF}
	Let $F \in \Sc^{\dagger}$.
	Let $X \geq 3$, $H \geq 1$ be real parameters.
	Then, for any $s \in \CC$, we have
	\begin{align*}
		\log{F(s)}
		= & \sum_{2 \leq n \leq X^{1 + 1/H}}\frac{\Lam_{F}(n)v_{f, H}(e^{\log{n}/\log{X}})}{n^{s} \log{n}}
		+ m_{F}^{*}(U_{0}((s - 1)\log{X}) + U_{0}(s \log{X}))                                              \\
		  & - \sum_{\substack{\rho_{F}                                                                     \\ \rho_{F} \not= 0, 1}}U_{0}((s - \rho_{F})\log{X})
		- \sum_{n = 0}^{\infty}\sum_{j = 1}^{k}U_{0}((s + (n+\mu_{j})/\lam_{j})\log{X}),
	\end{align*}
	where the number $m_{F}^{*}$ is the integer such that the function $(s - 1)^{m_{F}^{*}}F(s)$ is entire and not equal to zero at $s = 1$.
\end{proposition}

Using Theorem \ref{Main_F_S}, we obtain the following propositions.

\begin{proposition}	\label{KLI}
	Let $F \in \Sc^{\dagger} $ satisfying \eqref{SNC} and (A3).
	Let $\s \geq 1/2$, and $T$ be large.
	Then, there exist positive constants $\delta_{F}$, $A_{1} = A_{1}(F)$ such that
	for any $k \in \ZZ_{\geq 1}$, $3 \leq X \leq Y := T^{1/k}$,
	\begin{align*}
		 & \frac{1}{T}\int_{T}^{2T}\bigg| \log{F(\s + it)} - P_{F}(\s + it, X)
		- \sum_{|\s + it - \rho_{F}| \leq \frac{1}{\log{Y}}}\log((\s + it - \rho_{F})\log{Y}) \bigg|^{2k}dt \\
		 & \leq A_{1}^{k} k^{2k}T^{\delta_{F}(1 - 2\s)}
		+ A_{1}^{k} k^{k} \l( \sum_{X < p \leq Y}\frac{|a_{F}(p)|^2}{p^{2\s}} \r)^{k}.
	\end{align*}
\end{proposition}

\begin{proposition}	\label{KLST}
	Suppose the same situation as Proposition \ref{KLI}.
	Then, there exist positive constants $\delta_{F}$, $A_{2} = A_{2}(F)$ such that
	for any $k \in \ZZ_{\geq 1}$, $3 \leq X \leq Y := T^{1/k}$,
	\begin{align*}
		 & \frac{1}{T}\int_{T}^{2T}\l| \log{F(\s + it)} - P_{F}(\s + it, X) \r|^{2k}dt
		 \leq A_{2}^{k} k^{4k}T^{\delta_{F}(1 - 2\s)}
		+ A_{2}^{k} k^{k} \l( \sum_{X < p \leq Y}\frac{|a_{F}(p)|^{2}}{p^{2\s}} \r)^{k}.
	\end{align*}
\end{proposition}

Proposition \ref{KLST} has been essentially proved by Selberg \cite[Theorem 1]{S1992}.
Our range of $X$ and his are different, but this difference is filled by the argument in \cite[Sections 3, 5]{KTDT}.
On the other hand, we do not assume the Ramanujan Conjecture (S5).
Hence, we give the proof of Proposition \ref{KLST} for completeness.

\subsection{\textbf{Proof of Proposition \ref{GSEF} and Theorem \ref{Main_F_S}}}

We give the proofs of Proposition \ref{GSEF} and Theorem \ref{Main_F_S}, but the proofs are almost the same as
the proofs of \cite[Proposition 2]{II2019} and \cite[Theorem 1]{II2019}. Therefore, we give the sketches only.

\begin{lemma}	\label{GHF}
	Let $F \in \Sc^{\dagger} \setminus \{ 1 \}$. For all $s \in \CC$ neither equaling to a pole nor a zero of $F$, we have
	\begin{align}	\label{GLF}
		\frac{F'}{F}(\s + it)
		= & \sum_{\substack{\rho_{F}                                                                          \\ \rho_{F} \not= 0, 1}}\l( \frac{1}{s - \rho_{F}} + \frac{1}{\rho_{F}} \r)
		+ \eta_F - m_{F}\l(\frac{1}{s - 1} + \frac{1}{s}\r) - \log{Q}                                       \\
		  & - \sum_{n = 0}^{\infty}\sum_{j = 1}^{k}\lam_{j}\frac{\Gamma'}{\Gamma}\l( \lam_{j}s + \mu_{j} \r),
	\end{align}
	where $\eta_F$ is a complex number and satisfies $\Re(\eta_{F}) = -\Re\sum_{\rho_{F}}(1/\rho_{F})$.
	In particular, for $|t| \geq t_{0}(F)$, we have
	\begin{align}	\label{GLFE}
		\frac{F'}{F}(\s + it)
		= \sum_{\substack{\rho_{F} \\ \rho_{F} \not= 0, 1}}\l( \frac{1}{s - \rho_{F}} + \frac{1}{\rho_{F}} \r)
		+ O(d_{F} \log{|t|}).
	\end{align}
\end{lemma}

\begin{proof}
	We obtain equation \eqref{GLF} in the same way as the proof of \cite[eq. (10.29)]{MV}.
	Moreover, by applying the Stirling formula to equation \eqref{GLF}, we can also obtain equation \eqref{GLFE}.
\end{proof}

\begin{lemma}	\label{lem_far_zero}
	Let $F \in \Sc^{\dagger}$.
	For $|t| \geq t_0(F)$, $1 \leq H \leq \frac{|t|}{2}$, we have
	\begin{gather}
		\label{SIZDS}
		\sum_{|t - \gamma_{F}| \leq H}1 \ll d_{F} H \log{|t|},\\
		\label{EEZS}
		\sum_{|t - \gamma_{F}| \geq H}\frac{1}{(t - \gamma_{F})^2}
		\ll \frac{d_{F} \log{|t|}}{H}.
	\end{gather}
\end{lemma}

\begin{proof}
	Applying the Stirling formula to Lemma \ref{GHF},
	for $|t| \geq t_{0}(F)$, we have
	\begin{align*}
		\Re\l(\frac{F'}{F}(H + it)\r)
		= \sum_{\rho}\frac{H - \b_{F}}{(H - \b_{F})^2 + (t - \gamma_{F})^2} + O\l(d_{F} \log{|t|}\r).
	\end{align*}
	On the other hand, it holds that
	\begin{gather*}
		\sum_{\rho_{F}}\frac{H - \b_{F}}{(H - \b_{F})^2 + (t - \gamma_{F})^2}
		\gg \sum_{|t - \gamma_{F}| \leq H}\frac{1}{H},\\
		\sum_{\rho_{F}}\frac{H - \b_{F}}{(H - \b_{F})^2 + (t - \gamma_{F})^2}
		\gg \sum_{|t - \gamma_{F}| \geq H}\frac{H}{(t - \gamma_{F})^2}.
	\end{gather*}
	Since $b_{F}(n)\log{n} \ll_{F} n^{1/2}$ from (S4), we find that
	\begin{align*}
		|(F'/F)(H + it)| = \bigg|\sum_{n = 2}^{\infty}b_{F}(n)\log{n}/n^{H + it}\bigg|
		\ll_{F} \zeta(H - 1/2) - 1 \ll 2^{-H}.
	\end{align*}
	Hence we obtain \eqref{SIZDS} and \eqref{EEZS} for $H \geq 2$.
	In addition, we immediately obtain these inequality the case
	$1 \leq H \leq 2$ from \eqref{SIZDS} and \eqref{EEZS} in the case $H = 2$.
\end{proof}

\begin{lemma}	\label{ZAL}
	Let $F \in \Sc^{\dagger}$. For any $T \geq t_{0}(F)$, there exists some $t \in [T, T + 1]$ such that,
	uniformly for $1/2 \leq \s \leq 2$,
	\begin{align*}
		\frac{F'}{F}(\s + it)
		\ll_{F} (\log{T})^{2}.
	\end{align*}
\end{lemma}

\begin{proof}
	Using Lemma \ref{lem_far_zero}, we obtain this lemma in the same argument as the proof of \cite[Lemma 12.2]{MV}.
\end{proof}

\begin{proof}[Proof of Proposition \ref{GSEF}]
	By using Lemma \ref{ZAL}, we obtain Proposition \ref{GSEF}
	in the same method as the proof of \cite[Proposition 1]{II2019}.
\end{proof}

\begin{lemma}	\label{GUE}
	Let $d$ be a nonnegative integer with $d < D = D(f) + 1$.
	Let $s = \s + it$ be a complex number.
	Set $X \geq 3$ be a real parameter.
	Then, for any $z = a + ib$ with $a \in \RR$, $b \in \RR \setminus \{ t \}$, we have
	\begin{align*}
		U_{0}((s - z)\log{X})
		\ll_{f, d} \frac{X^{(1+1/H)(a - \s)} + X^{a - \s}}{|t - b|\log{X}}
		\min_{0 \leq l \leq d}\l\{\l(\frac{H}{|t - b|\log{X}} \r)^{l}\r\}.
	\end{align*}
\end{lemma}

\begin{proof}
	This lemma can be proved in the same way as \cite[Lemma 2]{II2019}.
\end{proof}

\begin{lemma}	\label{EUE}
	Let $s = \s + it$ be a complex number.
	Set $X \geq 3$ be a real parameter.
	Then, for any complex number $z = a + ib$ with $|t - b| \leq 1 / \log{X}$, we have
	\begin{align}
		U_{0}((s - z)\log{X})
		= \l\{
		\begin{array}{cl}
			-\log((s - z)\log{X}) + O(1)             & \text{if \; $|s - z| \leq 1/\log{X}$, } \\[5mm]
			O\l(X^{(1+1/H)(a - \s)} + X^{a - \s} \r) & \text{if \; $|s - z| > 1 / \log{X}$.}
		\end{array}
		\r.
	\end{align}
\end{lemma}

\begin{proof}
	This lemma can be proved in the same way as \cite[Lemma 3]{II2019}.
\end{proof}

\begin{proof}[Proof of Theorem \ref{Main_F_S}]
	It follows from Proposition \ref{GSEF}, Lemma \ref{GUE}, and Lemma \ref{EUE} that
  \begin{align*}
    \log{F(s)} = 
		\sum_{2 \leq n \leq X^{1+1/H}}\frac{\Lam_{F}(n)v_{f, H}\l( e^{\log{n}/\log{X}} \r)}{n^{s}\log{n}}
		+\sum_{|s - \rho_{F}| \leq \frac{1}{\log{X}}}\log((s - \rho_{F})\log{X}) + R_{F}(s, X, H),    
  \end{align*}
  where
	\begin{multline}	
		R_{F}(s, X, H)
		\ll_{f, d} \frac{m_{F}(X^{2(1-\s)} + X^{1-\s})}{t\log{X}}
		+ \sum_{|t - \gamma_{F}| \leq \frac{1}{\log{X}}}(X^{2(\b_{F} - \s)} + X^{\b_{F} - \s})\\
		+ \frac{1}{\log{X}}\sum_{|t - \gamma_{F}| > \frac{1}{\log{X}}}
		\frac{X^{2(\b_{F} - \s)} + X^{\b_{F} - \s}}{|t - \gamma_{F}|}
		\min_{0 \leq l \leq d}\l\{\l( \frac{H}{|t - \gamma_{F}|\log{X}} \r)^{l}\r\}.
	\end{multline}
	Hence, it suffices to show estimate \eqref{ESR_S2} on the range $|t| \geq t_{0}(F)$.
	Following the proof of \cite[Proposition 2]{II2019}, we see that it suffices to check
	\begin{align}	\label{KSIE}
		\sum_{\rho_{F}}\frac{\s_{X, t}(F) - 1/2}{(\s_{X, t}(F) - \b_{F})^2 + (t - \gamma_{F})^2}
		\ll \l| \sum_{n \leq X^3}\frac{\Lam_{F}(n)w_{X}(n)}{n^{\s_{X, t}(F) + it}} \r| + d_{F} \log{|t|},
	\end{align}
	which can be shown by the same proofs as  in \cite[eq. (4.4); eq. (4.7)]{SCR} by using equation \eqref{GLFE} instead of \cite[Lemma 11]{SCR}.
\end{proof}

\subsection{Proof of Propositions \ref{KLI} and Proposition \ref{KLST}}


\begin{lemma}	\label{SLL}
	Let $T \geq 5$, and let $X \geq 3$. Let $k$ be a positive integer such that $X^{k} \leq T$.
	Then, for any complex numbers $a(p)$, we have
	\begin{align*}
		\int_{T}^{2T}\bigg| \sum_{p \leq X}\frac{a(p)}{p^{it}} \bigg|^{2k}dt
		\ll T k! \l( \sum_{p \leq X}|a(p)|^2 \r)^{k}.
	\end{align*}
	Here, the above sums run over prime numbers.
\end{lemma}

\begin{proof}
	By \cite[Lemma 3.3]{KTDT}, we have
  \begin{align*}
    \int_{T}^{2T}\bigg| \sum_{p \leq X}\frac{a(p)}{p^{it}} \bigg|^{2k}dt
    &\leq T k! \l( \sum_{p \leq X}|a(p)|^2 \r)^{k}
    + O\l( k! \l( \sum_{p \leq X}p|a(p)|^2 \r)^{k} \r)\\
    &\leq T k! \l( \sum_{p \leq X}|a(p)|^2 \r)^{k}
    + O\l( k! X^{k}\l( \sum_{p \leq X}|a(p)|^2 \r)^{k} \r).
  \end{align*}
  Hence, we obtain this lemma when $X^{k} \leq T$.
\end{proof}


The next lemma is an analogue and a generalization of \cite[Lemma 5.2]{KTDT}.

\begin{lemma}	\label{SL_s_X}
	Let $F \in \Sc^{\dagger} \setminus \{ 1 \}$
	be an $L$-function satisfying (A3).
	Let $T$ be large, and $\kappa_{F}$ be the positive constant in \eqref{ZDC1}.
	For $k \in \ZZ_{\geq 1}$, ${3 \leq X \leq T^{2/3}}$, $\xi \geq 1$
	with $X \xi \leq T^{\kappa_{F} / 4}$, we have
	\begin{align*}
		\int_{T}^{2T}\l( \s_{X, t}(F) - \frac{1}{2} \r)^{k}\xi^{\s_{X, t}(F) - 1/2}dt
		\ll_{F} T \l( \frac{4^{k} \xi^{\frac{4}{\log{X}}}}{(\log{X})^{k}} + \frac{C^{k} k!}{\log{X}(\log{T})^{k-1}} \r),
	\end{align*}
	where $C = C(F)$ is a positive constant.
\end{lemma}

\begin{proof}
	By definition \eqref{def_s_X_F} of $\s_{X, t}(F)$, we obtain
	\begin{align}	\label{SL_INE1}
		 & \int_{T}^{2T}\l( \s_{X, t}(F) - \frac{1}{2} \r)^{k}\xi^{\s_{X, t}(F) - 1/2}dt \\
		 & \leq T \xi^{\frac{4}{\log{X}}}\l( \frac{4}{\log{X}} \r)^{k}
		+ \frac{2^{k}}{\log{X}}
		\sum_{\substack{T - \frac{X^{3|\b_{F} - 1/2|}}{\log{X}} \leq \gamma_{F} \leq 2T + \frac{X^{3|\b_{F} - \frac{1}{2}|}}{\log{X}}
		\\ \b_{F} \geq 1/2}}
		\l(\b_{F} - \frac{1}{2} \r)^{k}(X^3 \xi^2)^{\b_{F} - \frac{1}{2}}.
	\end{align}
	From the equation
	\begin{align*}
		 & \l(\b_{F} - \frac{1}{2} \r)^{k} (X^3 \xi^2)^{\b_{F} - \frac{1}{2}}
		= \int_{\frac{1}{2}}^{\b_{F}}
		\l\{ (\s - 1/2)^{k}\log{(X^3 \xi^2)} + k \l(\s - 1/2 \r)^{k-1} \r\}(X^3 \xi^2)^{\s - \frac{1}{2}}d\s,
	\end{align*}
	we find that
	\begin{align*}
		 & \sum_{\substack{T - \frac{X^{3|\b_{F} - 1/2|}}{\log{X}} \leq \gamma_{F} \leq 2T + \frac{X^{3|\b_{F} - \frac{1}{2}|}}{\log{X}}
		\\ \b_{F} \geq 1/2}}
		\l(\b_{F} - \frac{1}{2} \r)^{k}(X^3 \xi^2)^{\b_{F} - \frac{1}{2}}                                                                \\
		 & \leq \sum_{\substack{0 \leq \gamma_{F} \leq 3T                                                                                \\ \b_{F} \geq 1/2}}\int_{\frac{1}{2}}^{\b_{F}}
		\l\{ (\s - 1/2)^{k}\log{(X^3 \xi^2)} + k \l(\s - 1/2 \r)^{k-1} \r\}(X^3 \xi^2)^{\s - \frac{1}{2}}d\s                             \\
		 & \leq \int_{\frac{1}{2}}^{1}\l\{ (\s - 1/2)^{k}\log{(X^3 \xi^2)} + k \l(\s - 1/2 \r)^{k-1} \r\}(X^3 \xi^2)^{\s - \frac{1}{2}}
		\sum_{\substack{0 \leq \gamma_{F} \leq 3T                                                                                        \\ \b_{F} \geq \s}}1 d\s\\
		 & = \int_{\frac{1}{2}}^{1}\l\{ (\s - 1/2)^{k}\log{(X^3 \xi^2)} + k \l(\s - 1/2 \r)^{k-1} \r\}(X^3 \xi^2)^{\s - \frac{1}{2}}
		N_{F}(\s, 3T)d\s.
	\end{align*}
	By assumption (A3), we can use the estimate $N_{F}(\s, T) \ll_{F} T^{1 - \kappa_{F}(\s - 1/2)}\log{T}$,
	and so, for $X \xi \leq T^{\kappa_{F}/4}$, the above most right hand side is
	\begin{align*}
		 & \ll_{F} T \log{T}\int_{\frac{1}{2}}^{1}\l\{ (\s - 1/2)^{k}\log{(X^3 \xi^2)} + k \l(\s - 1/2 \r)^{k-1} \r\}
		\l(\frac{X^3 \xi^2}{T^{\kappa_{F}}}\r)^{\s - \frac{1}{2}}d\s
		\ll T \frac{C^{k} k!}{(\log{T})^{k-1}}
	\end{align*}
	for some $C = C(F) > 0$.
	Hence, by this estimate and inequality \eqref{SL_INE1}, we obtain this lemma.
\end{proof}



\begin{lemma}\label{MomentLambdab}
	Let $F$ be a Dirichlet series satisfying (S4), (S5), and \eqref{SNC}.
	Suppose there exists some positive constant $A$ such that $|w(n)|\leq (\log n)^A$.
	Then for any $c \in \RR$, $\sigma\geq \frac{1}{2}$, $k \in \ZZ_{\geq 1}$ and any $X \geq Z \geq 2$ with $X^{k} \leq T$, we have
	\begin{align}
		 & \int_T^{2T}\left| \sum_{Z \leq n \leq X}\frac{b_{F}(n)w(n) (\log(X n))^{c}}{n^{\s + it }} \right| ^{2k}dt \\
		 & \leq T C^{k} k^{k} \l((\log{X})^{2c} \sum_{Z \leq p \leq X}\frac{|b_{F}(p) w(p)|^2}{p^{2\sigma}}\r)^k
		+ T C^{k} k^{k} Z^{k(1 - 2\s)}(\log{X})^{2(A + c)k},
	\end{align}
	where $C=C(F)$ is some positive constant.
\end{lemma}

\begin{proof}
	Set $K_{1} = 2(1/2-\vartheta_F)^{-1}$.
	Using $|w(n)| \leq (\log n)^A$, we find that for any $\sigma \geq \frac{1}{2}$
	\begin{align}
		\Bigg|\sum_{\ell > K_{1}}\sum_{Z\leq p^{\ell} \leq X}
		\frac{b_{F}(p^{\ell}) w(p^{\ell}) (\log(X p^{\ell}))^{c}}{p^{\ell(\sigma + it)}}\Bigg|
		 & \ll_{F} (\log{X})^{c}\sum_{p>Z}\sum_{\ell > K_{1}}\frac{(\log (p^\ell))^A}{p^{\ell(\sigma-1/2)}p^{(1/2-\theta_F) \ell}} \\
		 & \ll_{F} (\log{X})^{c} Z^{1/2-\sigma}\sum_{p}\frac{1}{p^{2-\epsilon}}
		\ll_{F} (\log{X})^{c} Z^{1/2-\sigma}.\label{LargePrimespower}
	\end{align}
	Therefore, by Lemma \ref{SLL}, we have
	\begin{align}
		 & \int_T^{2T}\left|\sum_{Z\leq n\leq X}\frac{b_F(n)w(n)(\log(Xn))^{c}}{n^{\sigma+it}}\right|^{2k}dt \\
		 & \leq T C^k k^k \sum_{1 \leq \ell \leq K_1}
		\l(\sum_{Z^{1/\ell}\leq p\leq X^{1/\ell}}\frac{|b_F(p^\ell) w(p^\ell) (\log(X p^{\ell}))^{c}|^2}{p^{2\ell \sigma}}\r)^k
		+ T C^{k} Z^{k(1 - 2\s)} (\log{X})^{2ck}
	\end{align}
	for some constant $C=C(F)>0$.
	From assumption (S5), we see that, for $2 \leq \ell \leq K_{1}$,
	\begin{align*}
		\sum_{Z^{1/\ell}\leq p\leq X^{1/\ell}}\frac{|b_F(p^\ell) w(p^\ell) (\log(X p^{\ell}))^{c}|^2}{p^{2\ell \sigma}}
		\leq (\log{X})^{2(A + c)}Z^{1 - 2\s}\sum_{p}\frac{|b_{F}(p^{\ell})|^{2}}{p^{\ell}}
		\ll_{F} Z^{1 - 2\s} (\log{X})^{2(A + c)},
	\end{align*}
	which completes the proof of this lemma.
\end{proof}

\begin{lemma}\label{SumbLargeP}
	Let $F$ be a Dirichlet series satisfying (S4), (S5), and \eqref{SNC}.
	Then for any $x \geq 3$ and $c > 1$,
	\begin{gather}
		\label{SumbLargeP1}
		2\s_{F}(x)^{2}
		= \sum_{n\leq x}\frac{|b_F(n)|^2}{n} + O_{F}(1)
		=\sum_{p\leq x}\frac{|b_F(p)|^2}{p} + O_{F}(1)
		=n_F\log \log x+O_{F}(1),\\
		\label{SumbLargeP2}
		\sum_{n \geq x }\frac{|b_F(n)|^2}{n^c}\ll x^{1-c}, \quad x \rightarrow \infty,
	\end{gather}
	where $\s_{F}(x)$ is defined by \eqref{def_var}.
\end{lemma}

\begin{proof}
	Set $K_1=2(1/2-\vartheta_F)^{-1}$.
	Similar to \eqref{LargePrimespower}, we have
	\begin{align}
		\sum_{\ell \geq K_1}\sum_{p}\frac{|b_F(p^{\ell})|^2}{p^{\ell}}
		\ll_{F} \sum_{p}\sum_{\ell}\frac{1}{p^{2}}
		\ll 1.
	\end{align}
	From assumption (S5), we also have
	\begin{align}
		\sum_{2 \leq \ell \leq K_1}	\sum_{p}\frac{|b_F(p^\ell)|^2}{p^\ell}
		\ll_{F} 1,
	\end{align}
	and thus 
	\begin{align*}
		\sum_{\ell = 2}^{\infty}	\sum_{p}\frac{|b_F(p^\ell)|^2}{p^\ell}
		\ll_{F} 1.
	\end{align*}
	Therefore, by assumption \eqref{SNC}
	\begin{align}
		2\s_{F}(x)^2
		&= \sum_{p \leq x}\sum_{\ell = 1}^{\infty}\frac{|b_F(p^{\ell})|^2}{p^{\ell}}
		= \sum_{n\leq x}\frac{|b_F(n)|^2}{n} + O_{F}(1)\\
		&= \sum_{p\leq x}\frac{b_F(p)}{p} + O_{F}(1)
		= n_F\log\log x + O_{F}(1),
	\end{align}
  which is the assertion of \eqref{SumbLargeP1}.

	From partial summation and \eqref{SumbLargeP1}, we also find that
	\begin{align}
		\sum_{n\geq x}\frac{|b_F(n)|^2}{n^c}
		 & =\int_{x}^{\infty}\frac{1}{\xi^{c-1}}d\sum_{x \leq n \leq \xi}\frac{|b_F(n)|^2}{n}       \\
		 & = c\int_{x}^{\infty}\frac{n_{F}\log{\log{\xi} - n_{F}\log{\log{x}} + O_{F}(1)}}{\xi^{c}}
		\ll_{F} x^{1-c}.
	\end{align}
	\end{proof}

\begin{lemma}	\label{ESMDP}
	Let $F \in \Sc^{\dagger}$ be an $L$-function satisfying \eqref{SNC} and (A3).
	Let $T$ be large. Put $\delta_{F} = \min\{ \frac{1}{4}, \frac{\kappa_{F}}{20} \}$
	with $\kappa_{F}$ the positive constant in \eqref{ZDC1}.
	For any $k \in \ZZ_{\geq 1}$, $X \geq 3$ with $X \leq T^{\delta_{F} / k}$, we have
	\begin{align*}
		\int_{T}^{2T}\Bigg| \sum_{n \leq X^{3}}\frac{\Lam_{F}(n) w_{X}(n)}{n^{\s_{X, t}(F) + it}} \Bigg|^{2k}dt
		\ll T C^{k} k^{k} (\log{X})^{2k},
	\end{align*}
	where $w_{X}$ is the smoothing function defined by \eqref{def_w_X}, and $C = C(F)$ is a positive constant.
\end{lemma}

\begin{proof}
	For brevity, we write $\s_{X, t}(F)$ as $\s_{X, t}$.
	\begin{align*}
		\sum_{n \leq X^{3}}\frac{\Lam_{F}(n) w_{X}(n)}{n^{\s_{X, t} + it}}
		 & = \sum_{n \leq X^{3}}\frac{\Lam_{F}(n) w_{X}(n)}{n^{1/2 + it}}
		- \sum_{n \leq X^{3}}\frac{\Lam_{F}(n) w_{X}(n)}{n^{1/2 + it}}(1 - n^{1/2 - \s_{X, t}}) \\
		 & = \sum_{n \leq X^{3}}\frac{\Lam_{F}(n) w_{X}(n)}{n^{1/2 + it}}
		-\int_{1/2}^{\s_{X, t}}\sum_{n \leq X^{3}}\frac{\Lam_{F}(n) w_{X}(n) \log{n}}{n^{\a' + it}}d\a',
	\end{align*}
	Note that for $1/2 \leq \a' \leq \s_{X, t}$,
	\begin{align*}
		\bigg| \sum_{n \leq X^{3}}\frac{\Lam_{F}(n) w_{X}(n) \log{n}}{n^{\a' + it}} \bigg|
		 & = \Bigg|X^{\a' - 1/2} \int_{\a'}^{\infty}X^{1/2 - \a}
		\sum_{n \leq X^{3}}\frac{\Lam_{F}(n) w_{X}(n) \log{(X n)} \log{n}}{n^{\a + it}}d\a\Bigg|                                                   \\
		 & \leq X^{\s_{X, t}  - 1/2}\int_{1/2}^{\infty}X^{1/2 - \a}
		\bigg|\sum_{n \leq X^{3}}\frac{\Lam_{F}(n) w_{X}(n) \log{(Xn)} \log{n}}{n^{\a + it}}\bigg|d\a.
	\end{align*}
	Thus
	\begin{align}\label{S1+S2}
		 & \frac{1}{T}\int_T^{2T}\left|\sum_{n\leq X^3}\frac{\Lambda_F(n)w_X(n)}{n^{\sigma_{X,t}+it}} \right|^{2k}dt\ll C_1^k (\mathcal{S}_1+\mathcal{S}_2),
	\end{align}
	where
	\begin{align}
		 & \mathcal{S}_1=\frac{1}{T}\int_{T}^{2T}\left|\sum_{n\leq X^3}\frac{\Lambda_F(n)w_X(n)}{n^{1/2+it}}\right|^{2k}dt,         \\
		 & \mathcal{S}_2=\frac{1}{T}\int_T^{2T}\left| (\sigma_{X,t}-\frac{1}{2})X^{\s_{X, t}  - 1/2}\int_{1/2}^{\infty}X^{1/2 - \a}
		\bigg|\sum_{n \leq X^{3}}\frac{\Lam_{F}(n) w_{X}(n) \log{(Xn)} \log{n}}{n^{\a + it}}
		\bigg|d\a\right|^{2k}dt.
	\end{align}
	Since $\Lambda_F(n)=b_F(n)\log n$ and $b_F(p)=a_F(p)$, we have
	\begin{align}
		\sum_{n \leq X^{3}}\frac{|\Lam_{F}(n)|^2}{n}
		 & = (\log{X^{3}})^{2}\sum_{p \leq X^{3}}\frac{|a_{F}(p)|^2}{p}
		- \int_{2}^{X^{3}}\frac{2\log{\xi}}{\xi}\sum_{p \leq \xi}\frac{|a_{F}(p)|^2}{p} d\xi+O(1) \\ \nonumber
		 & = n_{F}(\log{X^{3}})^{2}\log{\log{X^3}}
		- 2n_{F}\int_{2}^{X^{3}}\frac{\log{\xi} \times \log{\log{\xi}}}{\xi}d\xi
		+ O_{F}\l( (\log{X})^{2} \r)                                                              \\ \nonumber
		 & = n_{F}(\log{X^{3}})^{2}\log{\log{X^3}}
		- n_{F}(\log{X^{3}})^{2}\log{\log{X^3}} + \frac{n_{F}}{2}(\log{X^3})^2
		+ O_{F}\l( (\log{X})^{2} \r)                                                              \\
		\label{Lambda^2/p}
		 & \ll_{F} (\log{X})^{2}.
	\end{align}Combining Lemma \ref{MomentLambdab} with \eqref{Lambda^2/p}, we obtain
	\begin{align}\label{S1}
		\mathcal{S}_1\ll C_2^k k^k \left(\sum_{n\leq X^3}\frac{|b_F(n)|^2\log^2 n}{n}\right)^k\ll C_3^k k^k (\log X)^{2k}.
	\end{align}
	By the Cauchy-Schwarz inequality and Lemma \ref{SL_s_X}, when $\delta_{F} \leq \kappa_{F} / 20$,
	it holds that
	\begin{align}
		\mathcal{S}_2
		 & \leq \Bigg(\int_{T}^{2T}(\s_{X, t}-1/2)^{4k}X^{4k(\s_{X, t} - 1/2)}dt\Bigg)^{1/2} \times \\
		\nonumber
		 & \qquad \quad \times \l(\int_{T}^{2T}\l( \int_{1/2}^{\infty}X^{1/2 - \a}
		\bigg|\sum_{n \leq X^{3}}
		\frac{\Lam_{F}(n) w_{X}(n) \log{(X n)} \log{n}}{n^{\a + it}}\bigg|d\a\r)^{4k}dt
		\r)^{\frac{1}{2}}                                                                           \\
		 & \ll\frac{T^{\frac{1}{2}} C_{4}^{k}}{(\log{X})^{2k}}
		\l(\int_{T}^{2T}\l( \int_{1/2}^{\infty}X^{1/2 - \a}
		\bigg|\sum_{n \leq X^{3}}
		\frac{\Lam_{F}(n) w_{X}(n) \log{(X n)} \log{n}}{n^{\a + it}}\bigg|d\a\r)^{4k}dt\r)^{\frac{1}{2}}\label{S21},
	\end{align}
	for some constant $C_{4} = C_{4}(F)>0$.
	We apply H\"older's inequality to the innermost integral to obtain
	\begin{align*}
		 & \l( \int_{1/2}^{\infty}X^{\frac{1}{2} - \a}
		\bigg|\sum_{n \leq X^{3}}\frac{\Lam_{F}(n) w_{X}(n) \log{(X n)} \log{n}}{n^{\a + it}}
		\bigg|d\a\r)^{4k}                                                      \\
		 & \leq \l(\int_{1/2}^{\infty}X^{\frac{1}{2} - \a}d\a\r)^{4k-1} \times
		\l( \int_{1/2}^{\infty} X^{\frac{1}{2} - \a}\bigg|
		\sum_{n \leq X^{3}}\frac{\Lam_{F}(n) w_{X}(n) \log{(X n)} \log{n}}
		{n^{\a + it}}\bigg|^{4k}d\a \r)                                      \\
		 & = \frac{1}{(\log{X})^{4k-1}}
		\int_{1/2}^{\infty} X^{1/2 - \a}\bigg|\sum_{n \leq X^{3}}
		\frac{\Lam_{F}(n) w_{X}(n) \log{(X n)} \log{n}}{n^{\a + it}}\bigg|^{4k}d\a.
	\end{align*}
	Therefore, by Lemma \ref{MomentLambdab} and \eqref{Lambda^2/p}, we find that
	\begin{align}
		 & \int_{T}^{2T}\l( \int_{1/2}^{\infty}X^{\frac{1}{2} - \a}
		\bigg|\sum_{n \leq X^{3}}\frac{\Lam_{F}(n) w_{X}(n) \log{(X n)} \log{n}}{n^{\a + it}}
		\bigg|d\a\r)^{4k}dt\nonumber                                                                                            \\
		 & \leq \frac{1}{(\log{X})^{4k-1}}\int_{1/2}^{\infty} X^{1/2-\a}\l(
		\int_{T}^{2T}\bigg|\sum_{n \leq X^{3}}\frac{b_{F}(n) \log{(X n)} \log{n}}{n^{\a + it}}
		\bigg|^{4k}dt\r)d\a\nonumber                                                                                            \\
		 & \leq \frac{T C^{k} k^{2k}}{(\log{X})^{4k-1}}\int_{1/2}^{\infty}X^{1/2-\a}
		\l\{\l( (\log{X})^{2}\sum_{n \leq X^{3}}\frac{|\Lam_{F}(n)|^2 (\log{n})^2}{n^{2  \a}} \r)^{2k} + (\log{X})^{8k} \r\}d\a \\
		 & \ll \frac{T C^{k} k^{2k}}{(\log{X})^{4k}}
		\l( \l((\log{X})^{4} \sum_{n \leq X^{3 }}\frac{|\Lam_{F}(n)|^2}{n} \r)^{2k} + (\log{X})^{8k} \r)                        \\
		 & \ll T C_{5}^{k} k^{2k} (\log{X})^{8k}.\label{S22}
	\end{align}
	Combining \eqref{S1+S2}, \eqref{S1}, \eqref{S21} and \eqref{S22}, we complete the proof of Lemma \ref{ESMDP}.
\end{proof}

\begin{lemma}	\label{ESRSIZ}
	Let $F \in \Sc^{\dagger}$ be an $L$-function satisfying \eqref{SNC} and (A3).
	Let $\s \geq 1/2$,  $T$ be large.
	Let $\kappa_{F}$, $\delta_{F}$ be the same constants as in Lemma \ref{ESMDP}.
	There exists a positive constant $C = C(F)$ such that
	for any $k \in \ZZ_{\geq 1}$, $3 \leq X \leq T^{\delta_{F}/k}$,
	\begin{align}	\label{ESRSIZ1}
		\int_{T}^{2T}\l( \sum_{|\s + it - \rho_F| \leq \frac{1}{\log{X}}}1 \r)^{2k}dt
		\leq C^{k} T^{1 - (2\s - 1)\delta_{F} + \frac{8\delta_{F}}{\log{X}}}\l( \frac{\log{T}}{\log{X}} \r)^{2k},
	\end{align}
	and
	\begin{align}	\label{ESRSIZ2}
		\begin{aligned}
			 & \int_{T}^{2T}\Bigg| \sum_{|\s + it - \rho_{F}| \leq \frac{1}{\log{X}}}\log((\s + it - \rho_{F}) \log{X}) \Bigg|^{k} dt \\
			 & \leq (C k)^{k} T^{1 - (\s - 1/2)\delta_{F} + \frac{4\delta_{F}}{\log{X}}}
			\l( \frac{\log{T}}{\log{X}} \r)^{k+\frac{1}{2}}.
		\end{aligned}
	\end{align}
\end{lemma}

\begin{proof}
	From \eqref{def_s_X_F},
	there are no zeros of $F$ with $|\s + it - \rho_{F}| \leq \frac{1}{\log{X}}$ when $\s \geq \s_{X, t}(F)$.
	Put $\xi := T^{\delta_{F} / k}$.
	Note that $\xi \geq 1$ because we suppose that $3 \leq X \leq T^{\delta_{F}/k} = \xi$.
	From these facts, we have
	\begin{align*}
		\sum_{|\s + it - \rho_{F}| \leq \frac{1}{\log{X}}}1
		\leq \xi^{\s_{X, t}(F) - \s}\sum_{|t - \gamma_{F}| \leq \frac{1}{\log{X}}}1
	\end{align*}
	for $\s \geq 1/2$.
	By definition \eqref{def_s_X_F} and the line symmetry of nontrivial zeros of $F$, we find that
	\begin{align*}
		\sum_{|t - \gamma_{F}| \leq \frac{1}{\log{X}}}1
		\leq 2\sum_{\substack{|t - \gamma_{F}| \leq \frac{1}{\log{X}} \\ \b_{F} \geq 1/2}}1
		\ll \sum_{\substack{|t - \gamma_{F}| \leq \frac{1}{\log{X}}   \\ \b_{F} \geq 1/2}}
		\frac{(\s_{X, t}(F) - 1/2)^2}{(\s_{X, t}(F) - \b_{F})^2 + (t - \gamma_{F})^2}.
	\end{align*}
	Applying estimate \eqref{KSIE} to the above right hand side, we obtain
	\begin{align}	\label{NZVDP}
		\sum_{|t - \gamma_{F}| \leq \frac{1}{\log{X}}}1
		\ll (\s_{X, t}(F) - 1/2)\l(\bigg| \sum_{n \leq X^3}\frac{\Lam_{F}(n) w_{X}(n)}{n^{\s_{X, t}(F) + it}}\bigg| + d_{F}\log{T} \r)
	\end{align}
	for $t \in [T, 2T]$.
	Noting $X \xi^{2k} \leq T^{\kappa_{F} / 4}$ and using Lemmas \ref{SL_s_X}, \ref{ESMDP},
	we have
	\begin{align}
		 & \int_{T}^{2T}(\s_{X, t}(F) - 1/2)^{2k}\xi^{2k(\s_{X, t}(F) - \s)}
		\l(\bigg| \sum_{n \leq X^3}\frac{\Lam_{F}(n) w_{X}(n)}{n^{\s_{X, t}(F) + it}}\bigg| + d_{F}\log{T} \r)^{2k}dt \\
		 & \leq C^{k}\xi^{2k(1/2 - \s)}\Bigg\{(\log{T})^{2k}
		\int_{T}^{2T}(\s_{X, t}(F) - 1/2)^{2k}\xi^{2k(\s_{X, t}(F)-1/2)}dt                                          \\
		 & + \l(\int_{T}^{2T}(\s_{X, t}(F) - 1/2)^{2k}\xi^{2k(\s_{X, t}(F)-1/2)}dt\r)^{1/2}
		\l( \int_{T}^{2T}\bigg| \sum_{n \leq X^3}\frac{\Lam_{F}(n) w_{X}(n)}{n^{\s_{X, t}(F) + it}}\bigg|^{2k}dt \r)^{1/2}
		\Bigg\}                                                                                                         \\
		 & \leq C^{k} \xi^{2k(1/2 - \s)} \l(T \xi^{\frac{8k}{\log{X}}} \l( \frac{\log{T}}{\log{X}} \r)^{2k}
		+ T \xi^{\frac{4k}{\log{X}}} k^{k/2}\r)                                                                         \\
		\label{PESRSIZ1}
		 & \leq C_{2}^{k} T^{1 - (2\s - 1)\delta_{F} + \frac{8\delta_{F}}{\log{X}}}\l( \frac{\log{T}}{\log{X}} \r)^{2k}
	\end{align}
	for some constant $C_{2} = C_{2}(F) > 0$.
	Hence, we obtain estimate \eqref{ESRSIZ1}.

	Next, we show estimate \eqref{ESRSIZ2}.
	We find that
	\begin{align*}
		\Bigg|\sum_{|\s + it - \rho_{F}| \leq \frac{1}{\log{X}}}\log\l( (\s + it - \rho_{F})\log{X} \r)\Bigg|
		\leq \l(g_{X}(s) + \pi\r)
		\times \sum_{|\s + it - \rho_{F}| \leq \frac{1}{\log{X}}}1,
	\end{align*}
	where $g_{X}(s) = g_{X}(\s+it)$ is the function defined by
	\begin{align*}
		g_{X}(s) = \l\{
		\begin{array}{cl}
			\log\l(\frac{1}{|(\s + it - \rho_{s})\log{X}|}\r)
			  & \text{if there exists a zero $\rho_{F}$ with $|\s + it - \rho_{F}| \leq \frac{1}{\log{X}}$,} \\
			0 & \text{otherwise.}
		\end{array}
		\r.
	\end{align*}
	Here, $\rho_{s}$ indicates the zero of $F$ nearest from $s = \s + it$.
	By using the Cauchy-Schwarz inequality and estimate \eqref{ESRSIZ1}, we obtain
	\begin{align}
		\label{PESRSIZ2}
		 & \int_{T}^{2T}\Bigg| \sum_{|\s + it - \rho_{F}| \leq \frac{1}{\log{X}}}\log((\s + it - \rho_{F}) \log{X}) \Bigg|^{k} dt \\
		 & \leq C_{3}^{k}\l(\int_{T}^{2T} g_{X}(\s+it)^{2k}dt + \pi^{2k} T\r)^{1/2}
		\times T^{\frac{1}{2} - (\s - 1/2)\delta_{F} + \frac{4\delta_{F}}{\log{X}}}
		\l( \frac{\log{T}}{\log{X}} \r)^{k}
	\end{align}
	for some constant $C_{3} = C_{3}(F) > 0$.
	Moreover, we find that
	\begin{align*}
		\int_{T}^{2T}g_{X}(s)^{2k}dt
		 & \leq \int_{T}^{2T}\sum_{|\s + it - \rho_{F}| \leq \frac{1}{\log{X}}}
		\l( \log{\l( \frac{1}{|\s+it-\rho_{F}|\log{X}} \r)} \r)^{2k}dt                             \\
		 & \leq \sum_{T-\frac{1}{\log{X}} \leq \gamma_{F} \leq 2T + \frac{1}{\log{X}}}
		\int_{\gamma_{F} - \frac{1}{\log{X}}}^{\gamma_{F} + \frac{1}{\log{X}}}
		\l(\log\l( \frac{1}{|t - \gamma_{F}|\log{X}} \r)\r)^{2k}dt                                 \\
		 & \ll \frac{1}{\log{X}}\sum_{T-1 \leq \gamma_{F} \leq 2T + 1}
		\int_{0}^{1}\l(\log\l( \frac{1}{\ell} \r)\r)^{2k}d\ell                                     \\
		 & \ll_{F} T\frac{\log{T}}{\log{X}}\int_{0}^{1}\l(\log\l( \frac{1}{\ell} \r)\r)^{2k}d\ell.
	\end{align*}
	By induction, we can easily confirm that the last integral is equal to $(2k)!$. Hence, we obtain
	\begin{align*}
		\int_{T}^{2T}g_{X}(s)^{2k}dt
		\ll_{F} (2k)!T\frac{\log{T}}{\log{X}}.
	\end{align*}
	By substituting this estimate to inequality \eqref{PESRSIZ2}, we obtain this lemma.
\end{proof}

\begin{proof}[Proof of Proposition \ref{KLI}]
	Let $f$ be a fixed function satisfy the condition of this paper (see Notation) and $D(f) \geq 2$.
	Let $\s \geq 1/2$, $k \in \ZZ_{\geq 1}$, and $3 \leq Z := T^{\delta_{F} / k}$,
	where $\delta_{F} = \min\{ \frac{\kappa_{F}}{20}, \frac{1}{4} \}$.
	It holds that
	\begin{align*}
		Z^{2(\s_{Z, t}(F) - \s)} + Z^{\s_{Z, t}(F) - \s}
		&= Z^{2(1/2-\s)} \cdot Z^{2(\s_{Z, t}(F) - 1/2)} + Z^{1/2-\s} \cdot Z^{\s_{Z, t}(F) - 1/2}\\
		&\leq 2Z^{1/2 - \s} \cdot Z^{2(\s_{Z, t}(F) - 1/2)}
	\end{align*}
	for $\s \geq 1/2$.
	Using this inequality and estimate \eqref{ESR_S2} as $H = 1$, we find that
	there exists a positive constant $C_{1} = C_{1}(F)$ such that
	\begin{align}	\label{PKLI0}
		 & \bigg| \log{F(\s + it)} - P_{F}(\s + it, Z)
		- \sum_{|\s + it - \rho_{F}| \leq \frac{1}{\log{Z}}}\log((\s + it - \rho_{F})\log{Z}) \bigg|^{2k} \\
		 & \leq C_{1}^{k}\bigg| \sum_{Z < n \leq Z^{2}}
		\frac{\Lam_{F}(n) v_{f, 1}(e^{\log{n} / \log{Z}})}{n^{\s+it}\log{n}} \bigg|^{2k} +                \\ \nonumber
		 & \quad + C_{1}^{k}Z^{2k(1/2-\s)}(\s_{Z, t}(F) - 1/2)^{2k}Z^{4k(\s_{Z, t}(F) - 1/2)}
		\l( \bigg| \sum_{n \leq Z^3}\frac{\Lam_{F}(n) w_{Z}(n)}{n^{\s_{Z, t}(F) + it}} \bigg| + \log{T} \r)^{2k}
	\end{align}
	for $t \in [T, 2T]$.
	Since  $|v_{f, 1}(e^{\log{p^{\ell}} / \log{Z}})| \leq 1$, $\Lam_{F}(n) = b_{F}(n)\log n$, and $a_{F}(p) = b_{F}(p)$,
	we can apply Lemma \ref{MomentLambdab} and Lemma \ref{SumbLargeP} to bound the first sum in \eqref{PKLI0} so that
	\begin{align}
		\hspace{-5mm}
		\label{PKLI3}
		\frac{1}{T}\int_{T}^{2T}\bigg| \sum_{Z < n \leq Z^{2}}
		\frac{\Lam_{F}(n) v_{f, 1}(e^{\log{n} / \log{Z}})}{n^{\s+it}\log{n}} \bigg|^{2k}dt
		\leq C_3^k k^k\l( \sum_{Z< n \leq Z^2}\frac{|a_{F}(p)|^2}{p^{2\s}}\r)^{k}+C_3^k k^{k} Z^{k(1 - 2\sigma)}.
	\end{align}
	For the second sum  in \eqref{PKLI0} , we can follow the proof of estimate \eqref{PESRSIZ1} and obtain
	\begin{align}
		\nonumber
		 & Z^{2k(1/2 - \s)}\int_{T}^{2T} (\s_{Z, t}(F) - 1/2)^{2k}Z^{4k(\s_{Z, t}(F) - 1/2)}
		\l( \bigg| \sum_{n \leq Z^3}\frac{\Lam_{F}(n) w_{Z}(n)}{n^{\s_{Z, t}(F) + it}} \bigg| + \log{T} \r)^{2k} dt \\
		\label{PKLI1}
		 & \ll C^{k} T^{1 - \delta_{F}(2\s - 1) + \frac{16\delta_F}{\log{Z}}} \l( \frac{\log{T}}{\log{Z}} \r)^{2k}
		\ll T^{1 - \delta_{F}(2\s - 1)} C_{2}^{k} k^{2k}
	\end{align}
	for some positive constant $C_{2} = C_{2}(F)$.
		Combining \eqref{PKLI0}, \eqref{PKLI1} and \eqref{PKLI3}, we have 
  \begin{align} \label{PKLI4}
    & \frac{1}{T}\int_{T}^{2T}\bigg| \log{F(\s + it)} - P_{F}(\s + it, Z)
    - \sum_{|\s + it - \rho_{F}| \leq \frac{1}{\log{Z}}}\log((\s + it - \rho_{F})\log{Z}) \bigg|^{2k}dt \\
    & \leq C^{k} k^{2k}T^{\delta_{F}(1 - 2\s)}
    + C^{k} k^{k} \l( \sum_{Z< p \leq Z^{2}}\frac{|a_{F}(p)|^2}{p^{2\s}} \r)^{k}
  \end{align}
  for some constant $C = C(F) > 0$.

 Now for any $3 \leq X \leq Y = T^{1/k}$, 
  we write
  \begin{align*}
    &\log{F(\s + it)} - P_{F}(\s + it, X)
		- \sum_{|\s + it - \rho_{F}| \leq \frac{1}{\log{Y}}}\log((\s + it - \rho_{F})\log{Y})\\
    &= \log{F(\s + it)} - P_{F}(\s + it, Z) - \sum_{|\s + it - \rho_{F}| \leq \frac{1}{\log{Z}}}\log((\s + it - \rho_{F})\log{Z})\\
    &- (P_{F}(\s + it, X) - P_{F}(\s + it, Z)) 
    + \sum_{\frac{1}{\log{Y}} < |\s + it - \rho_{F}| \leq \frac{1}{\log{Z}}}\log((\s + it - \rho_{F})\log{Z})\\
    &\qqqquad \qqquad - \sum_{|\s + it - \rho_{F}| \leq \frac{1}{\log{Y}}}\log\l( \frac{\log{Y}}{\log{Z}} \r).
  \end{align*}
  By Lemma \ref{MomentLambdab}, we have
  \begin{align*}
    \int_{T}^{2T}|P_{F}(\s + it, X) - P_{F}(\s + it, Z)|^{2k}dt
    \leq T C^{k} k^{k} \l( \sum_{\min\{ X, Z \} < p \leq \max\{ X, Z \}}\frac{|a_{F}(p)|^2}{p^{2\s}} \r)^{k}.
  \end{align*}
  for some constant $C = C(F) > 0$.
  If $Z < X \leq Y$, we see, using \eqref{SNC}, that the sum on the right hand side is $\ll_{F} Z^{1 - 2\s} = T^{\delta_{F}(1 - 2\s)/k}$. 
  If $X < Z$, it holds that the sum is $\leq \sum_{X < p \leq Y}\frac{|a_{F}(p)|^2}{p^{2\s}}$.
  Hence, we have
  \begin{align}\label{XToZ}
    \hspace{-5mm}
    \int_{T}^{2T}|P_{F}(\s + it, X) - P_{F}(\s + it, Z)|^{2k}dt
    \leq T C^{k} k^{k} T^{\delta_{F}(1 - 2\s)} + T C^{k} k^{k} \l( \sum_{X < p \leq Y}\frac{|a_{F}(p)|^2}{p^{2\s}} \r)^{k}.
  \end{align}
  Also, it holds that
  \begin{align*}
    \sum_{\frac{1}{\log{Y}} < |\s + it - \rho_{F}| \leq \frac{1}{\log{Z}}}\log((\s + it - \rho_{F})\log{Z})
    - \sum_{|\s + it - \rho_{F}| \leq \frac{1}{\log{Y}}}\log\l( \frac{\log{Y}}{\log{Z}} \r)
    \ll \sum_{|\s + it - \rho_{F}| \leq \frac{1}{\log{Z}}}1.
  \end{align*}
  Therefore, by this inequality and Lemma \ref{ESRSIZ}, we have
  \begin{align}
    &\int_{T}^{2T}\l\{ \sum_{\frac{1}{\log{Y}} < |\s + it - \rho_{F}| \leq \frac{1}{\log{Z}}}\log((\s + it - \rho_{F})\log{Z})
    + \sum_{|\s + it - \rho_{F}| \leq \frac{1}{\log{Y}}}\log\l( \frac{\log{Y}}{\log{Z}} \r) \r\}^{2k} dt\\
    &\leq T^{1 - \delta_{F}(2\s - 1)} C^{k} k^{2k}.\label{zerosYZ}
  \end{align}
  Combining \eqref{PKLI4}, \eqref{XToZ} and \eqref{zerosYZ}, we obtain
  \begin{align*}
    &\frac{1}{T}\int_{T}^{2T}\bigg|\log{F(\s + it)} - P_{F}(\s + it, X)
		- \sum_{|\s + it - \rho_{F}| \leq \frac{1}{\log{Y}}}\log((\s + it - \rho_{F})\log{Y})\bigg|^{2k}\\
    &\leq C^{k} k^{2k} T^{\delta_{F}(1 - 2\s)} + C^{k} k^{k} \l( \sum_{Z < p \leq Z^{2}}\frac{|a_{F}(p)|^2}{p^{2\s}} \r)^{k}
    + C^{k} k^{k} \l( \sum_{X < p \leq Y}\frac{|a_{F}(p)|^2}{p^{2\s}} \r)^{k}.
  \end{align*}
  Using \eqref{SNC}, we see that $\sum_{Z < p \leq Z^{2}}\frac{|a_{F}(p)|^2}{p^{2\s}} \ll_{F} Z^{1 - 2\s} = T^{\delta_{F}(1 - 2\s)/k}$.
  Therefore, the second term is absorbed into the first term. Thus, we complete the proof of Proposition \ref{KLI}.
\end{proof}

\begin{proof}[Proof of Proposition \ref{KLST}]
	By Proposition \ref{KLI}, it suffices to show that there exists a positive constant $A_{3} = A_{3}(F)$ such that
	\begin{align} \label{PKLST1}
		\int_{T}^{2T}\Bigg| \sum_{|\s + it - \rho_{F}| \leq \frac{1}{\log{Y}}}\log((\s+it - \rho_{F})\log{Y}) \Bigg|^{2k}
		\leq T^{1 - (2\s - 1)\delta_{F}} A_{3}^{k} k^{4k}
	\end{align}
	with $Y = T^{1 / k}$. 
  Put $Z = T^{\delta_{F} / k}$.
  We write as
  \begin{align*}
    &\sum_{|\s + it - \rho_{F}| \leq \frac{1}{\log{Y}}}\log((\s+it - \rho_{F})\log{Y})\\
    &= \sum_{|\s + it - \rho_{F}| \leq \frac{1}{\log{Z}}}\log((\s + it - \rho_{F})\log{Z})
    -\sum_{\frac{1}{\log{Y}} < |\s + it - \rho_{F}| \leq \frac{1}{\log{Z}}}\log((\s + it - \rho_{F})\log{Z})\\
    &\qqqquad \qqquad + \sum_{|\s + it - \rho_{F}| \leq \frac{1}{\log{Y}}}\log\l( \frac{\log{Y}}{\log{Z}} \r).
  \end{align*}
  It holds that
  \begin{align*}
    -\sum_{\frac{1}{\log{Y}} < |\s + it - \rho_{F}| \leq \frac{1}{\log{Z}}}\log((\s + it - \rho_{F})\log{Z})
    + \sum_{|\s + it - \rho_{F}| \leq \frac{1}{\log{Y}}}\log\l( \frac{\log{Y}}{\log{Z}} \r)
    \ll \sum_{|\s + it - \rho_{F}| \leq \frac{1}{\log{Z}}} 1.
  \end{align*}
  Hence, by using Lemma \ref{ESRSIZ}, we obtain \eqref{PKLST1}.
\end{proof}

\section{\textbf{Proofs of results for Dirichlet polynomials}}\label{dpoly}

In this section, we prove Propositions \ref{Main_Prop_JVD}, \ref{Main_Prop_JVD3}, Theorem \ref{GMDP} and Corollary \ref{jointDirichlet}.
Throughout this section, we assume that $\bm{F} = (F_1, \dots, F_{r})$ is an $r$-tuple of Dirichlet series,
and $\bm{\theta} = (\theta_{1}, \dots, \theta_{r}) \in \RR^{r}$.
We also assume that $\{ \mathcal{X}(p) \}_{p \in \mathcal{P}}$ is a sequence of independent random variables 
on a probability space $(\Omega, \mathscr{A}, \PP)$ with uniformly distributed on the unit circle in $\CC$.

\subsection{Approximate formulas for moment generating functions}


We will relate the moment generating function of $(P_{F_{j}}(\tfrac{1}{2}+it, X))_{j = 1}^{r}$
to the moment generating function of random $L$-series.
Recall that the Dirichlet polynomial $P_{F_{j}}(s, X)$ is defined by \eqref{def_D_P}.
To do this, we work with a subset of $[T,2T]$ such that the Dirichlet polynomials do not obtain large values. 
More precisely, define the set $\mathcal{A} = \mathcal{A}(T, X, \bm{F})$ by
\begin{align} \label{def_sA_JVDPP}
	\mathcal{A}
	= \bigcap_{j = 1}^{r}\set{t \in [T, 2T]}{\frac{|P_{F_{j}}(\tfrac{1}{2}+it, X)|}{\s_{F_{j}}(X)} \leq (\log{\log{X}})^{2(r + 1)}}.
\end{align}
We show that the measure of $\mathcal A$ is sufficiently close to $T$, 
and thus it enough to consider the moment generating function of $(P_{F_{j}}(\tfrac{1}{2}+it, X))_{j = 1}^{r}$ on $\mathcal A$.

We first show that the measure of $\mathcal A$ is close to $T$.
\begin{lemma} \label{ESAEG_JVDPP}
	Assume that $\bm{F}$ satisfies (S4), (S5), (A1).
	Let $T, X$ be large with $X^{(\log{\log{X}})^{4(r + 1)}} \leq T$.
	Then there exists a positive constant $b_{0} = b_{0}(\bm{F})$ such that
	\begin{align*}
		\frac{1}{T}\meas([T, 2T] \setminus \mathcal{A})
		\leq \exp\l( -b_{0}(\log{\log{X}})^{4(r + 1)} \r).
	\end{align*}
\end{lemma}

\begin{proof}
	By Lemmas \ref{MomentLambdab}, \ref{SumbLargeP}, there exists a constant $C_{j}$ depending only on $F_{j}$ such that
	\begin{align} \label{MVEDPCLPP}
		\int_{T}^{2T}| P_{F_{j}}(\tfrac{1}{2}+it, X)|^{2k}dt
		\leq T C_{j}^k k^k \l( \sum_{n\leq X}\frac{|b_{F_{j}}(n)|^2}{n} \r)^{k}
		\leq T (2C_{j} k \s_{F_{j}}(X)^2)^{k}
	\end{align}
	for $1 \leq k \leq 2^{-1}(\log{\log{X}})^{4(r + 1)}$.
	This implies that
	\begin{align*}
		\frac{1}{T}\meas\set{t \in [T, 2T]}{\frac{|P_{F_{j}}(\tfrac{1}{2}+it, X)|}{\s_{F_{j}}(X)} > (\log{\log{X}})^{2(r + 1)}}
		\ll \l( \frac{2C_{j} k}{(\log{\log{X}})^{4(r + 1)}} \r)^{k}.
	\end{align*}
	Hence, there exists some constant $C = C(\bm{F}) > 0$ such that
	\begin{align*}
		\frac{1}{T}\meas([T, 2T] \setminus \mathcal{A})
		\leq \l( \frac{2Ck}{(\log{\log{X}})^{4(r + 1)}} \r)^{k}.
	\end{align*}
	Choosing $k = \lfloor (2 e C)^{-1} (\log{\log{X}})^{4(r + 1)} \rfloor$, we obtain this lemma.
\end{proof}
Next we consider the moment generating function of $(P_{F_{j}}(\tfrac{1}{2}+it, X))_{j = 1}^{r}$ on $\mathcal A$.
\begin{proposition}	\label{RKLJVDPP}
	Assume that $\bm{F}$, $\bm{\theta}$ satisfy (S4), (S5), (A1), and (A2).
	Let $T$, $X$ be large numbers with $X^{(\log{\log{X}})^{4(r + 1)}} \leq T$.
	For any $\bm{z} = (z_1, \dots, z_r) \in \CC^{r}$ with $\norm[]{\bm{z}} \leq 2(\log{\log{X}})^{2r}$,
  \begin{align} \label{RKLJVDPP1}
		& \frac{1}{T}\int_{\mathcal{A}}\exp\l( \sum_{j = 1}^{r} z_{j} \Re e^{-i\theta_{j}}P_{F_{j}}(\tfrac{1}{2}+it, X) \r)dt\\
		&= \prod_{p \leq X}M_{p, \frac{1}{2}}(\bm{z})
		+ O\l(\exp\l( -b_{1}(\log{\log{X}})^{4(r + 1)} \r)\r),
	\end{align}
	where
	$M_{p, \frac{1}{2}}$ is the function defined by \eqref{def_RDP_MGF}. 
	Here, $b_{1}$ is a positive constant depending only on $\bm{F}$.
\end{proposition}

We prepare some lemmas for the proof of Proposition \ref{RKLJVDPP}.

\begin{lemma} \label{GRLR}
	Let $\{b_{1}(m)\}, \dots, \{ b_{n}(m) \}$ be complex sequences.
	For any $q_{1}, \dots, q_{s}$ distinct prime numbers and any $k_{1, 1}, \dots k_{1, n}, \dots,  k_{s, 1}, \dots, k_{s, n} \in \ZZ_{\geq 1}$,
	we have
	\begin{align*}
		 & \frac{1}{T}\int_{T}^{2T}\prod_{\ell = 1}^{s}
		\l(\Re b_{1}(q_{\ell}^{k_{1, \ell}}) q_{\ell}^{-i t k_{1, \ell}}\r)
		\cdots \l(\Re b_{n}(q_{\ell}^{k_{n, \ell}}) q_{\ell}^{-i t k_{n, \ell}}\r)dt                               \\
		 & =  \EXP{\prod_{\ell = 1}^{s}\l(\Re b_{1}(q_{\ell}^{k_{1, \ell}}) \mathcal{X}(q_{\ell})^{k_{1, \ell}}\r)
		\cdots \l(\Re b_{n}(q_{\ell}^{k_{n, \ell}}) \mathcal{X}(q_{\ell})^{k_{n, \ell}}\r)}
		+ O\l( \frac{1}{T}\prod_{\ell = 1}^{s}\prod_{j = 1}^{n}q_{\ell}^{k_{j, \ell}}|b_{j}(q_{\ell}^{k_{j, \ell}})| \r).
	\end{align*}
\end{lemma}

\begin{proof}
	Denote $\psi_{h, \ell} = \arg b_{h}(q_{\ell}^{k_{h, \ell}})$.
	Then we can write
	\begin{align}
		 & \l(\Re b_{1}(q_{\ell}^{k_{1, \ell}}) q_{\ell}^{-i t k_{1, \ell}}\r)
		\cdots (\Re b_{n}(q_{\ell}^{k_{n, \ell}}) q_{\ell}^{-i t k_{n, \ell}})                         \\
		 & = \frac{|b_{1}(q_{\ell}^{k_{1, \ell}}) \cdots b_{n}(q_{\ell}^{k_{n, \ell}})|}{2^{n}}
		\sum_{\e_{1}, \dots, \e_{n} \in \{ -1, 1 \}}e^{i(\e_{1}\psi_{1, \ell} + \cdots + \e_{n}\psi_{n, \ell})}
		q_{\ell}^{-i(\e_{1} k_{1, \ell} + \cdots \e_{n} k_{n, \ell})t}                                 \\
		\label{pGRLR1}
		 & = \frac{|b_{1}(q_{\ell}^{k_{1, \ell}}) \cdots b_{n}(q_{\ell}^{k_{n, \ell}})|}{2^{n}}
		\sum_{\substack{\e_{1}, \dots, \e_{n} \in \{ -1, 1 \}                                          \\ \e_{1}k_{1, \ell} + \cdots + \e_{n}k_{n, \ell} = 0}}
		e^{i(\e_{1}\psi_{1, \ell} + \cdots + \e_{n}\psi_{n, \ell})}                                    \\
		 & \qquad + \frac{|b_{1}(q_{\ell}^{k_{1, \ell}}) \cdots b_{n}(q_{\ell}^{k_{n, \ell}})|}{2^{n}}
		\sum_{\substack{\e_{1}, \dots, \e_{n} \in \{ -1, 1 \}                                          \\ \e_{1}k_{1, \ell} + \cdots + \e_{n}k_{n, \ell} \not= 0}}
		e^{i(\e_{1}\psi_{1, \ell} + \cdots + \e_{n}\psi_{n, \ell})}e^{-i(\e_{1} k_{1, \ell} + \cdots + \e_{n} k_{n, \ell})t \log{q_{\ell}}}.
	\end{align}
	Thus, we obtain
	\begin{align*}
		 & \prod_{\ell = 1}^{s}\l(\Re b_{1}(q_{\ell}^{k_{1, \ell}}) q_{\ell}^{-i t k_{1, \ell}}\r)
		\cdots (\Re b_{n}(q_{\ell}^{k_{n, \ell}}) q_{\ell}^{-i t k_{n, \ell}})                                      \\
		 & = \prod_{\ell = 1}^{s}\frac{|b_{1}(q_{\ell}^{k_{1, \ell}}) \cdots b_{n}(q_{\ell}^{k_{n, \ell}})|}{2^{n}}
		\sum_{\substack{\e_{1}, \dots, \e_{n} \in \{ -1, 1 \}                                                       \\ \e_{1}k_{1} + \cdots + \e_{n}k_{n} = 0}}
		e^{i(\e_{1}\psi_{1, \ell} + \cdots + \e_{n}\psi_{n, \ell})}
		+ E(t),
	\end{align*}
	where $E(t)$ is the sum whose the number of terms is less than $2^{n s}$,
	and the form of each term is $\delta'' e^{it(\b_{1}\log{q_{1}} + \cdots + \b_{s}\log{q_{s}})}$
	with $0 \leq |\b_{\ell}| \leq \a_{\ell} := k_{1, \ell} + \cdots k_{n, \ell}$ and $\b_{u} \not= 0$ for some $1 \leq u \leq s$.
	Here, $\delta''$ is a complex number independent of $t$, and satisfies
	$|\delta''| \leq  2^{-n s}\prod_{\ell = 1}^{s}\prod_{j = 1}^{n}|b_{j}(q_{\ell}^{k_{j, \ell}})|$.
	Since $|\b_{1}\log{q_{1}} + \cdots + \b_{s}\log{q_{s}}| \gg \prod_{\ell = 1}^{s}\prod_{j = 1}^{n}q_{\ell}^{-k_{j, \ell}}$,
	the integral of each term of $E(t)$ is bounded by
	$2^{-n s}\prod_{\ell = 1}^{s}\prod_{j = 1}^{n}q_{\ell}^{k_{j, \ell}}|b_{j}(q_{\ell}^{k_{j, \ell}})|$.
	Hence, by this bound of $E$ and the bound for the number of terms of $E$,
	we have
	\begin{align*}
		\int_{T}^{2T}E(t) dt
		\ll \prod_{\ell = 1}^{s}\prod_{j = 1}^{n}q_{\ell}^{k_{j, \ell}}|b_{j}(q_{\ell}^{k_{j, \ell}})|.
	\end{align*}
	Therefore, we have
	\begin{align*}
		 & \frac{1}{T}\int_{T}^{2T}\prod_{\ell = 1}^{s}\l(\Re b_{1}(q_{\ell}^{k_{1, \ell}}) q_{\ell}^{-i t k_{1, \ell}}\r)
		\cdots \l(\Re b_{n}(q_{\ell}^{k_{n, \ell}}) q_{\ell}^{-i t k_{n, \ell}}\r)dt                                       \\
		 & = \prod_{\ell = 1}^{s}\frac{|b_{1}(q_{\ell}^{k_{1, \ell}}) \cdots b_{n}(q_{\ell}^{k_{n, \ell}})|}{2^{n}}
		\sum_{\substack{\e_{1}, \dots, \e_{n} \in \{ -1, 1 \}                                                              \\ \e_{1}k_{1} + \cdots + \e_{n}k_{n} = 0}}
		e^{i(\e_{1}\psi_{1, \ell} + \cdots + \e_{n}\psi_{n, \ell})}
		+ O\l( \frac{1}{T}\prod_{\ell = 1}^{s}\prod_{j = 1}^{n}q_{\ell}^{k_{j, \ell}}|b_{j}(q_{\ell}^{k_{j, \ell}})| \r).
	\end{align*}
	On the other hand, by the independence of $\mathcal{X}(p)$'s, it follows that
	\begin{align*}
		 & \EXP{\prod_{\ell = 1}^{s}\l(\Re b_{1}(q_{\ell}^{k_{1, \ell}}) \mathcal{X}(q_{\ell})^{k_{1, \ell}}\r)
		\cdots \l(\Re b_{n}(q_{\ell}^{k_{n, \ell}}) \mathcal{X}(q_{\ell})^{k_{n, \ell}}\r)}                       \\
		 & = \prod_{\ell = 1}^{s}\EXP{\l(\Re b_{1}(q_{\ell}^{k_{1, \ell}}) \mathcal{X}(q_{\ell})^{k_{1, \ell}}\r)
		\cdots \l(\Re b_{n}(q_{\ell}^{k_{n, \ell}}) \mathcal{X}(q_{\ell})^{k_{n, \ell}}\r)}.
	\end{align*}
	As in \eqref{pGRLR1}, we can write
	\begin{align*}
		 & \l(\Re b_{1}(q_{\ell}^{k_{1, \ell}}) \mathcal{X}(q_{\ell})^{k_{1, \ell}}\r)
		\cdots \l(\Re b_{n}(q_{\ell}^{k_{n, \ell}}) \mathcal{X}(q_{\ell})^{k_{n, \ell}}\r)             \\
		 & = \frac{|b_{1}(q_{\ell}^{k_{1, \ell}}) \cdots b_{n}(q_{\ell}^{k_{n, \ell}})|}{2^{n}}
		\sum_{\substack{\e_{1}, \dots, \e_{n} \in \{ -1, 1 \}                                          \\ \e_{1}k_{1, \ell} + \cdots + \e_{n}k_{n, \ell} = 0}}
		e^{i(\e_{1}\psi_{1, \ell} + \cdots + \e_{n}\psi_{n, \ell})}                                    \\
		 & \qquad + \frac{|b_{1}(q_{\ell}^{k_{1, \ell}}) \cdots b_{n}(q_{\ell}^{k_{n, \ell}})|}{2^{n}}
		\sum_{\substack{\e_{1}, \dots, \e_{n} \in \{ -1, 1 \}                                          \\ \e_{1}k_{1, \ell} + \cdots + \e_{n}k_{n, \ell} \not= 0}}
		e^{i(\e_{1}\psi_{1, \ell} + \cdots + \e_{n}\psi_{n, \ell})}\mathcal{X}(q_{\ell})^{-i(\e_{1} k_{1, \ell} + \cdots + \e_{n} k_{n, \ell})}.
	\end{align*}
	Since $\mathcal{X}(p)$ is uniformly distributed on the unit circle in $\CC$, we have
	\begin{align} \label{BEQXp}
		\EXP{\mathcal{X}(p)^{a}}
		= \l\{
		\begin{array}{cc}
			1 & \text{if \; $a = 0$,} \\
			0 & \text{otherwise}
		\end{array}
		\r.
	\end{align}
	for any $a \in \ZZ$.
	Hence, we obtain
	\begin{align*}
		 & \prod_{\ell = 1}^{s}\EXP{\l(\Re b_{1}(q_{\ell}^{k_{1, \ell}}) \mathcal{X}(q_{\ell})^{k_{1, \ell}}\r)
		\cdots \l(\Re b_{n}(q_{\ell}^{k_{n, \ell}}) \mathcal{X}(q_{\ell})^{k_{n, \ell}}\r)}                         \\
		 & = \prod_{\ell = 1}^{s}\frac{|b_{1}(q_{\ell}^{k_{1, \ell}}) \cdots b_{n}(q_{\ell}^{k_{n, \ell}})|}{2^{n}}
		\sum_{\substack{\e_{1}, \dots, \e_{n} \in \{ -1, 1 \}                                                       \\ \e_{1}k_{1, \ell} + \cdots + \e_{n}k_{n, \ell} = 0}}
		e^{i(\e_{1}\psi_{1, \ell} + \cdots + \e_{n}\psi_{n, \ell})},
	\end{align*}
	which completes the proof of the lemma.
\end{proof}

\begin{lemma} \label{UBDRP}
	Let $\{a(p)\}_{p \in \mathcal{P}}$ be a complex sequence.
	Then, for any $k \in \ZZ_{\geq 1}$, $X \geq 3$, we have
	\begin{align*}
		\EXP{\bigg| \sum_{p \leq X}a(p)\mathcal{X}(p) \bigg|^{2k}}
		\leq k! \l( \sum_{p \leq X}|a(p)|^2 \r)^{k}.
	\end{align*}
\end{lemma}

\begin{proof}
	Since $\mathcal{X}(p)$'s are independent and uniformly distributed on the unit circle in $\CC$, it holds that
	\begin{align}	\label{BEUIR}
		\EXP{ \frac{\mathcal{X}(p_1)^{a_{1}} \cdots \mathcal{X}(p_{k})^{a_{k}}}{ \mathcal{X}(q_1)^{b_1} \cdots \mathcal{X}(q_{\ell})^{b_{\ell}} } }
		=
		\begin{cases}
			1 & \text{if $p_{1}^{a_1} \cdots p_{k}^{a_{k}} = q_{1}^{b_1} \cdots q_{\ell}^{b_\ell}$},    \\
			0 & \text{if $p_{1}^{a_1} \cdots p_{k}^{a_{k}} \neq q_{1}^{b_1} \cdots q_{\ell}^{b_\ell}$}.
		\end{cases}
	\end{align}
	It follows that
	\begin{align*}
		\EXP{ \bigg| \sum_{p \leq X}a(p) \mathcal{X}(p) \bigg|^{2k} }
		 & = \sum_{\substack{p_1, \dots, p_{k} \leq X \\ q_{1}, \dots, q_{k} \leq X}}
		a(p_1) \cdots a(p_{k}) \ol{a(q_{1}) \cdots a(q_{k})}
		\EXP{\frac{\mathcal{X}(p_1) \cdots \mathcal{X}(p_{k}) }{ \mathcal{X}(q_{1}) \cdots \mathcal{X}(q_{k})} } \\
		 & \leq k!\sum_{p_{1}, \dots, p_{k} \leq X}|a(p_{1})|^2 \cdots |a(p_{k})|^{2}
		\leq k! \l( \sum_{p \leq X}|a(p)|^2 \r)^{k},
	\end{align*}
	which completes the proof of the lemma.
\end{proof}

\begin{lemma} \label{UBMVRDP}
	Assume that $F$ is a Dirichlet series satisfying (S4), (S5), and \eqref{SNC}.
  Let $P_{F}(\s, \mathcal{X}, X)$ be the random Dirichlet polynomial defined by \eqref{def_RP}.
	There exists a positive constant $C = C(F)$ such that for any $k \in \ZZ_{\geq 1}$, and any large $X$, we have
	\begin{align*}
		\EXP{| P_{F}(\tfrac{1}{2}, \mathcal{X}, X)|^{2k}}
		\ll (C k \s_{F}(X)^{2})^{k}.
	\end{align*}
\end{lemma}

\begin{proof}
	Using Lemma \ref{UBDRP}, we can prove this lemma in the same way as Lemma \ref{MomentLambdab}.
\end{proof}

\begin{proof}[Proof of Proposition \ref{RKLJVDPP}] 
	Let $T$, $X$ be large numbers such that $X^{(\log{\log{X}})^{4(r+1)}} \leq T$.
	Let $\bm{z} = (z_{1}, \dots, z_{r}) \in \CC^{r}$ with $\norm[]{\bm{z}} \leq 2(\log{\log{X}})^{2r}$.
	From \eqref{def_sA_JVDPP}, we have
	\begin{align*}
		 & \frac{1}{T}\int_{\mathcal{A}}\exp\l( \sum_{j = 1}^{r} z_{j} \Re e^{-i\theta_{j}}P_{F_{j}}(\tfrac{1}{2}+it, X) \r)dt \\
		 & = \frac{1}{T}\sum_{0 \leq k \leq Y}\frac{1}{k!}\int_{\mathcal{A}}
		\bigg(\sum_{j = 1}^{r} z_{j} \Re e^{-i\theta_{j}}P_{F_{j}}(\tfrac{1}{2}+it, X) \bigg)^{k}dt                            \\
		 & \qqqquad \qqquad+ O\l(\sum_{k > Y}\frac{1}{k!}
		\l( C (\log{\log{X}})^{2r + 5/2} \norm[]{\bm{z}} \r)^{k} \r)
	\end{align*}
	with $Y = \frac{1}{4}(\log{\log{X}})^{4(r+1)}$.
	Here, $C = C(\bm{F})$ is some positive constant.
	We see that this $O$-term is $\ll \exp\l( - (\log{\log{X}})^{4(r+1)} \r)$ by the bound for $\|\bm z\|$.
	Using the Cauchy-Schwarz inequality, we find that
	\begin{align*}
		 & \frac{1}{T}\int_{\mathcal{A}}\bigg(\sum_{j = 1}^{r} z_{j} \Re e^{-i\theta_{j}}P_{F_{j}}(\tfrac{1}{2}+it, X) \bigg)^{k}dt \\
		 & = \frac{1}{T}\int_{T}^{2T}\bigg(\sum_{j = 1}^{r} z_{j} \Re e^{-i\theta_{j}}P_{F_{j}}(\tfrac{1}{2}+it, X) \bigg)^{k}dt    \\
		 & +O\l( \frac{1}{T}(\meas([T, 2T] \setminus \mathcal{A}))^{1/2}
		\l(\int_{T}^{2T} \bigg|\sum_{j = 1}^{r} z_{j} \Re e^{-i\theta_{j}}P_{F_{j}}(\tfrac{1}{2}+it, X) \bigg|^{2k}dt\r)^{1/2} \r).
	\end{align*}
	By Lemma \ref{ESAEG_JVDPP}, estimate \eqref{MVEDPCLPP}, and the bound for $\|\bm{z}\|$, this $O$-term is
	\begin{align*}
		\ll \exp\l( -(2 b_{0})^{-1} (\log{\log{X}})^{4(r+1)} \r) \l( C_{1} k^{1/2}(\log{\log{X}})^{2r + 1/2}  \r)^{k}
	\end{align*}
	for $0 \leq k \leq Y$, where $C_{1} = C_{1}(\bm{F}) > 0$ is a positive constant.
	It holds from this estimate and the Stirling formula that
	\begin{align}	\label{GRKLG5}
		 & \frac{1}{T}\int_{\mathcal{A}}\exp\l( \sum_{j = 1}^{r} z_{j} \Re e^{-i\theta_{j}}P_{F_{j}}(\tfrac{1}{2}+it, X) \r)dt \\
		 & = \frac{1}{T}\sum_{0 \leq k \leq Y}\frac{1}{k!}
		\int_{T}^{2T}\bigg(\sum_{j = 1}^{r} z_{j} \Re e^{-i\theta_{j}}P_{F_{j}}(\tfrac{1}{2}+it, X) \bigg)^{k}dt               \\
		 & \qqquad+ O\l(\exp\l( -(2 b_{0})^{-1} (\log{\log{X}})^{4(r + 1)} \r)
		\sum_{0 \leq k \leq Y}\frac{(C_{1} e  (\log{\log{X}})^{2r+1/2})^{k}}{k^{k/2}}\r).
	\end{align}
	When $X$ is sufficiently large, it follows that
	\begin{align*}
		\sum_{0 \leq k \leq Y}\frac{(C_{1} e  (\log{\log{X}})^{2r+1/2})^{k}}{k^{k/2}}
		 & = \sum_{0 \leq k \leq (\log{\log{X}})^{4r + 2}}\frac{(C_{1} e  (\log{\log{X}})^{2r+1/2})^{k}}{k^{k/2}} + O(1) \\
		 & \ll \exp\l( (\log{\log{X}})^{4r + 3} \r).
	\end{align*}
	Hence, the $O$-term on the right hand side of \eqref{GRKLG5} is $\ll \exp\l( -(3c_{1})^{-1}(\log{\log{X}})^{4(r + 1)} \r)$.
	
	Now, we can write
	\begin{align*}
		 & \int_{T}^{2T}\bigg(\sum_{j = 1}^{r} z_{j} \Re e^{-i\theta_{j}}P_{F_{j}}(\tfrac{1}{2}+it, X) \bigg)^{k}dt      \\
		 & = \sum_{1 \leq j_{1}, \dots, j_{k} \leq r}z_{j_{1}} \cdots z_{j_{k}}
		\sum_{p_{1}^{\ell_{1}}, \dots, p_{k}^{\ell_{k}} \leq X}\frac{1}{p_{1}^{\ell_{1}/2} \cdots p_{k}^{\ell_{k}/2}}    \\
		 & \qqquad \times \int_{T}^{2T}\l(\Re e^{-i\theta_{j_{1}}} b_{F_{j_{1}}}(p^{\ell_{1}}) p_{1}^{-i t \ell_{1}} \r)
		\cdots \l(\Re e^{-i\theta_{j_{k}}} b_{F_{j_{k}}}(p^{\ell_{k}}) p_{k}^{-i t \ell_{k}} \r)dt.
	\end{align*}
	From this equation and Lemma \ref{GRLR}, we have
	\begin{align*}
		 & \frac{1}{T}\int_{T}^{2T}\bigg(\sum_{j = 1}^{r} z_{j} \Re e^{-i\theta_{j}}P_{F_{j}}(\tfrac{1}{2}+it, X) \bigg)^{k}dt \\
		 &= \EXP{\bigg(\sum_{j = 1}^{r} z_{j} \Re e^{-i\theta_{j}}P_{F_{j}}(\tfrac{1}{2}, \mathcal{X}, X) \bigg)^{k}}         \\
		  &\quad+O\l( \frac{1}{T}\sum_{1 \leq j_{1}, \dots, j_{k} \leq r}|z_{j_{1}} \cdots z_{j_{k}}|
		\sum_{p_{1}^{\ell_{1}}, \dots, p_{k}^{\ell_{k}} \leq X}|b_{F_{j_{1}}}(p_{1}^{\ell_{1}}) \cdots b_{F_{j_{k}}}(p_{k}^{\ell_{k}})|
		p_{1}^{\ell_{1}/2} \cdots p_{k}^{\ell_{k}/2} \r).
	\end{align*}
	Additionally, we see that this $O$-term is
	\begin{align*}
		\ll \frac{1}{T}\l( \sum_{j = 1}^{r}|z_{j}|\sum_{n \leq X}|b_{F_{j}}(n)|n^{1/2} \r)^{k}
		\leq \frac{(C X)^{2k}}{T}
		\leq \frac{C^{2k}}{T^{1/2}}
		\leq \exp\l( -(3b_{0})^{-1} (\log{\log{X}})^{4(r + 1)} \r)
	\end{align*}
	for $0 \leq k \leq Y$ when $T$ is sufficiently large.
	Therefore, we have
	\begin{align*}
		 & \frac{1}{T}\int_{\mathcal{A}}\exp\l( \sum_{j = 1}^{r} z_{j} \Re e^{-i\theta_{j}}P_{F_{j}}(\tfrac{1}{2}+it, X) \r)dt                     \\
		 & = \sum_{0 \leq k \leq Y}\frac{1}{k!}\EXP{\l( \sum_{j = 1}^{r}z_{j}\Re e^{-i\theta_{j}} P_{F_{j}}(\tfrac{1}{2}, \mathcal{X}, X) \r)^{k}}
		+ O\l( \exp\l( -(3c_{1})^{-1}(\log{\log{X}})^{4(r + 1)} \r) \r)                                                                            \\
		 & = \EXP{\exp\l( \sum_{j = 1}^{r}z_{j} \Re e^{-i\theta_{j}}P_{F_{j}}(\tfrac{1}{2}, \mathcal{X}, X) \r)}
		- \sum_{k > Y}\frac{1}{k!}\EXP{\l( \sum_{j = 1}^{r}z_{j}\Re e^{-i\theta_{j}} P_{F_{j}}(\tfrac{1}{2}, \mathcal{X}, X) \r)^{k}}              \\
		 & \qqqquad \qqqquad + O\l( \exp\l( -(3b_{0})^{-1}(\log{\log{X}})^{4(r + 1)} \r) \r).
	\end{align*}
	The independence of $\mathcal{X}(p)$'s yields that
	\begin{align*}
		\EXP{\exp\l( \sum_{j = 1}^{r}z_{j} \Re e^{-i\theta_{j}}P_{F_{j}}(\tfrac{1}{2}, \mathcal{X}, X) \r)}
		= \prod_{p \leq X}M_{p, \frac{1}{2}}(\bm{z}).
	\end{align*}
	Using Lemmas \ref{SumbLargeP}, \ref{UBMVRDP}, the Cauchy-Schwarz inequality, and the bound for $\| \bm{z} \|$, we obtain
	\begin{align}
		\EXP{\l( \sum_{j = 1}^{r}z_{j}\Re e^{-i\theta_{j}} P_{F_{j}}(\tfrac{1}{2}, \mathcal{X}, X) \r)^{k}}
		\leq k^{k/2}\l(C (\log{\log{X}})^{2r + 1/2}\r)^{k}
	\end{align}
	for some constant $C = C(\bm{F}) > 0$.
	Therefore, it holds that
	\begin{align*}
		\sum_{k > Y}\frac{1}{k!}\EXP{\l( \sum_{j = 1}^{r}z_{j}\Re e^{-i\theta_{j}} P_{F_{j}}(\tfrac{1}{2}, \mathcal{X}, X) \r)^{k}}
		 & \leq \sum_{k > Y}\frac{\l(C (\log{\log{X}})^{2r + 1/2}\r)^{k}}{k^{k/2}}
		\leq \sum_{k > Y}\l( \frac{2C}{(\log{\log{X}})^{1/2}} \r)^{k}              \\
		 & \leq \exp\l( -(3b_{0})^{-1}(\log{\log{X}})^{4(r + 1)} \r).
	\end{align*}
	Thus, we complete the proof of Proposition \ref{RKLJVDPP}.
\end{proof}

\subsection{Estimates for the main term of the moment generating function.}
We give some lemmas to estimate the main term in Proposition \ref{RKLJVDPP}.

\begin{lemma} \label{EQTM}
	Assume that $\bm{F}$, $\bm{\theta}$ satisfy (S4), (S5).
	For any $\bm{z} = (z_{1}, \dots, z_{r}) \in \CC^{r}$, $\s \geq 1/2$, we have
	\begin{align*}
		M_{p, \s}(\bm{z})
		 & = 1 + \frac{1}{4}\sum_{1 \leq j_{1}, j_{2} \leq r}z_{j_{1}} z_{j_{2}}\sum_{\ell = 1}^{\infty}
		\frac{\Re e^{-i\theta_{j_{1}}}b_{F_{j_{1}}}(p^{\ell})\ol{e^{-i\theta_{j_{2}}}b_{F_{j_{2}}}(p^{\ell})}}{p^{2 \ell \s}} \\
		 & \qqqquad
		+ O\l( \sum_{n = 3}^{\infty}\frac{1}{n!}\l(\sum_{j = 1}^{r} |z_{j}|
		\sum_{\ell = 1}^{\infty}\frac{|b_{F_{j}}(p^{\ell})|}{p^{\ell \s}} \r)^{n} \r).
	\end{align*}
\end{lemma}

\begin{proof}
	We write 
	\begin{align*}
		 & \exp\l( \sum_{j = 1}^{r}z_{j} \Re e^{-i\theta_{j}}\sum_{\ell = 1}^{\infty}\frac{b_{F_{j}}(p^{\ell})\mathcal{X}(p)^{\ell}}{p^{\ell \s}} \r) \\
		 & =1 + \sum_{n = 1}^{\infty}\frac{1}{n!}\sum_{1 \leq j_{1}, \dots, j_{n} \leq r}z_{j_{1}} \cdots z_{j_{n}}
		\sum_{\ell_{1}, \dots, \ell_{n} = 1}^{\infty}
		\frac{\Re e^{-i\theta_{j_{1}}} b_{F_{j_{1}}}(p^{\ell_{1}}) \mathcal{X}(p)^{\ell_{1}}}{p^{\ell_{1} \s}}
		\cdots
		\frac{\Re e^{-i\theta_{j_{n}}} b_{F_{j_{n}}}(p^{\ell_{n}}) \mathcal{X}(p)^{\ell_{n}}}{p^{\ell_{n} \s}}.
	\end{align*}
	Let $\psi_{h} = \arg b_{F_{h}}(p^{\ell_{h}}) - \theta_{j_{h}}$.
	Then, using equation \eqref{BEQXp}, we obtain
	\begin{align*}
		 & \EXP{\exp\l( \sum_{j = 1}^{r}z_{j} \Re e^{-i\theta_{j}}
		\sum_{\ell = 1}^{\infty}\frac{b_{F_{j}}(p^{\ell})\mathcal{X}(p)^{\ell}}{p^{\ell \s}} \r)} =                      \\
		 & 1 + \sum_{n = 1}^{\infty}\frac{1}{2^{n} n!}\sum_{1 \leq j_{1}, \dots, j_{n} \leq r}z_{j_{1}} \cdots z_{j_{n}}
		\sum_{\ell_{1}, \dots, \ell_{n} = 1}^{\infty}\frac{|b_{F_{j_{1}}}(p^{\ell_{1}}) \cdots b_{F_{j_{n}}}(p^{\ell_{n}})|}
		{p^{(\ell_{1} + \cdots + \ell_{n})\s}}
		\sum_{\substack{\e_{1}, \dots, \e_{n} \in \{ -1, 1 \}                                                            \\ \e_{1}\ell_{1} + \cdots + \e_{n}\ell_{n} = 0}}
		e^{i(\e_{1} \psi_{1} + \cdots + \e_{n} \psi_{n})}.
	\end{align*}
	If $n = 1$, then the equation $\e_{1}\ell_{1} = 0$ is always false, so the term corresponds to $n = 1$ vanishes.
	The term becomes corresponds to $n=2$ is 
	\begin{align*}
		 & \frac{1}{4}\sum_{1 \leq j_{1}, j_{2} \leq r}z_{j_{1}} z_{j_{2}}\sum_{\ell = 1}^{\infty}
		\frac{\Re e^{-i\theta_{j_{1}}}b_{F_{j_{1}}}(p^{\ell})\ol{e^{-i\theta_{j_{2}}}b_{F_{j_{2}}}(p^{\ell})}}{p^{2 \ell \s}}.
	\end{align*}
	The term corresponds to $n\geq 3$ can be bounded by
	\begin{align*}
		 & \leq \frac{1}{n!}\sum_{1 \leq j_{1}, \dots, j_{n} \leq r}|z_{j_{1}} \cdots z_{j_{n}}|
		\sum_{\ell_{1}, \dots, \ell_{n} = 1}^{\infty}\frac{|b_{F_{j_{1}}}(p^{\ell_{1}}) \cdots b_{F_{j_{n}}}(p^{\ell_{n}})|}
		{p^{(\ell_{1} + \cdots + \ell_{n})\s}}
		= \frac{1}{n!}\l( \sum_{j = 1}^{r} |z_{j}| \sum_{\ell = 1}^{\infty}\frac{|b_{F_{j}}(p^{\ell})|}{p^{\ell \s}} \r)^{n}.
	\end{align*}
	This completes the proof of this lemma.
\end{proof}

\begin{lemma}	\label{Prop_Psi_FPP}
	Assume that $\bm{F}$ satisfies (S4), (S5), (A1), and (A2).
	Put
	\begin{align} \label{def_Psi_F}
		\Psi(\bm{z})
		= \Psi(\bm{z}; \bm{F}, \bm{\theta})
		:= \prod_{p}\frac{M_{p, \frac{1}{2}}(\bm{z})}
		{\exp\l( K_{\bm{F}, \bm{\theta}}(p, \bm{z}) / 4 \r)}.
	\end{align}
	Then, the infinity product is uniformly convergent on any compact set in $\CC^{r}$.
  In particular, $\Psi$ is analytic on $\CC^{r}$.
  Moreover, it holds that
	\begin{align} \label{Psi_F_1}
		| \Psi(\bm{z})|
		\leq \bigg|\prod_{j = 1}^{r}\exp\l( -\frac{z_{j}^2}{2}\s_{F_{j}}\l( |z_{j}| \r)^2
		+ O_{\bm{F}}\l( |z_{j}|^{2} + |z_{j}|^{\frac{2 - 2\vartheta_{\bm{F}}}{1 - 2\vartheta_{\bm{F}}}} \r) \r)\bigg|
	\end{align}
	for any $\bm{z} = (z_{1}, \dots, z_{r}) \in \CC$, and that
	\begin{align} \label{Psi_F_3}
		\Psi(\bm{x})
		= \prod_{j = 1}^{r}\exp\l( -\frac{x_{j}^2}{2}\s_{F_{j}}\l( |x_{j}| \r)^2
		+ O_{\bm{F}}\l( x_{j}^{2} + |x_{j}|^{\frac{2 - 2\vartheta_{\bm{F}}}{1 - 2\vartheta_{\bm{F}}}} \r) \r)
	\end{align}
	for $\bm{x} \in \RR^{r}$.
	Furthermore, for any $\bm{z} = (x_{1} + iu_{1}, \dots, x_{r} + iu_{r}) \in \CC^{r}$
	satisfying $x_{j}, u_{j} \in \RR$ with $\| \bm{u} \| \leq 1$,
	we have
	\begin{align} \label{Psi_F_2}
		\Psi(\bm{z})
		=\Psi(x_{1}, \dots, x_{r})\prod_{j = 1}^{r}\l(1
		+ O_{\bm{F}}\l(|u_{j}| \exp\l( D_{1} \|\bm{x}\|^{\frac{2 - 2\vartheta_{\bm{F}}}{1 - 2\vartheta_{\bm{F}}}} \r) \r)\r).
	\end{align}
	Here, $D_{1} = D_{1}(\bm{F})$ is a positive constant.
\end{lemma}

\begin{proof}
	By the definition of $K_{\bm{F}, \bm{\theta}}$ (see \eqref{def_K_F}), it holds that 
	\begin{align}	\label{ES_sInp2}
		|K_{\bm{F}, \bm{\theta}}(p, \bm{z})|
		\leq 2 r \| \bm{z} \|^{2} \sum_{\ell = 1}^{\infty}\frac{\|b_{\bm{F}}(p^{\ell})\|^2}{p^{\ell}}.
	\end{align}
	Here, $b_{\bm{F}}(p^{\ell})$ indicates the vector $(b_{F_{1}}(p^{\ell}), \dots, b_{F_{r}}(p^{\ell}))$.
	Using this estimate and Lemma \ref{EQTM},
	we find that
	\begin{align}
		 & \frac{M_{p, \frac{1}{2}}(\bm{z})}{\exp(K_{\bm{F}, \bm{\theta}}(p, \bm{z}) / 4)}         \\
		 & = \l\{ 1 + \frac{1}{4}K_{\bm{F}, \bm{\theta}}(p, \bm{z})
		+ O\l( \sum_{n = 3}^{\infty}\frac{1}{n!}\l(r \| \bm{z} \|
		\sum_{\ell = 1}^{\infty}\frac{\| b_{\bm{F}}(p^{\ell}) \|}{p^{\ell/2}} \r)^{n} \r) \r\}      \\
		 & \qqqquad \times \l\{ 1 - \frac{1}{4}K_{\bm{F}, \bm{\theta}}(p, \bm{z})
		+ O\l( \sum_{n = 3}^{\infty}\frac{1}{n!}\l(\frac{r}{2} \| \bm{z} \|^2
		\sum_{\ell = 1}^{\infty}\frac{\| b_{\bm{F}}(p^{\ell}) \|^2}{p^{\ell}} \r)^{n} \r) \r\} \\
		\label{ES_sI_new_p1}
		 & = 1 + O\l( \sum_{n = 3}^{\infty}\frac{(r \| \bm{z} \|)^{n}}{n!}
		\l( \sum_{\ell = 1}^{\infty}\frac{\| b_{\bm{F}}(p^{\ell}) \|}{p^{\ell/2}} \r)^{n}
		+ \sum_{n = 2}^{\infty}\frac{(r \|\bm{z}\|^2)^{n}}{n!}
		\l(\sum_{\ell = 1}^{\infty}\frac{\| b_{\bm{F}}(p^{\ell}) \|^2}{p^{\ell}} \r)^{n} \r). 
	\end{align}
	Let $Y \geq C_{1} (\| \bm{z} \|)^{\frac{2}{1 - 2\vartheta_{\bm{F}}}}$ with $C_{1} = C_{1}(\bm{F}) > 0$ a suitably large constant.
	Then, by estimate \eqref{ES_sI_new_p1} and the definition of $\vartheta_{\bm{F}}$, it holds that
	\begin{align*}
		\bigg|\frac{M_{p, \frac{1}{2}}(\bm{z})}{\exp(K_{\bm{F}, \bm{\theta}}(p, \bm{z}) / 4)} - 1\bigg|
		\leq \frac{1}{2}
	\end{align*}
	for any $p \geq Y$.
	From this inequality and estimate \eqref{ES_sI_new_p1}, we obtain
	\begin{align}
		 & \frac{M_{p, \frac{1}{2}}(\bm{z})}{\exp(K_{\bm{F}, \bm{\theta}}(p, \bm{z}) / 4)}
		= \exp\l(\log \frac{M_{p, \frac{1}{2}}(\bm{z})}{\exp(K_{\bm{F}, \bm{\theta}}(p, \bm{z}) / 4)}\r) \\
		\label{ES_sI_PP_p1}
		 & = \exp\l\{ O\l( \sum_{n = 3}^{\infty}\frac{(r \| \bm{z} \|)^{n}}{n!}
		\l( \sum_{\ell = 1}^{\infty}\frac{\| b_{\bm{F}}(p^{\ell}) \|}{p^{\ell/2}} \r)^{n}
		+ \sum_{n = 2}^{\infty}\frac{(r \|\bm{z}\|^2)^{n}}{n!}
		\l(\sum_{\ell = 1}^{\infty}\frac{\| b_{\bm{F}}(p^{\ell}) \|^2}{p^{\ell}} \r)^{n} \r) \r\}
	\end{align}
	for any $p \geq Y$.
	Set $K_{1} = 2(\frac{1}{2} - \vartheta_{\bm{F}})^{-1}$.
	Then, it holds that
	$$
		\sum_{\ell > K_{1}}\frac{\| b_{\bm{F}}(p^{\ell}) \|}{p^{\ell/2}}
		\ll_{\bm{F}} p^{-2},\text{ and }
		\sum_{\ell > K_{1}}\frac{\| b_{\bm{F}}(p^{\ell}) \|^2}{p^{\ell}}
		\ll_{\bm{F}} p^{-4}.
	$$
	Therefore, there exists some $C = C(F)>0$ such that
	\begin{align} \label{ES_sI_PP_p5}
		\l( \sum_{\ell = 1}^{\infty}\frac{\| b_{\bm{F}}(p^{\ell}) \|}{p^{\ell/2}} \r)^{n}
		 & \leq C^{n}\l(\max_{1 \leq \ell \leq K_{1}} \frac{\| b_{\bm{F}}(p^{\ell}) \|^{n}}{p^{\ell n/2}} + p^{-2n} \r)
		\leq C^{n}\l(\sum_{1 \leq \ell \leq K_{1}} \frac{\| b_{\bm{F}}(p^{\ell}) \|^{n}}{p^{\ell n/2}} + p^{-2n} \r)                    \\
		 & \leq C^{n}\l(\sum_{j = 1}^{r} \sum_{1 \leq \ell \leq K_{1}} \frac{\| b_{F_{j}}(p^{\ell}) \|^{n}}{p^{\ell n/2}} + p^{-2n} \r)
	\end{align}
	and similarly that
	\begin{align} \label{ES_sI_PP_p51}
		\l(\sum_{\ell = 1}^{\infty}\frac{\| b_{\bm{F}}(p^{\ell}) \|^2}{p^{\ell}} \r)^{n}
		\leq C^{n}\l(\sum_{j = 1}^{r} \sum_{1 \leq \ell \leq K_{1}} \frac{\| b_{F_{j}}(p^{\ell}) \|^{2 n}}{p^{\ell n}} + p^{-4n} \r).
	\end{align}
	Using these inequalities and Lemma \ref{SumbLargeP}, we have for any $Z \geq 2, n\in \mathbb{Z}$
	\begin{align} \label{ES_sI_PP_p3}
		\sum_{p > Z}\l( \sum_{\ell = 1}^{\infty}\frac{\| b_{\bm{F}}(p^{\ell}) \|}{p^{\ell/2}} \r)^{n}
		 & \leq C^{n}\sum_{j = 1}^{r}\sum_{m > Z}\frac{|b_{F_{j}}(m)|^2}{m^{1 + (1/2 - \vartheta_{\bm{F}})(n - 2)}} + C^{n} Z^{1 - 2n} \\
		 & \ll_{\bm{F}} C^{n} Z^{-(\frac{1}{2} - \vartheta_{\bm{F}})(n-2)}, \quad n \geq 3
	\end{align}
	and
	\begin{align} \label{ES_sI_PP_p4}
		\sum_{p > Z} \l(\sum_{\ell = 1}^{\infty}\frac{\| b_{\bm{F}}(p^{\ell}) \|^2}{p^{\ell}} \r)^{n}
		 & \leq C^{n}\sum_{j = 1}^{r}\sum_{m > Z}\frac{|b_{F_{j}}(m)|^2}{m^{1 + (1 - 2\vartheta_{\bm{F}})(n - 1)}} + C^{n} Z^{1 - 4n} \\
		 & \ll_{\bm{F}} C^{n} Z^{-(1 - 2\vartheta_{\bm{F}})(n-1)}, \quad n \geq 2.
	\end{align}
	Applying these estimates to \eqref{ES_sI_PP_p1}, we obtain
	\begin{align}
		\prod_{p > Y}\frac{M_{p, \frac{1}{2}}(\bm{z})}{\exp(K_{\bm{F}, \bm{\theta}}(p, \bm{z}) / 4)}
		 & = \exp\l( O_{\bm{F}}\l( \frac{\| \bm{z} \|^3}{Y^{\frac{1}{2} - \vartheta_{\bm{F}}}} \r) \r)             \\
		\label{ES_sI_PP_p2}
		 & = \prod_{j = 1}^{r}\exp\l( O_{\bm{F}}\l( \frac{|z_{j}|^3}{Y^{\frac{1}{2} - \vartheta_{\bm{F}}}} \r) \r)
	\end{align}
	for any $Y \geq C_{1} (\| \bm{z} \|)^{\frac{2}{1 - 2\vartheta_{\bm{F}}}}$ with $C_{1} = C_{1}(\bm{F})$ a suitably large constant.
	In particular, we see that this estimate implies the infinite product
	$
		\prod_{p}\frac{M_{p, \frac{1}{2}}(\bm{z})}{\exp(K_{\bm{F}, \bm{\theta}}(p, \bm{z}) / 4)}
	$
	is uniformly convergent for $\bm{z} \in D$ with $D$ an arbitrary compact set in $\CC^{r}$.
	Hence, $\Psi$ is analytic on $\CC^{r}$.

	Next, we prove \eqref{Psi_F_1}.
	Put $L = C_{1} (\| \bm{z} \|)^{\frac{2}{1 - 2\vartheta_{\bm{F}}}}$.
	Then we divide the range of the product as
	\begin{align}	\label{p_Psi_F_1}
		\prod_{p}\frac{M_{p, \frac{1}{2}}(\bm{z})}{\exp\l( K_{\bm{F}, \bm{\theta}}(p, \bm{z}) / 4 \r)}
		= \l(\prod_{p \leq L} \times \prod_{p > L}\r)
		\frac{M_{p, \frac{1}{2}}(\bm{z})}{\exp\l( K_{\bm{F}, \bm{\theta}}(p, \bm{z}) / 4 \r)}.
	\end{align}
	If $L < 2$, then
	$
		\prod_{p \leq L}M_{p, \frac{1}{2}}(\bm{z})
		= 1
		= \exp\l( O\l( \| \bm{z} \|^{\frac{2 - 2\vartheta_{\bm{F}}}{1 - 2\vartheta_{\bm{F}}}} \r) \r).
	$
	Next, we suppose $L \geq 2$.
	Using Lemma \ref{EQTM}, we see that
	\begin{align*}
		|M_{p, \frac{1}{2}}(\bm{z})|
		\leq \exp\l( \sum_{j = 1}^{r} |z_{j}| \sum_{\ell = 1}^{\infty}\frac{|b_{F_{j}}(p^{\ell})|}{p^{\ell/2}} \r)
		\leq \exp\l( r \| \bm{z} \| \sum_{\ell = 1}^{\infty}\frac{\| b_{\bm{F}}(p^{\ell}) \|}{p^{\ell/2}} \r).
	\end{align*}
	By \eqref{ES_sI_PP_p5} and the Cauchy-Schwarz inequality, we also find that
	\begin{align*}
		\sum_{p \leq L}\sum_{\ell = 1}^{\infty}\frac{\|b_{\bm{F}}(p^{\ell})\|}{p^{\ell/2}}
		\ll_{\bm{F}} \max_{1 \leq \ell \leq K_{1}}\sum_{p \leq L}\frac{\| b_{\bm{F}}(p^{\ell}) \|}{p^{\ell/2}} + 1
		\leq \max_{1 \leq \ell \leq K_{1}}\l(\sum_{p \leq L}\frac{\| b_{\bm{F}}(p^{\ell}) \|^2}{p^{\ell}}\r)^{1/2}\l( \sum_{p \leq L} 1 \r)^{1/2} + 1.
	\end{align*}
	Assumptions (S5), (A1), and the prime number theorem yield that
	\begin{align*}
		\max_{1 \leq \ell \leq K_{1}}\l(\sum_{p \leq L}\frac{\| b_{\bm{F}}(p^{\ell}) \|^2}{p^{\ell}}\r)^{1/2}\l( \sum_{p \leq L} 1 \r)^{1/2}
		\ll_{\bm{F}} L^{1/2}
		\ll_{\bm{F}} \| \bm{z} \|^{\frac{1}{1 - 2\vartheta_{\bm{F}}}}.
	\end{align*}
	Therefore, we obtain
	\begin{align*}
		\prod_{p \leq L}M_{p, \frac{1}{2}}(\bm{z})
		\leq \exp\l( O_{\bm{F}}\l( \| \bm{z} \|^{\frac{2 - 2\vartheta_{\bm{F}}}{1 - 2\vartheta_{\bm{F}}}} \r) \r)
		= \prod_{j = 1}^{r} \exp\l( O_{\bm{F}}\l( |z_{j}|^{\frac{2 - 2\vartheta_{\bm{F}}}{1 - 2\vartheta_{\bm{F}}}} \r) \r)
	\end{align*}
	for $L \geq 2$.
	Combing this estimate with the estimate in the case $L < 2$, we have
	\begin{align} \label{p_Psi_F_5}
		\prod_{p \leq L}M_{p, \frac{1}{2}}(\bm{z})
		\leq \prod_{j = 1}^{r} \exp\l( O_{\bm{F}}\l( |z_{j}|^{\frac{2 - 2\vartheta_{\bm{F}}}{1 - 2\vartheta_{\bm{F}}}} \r) \r)
	\end{align}
	for all $L \geq 0$.
	Also, it follows from the definitions of $\s_{F}(X)$, $\tau_{j_{1}, j_{2}}(X)$, and $K_{\bm{F}, \bm{\theta}}(p, \bm{z})$ that
	\begin{align}  \label{ES_sI_PP_p6}
		\sum_{p \leq L}K_{\bm{F}, \bm{\theta}}(p, \bm{z})
		= 2\sum_{j = 1}^{r}z_{j}^2 \s_{F_{j}}(L)^2
		+ 4\sum_{1 \leq j_{1} < j_{2} \leq r} z_{j_{1}} z_{j_{2}} \tau_{j_{1}, j_{2}}(L).
	\end{align}
	Using (A1), (A2), we find that
	\begin{align*}
		\tau_{j_{1}, j_{2}}(Z)
		= \frac{1}{2}\sum_{p \leq Z}\sum_{\ell = 2}^{\infty}
		\frac{\Re e^{-i\theta_{j_{1}}}b_{F_{j_{1}}}(p^{\ell}) \ol{e^{-i\theta_{j_{2}}}b_{F_{j_{2}}}(p^{\ell})}}{p^{\ell}}
		+ O_{\bm{F}}(1)
	\end{align*}
	for any $Z \geq 2$ when $j_{1} \not= j_{2}$.
	Similarly to \eqref{ES_sI_PP_p51}, we obtain
	\begin{align*}
		\frac{1}{2}\sum_{p \leq Z}\sum_{\ell = 2}^{\infty}
		\frac{\Re e^{-i\theta_{j_{1}}}b_{F_{j_{1}}}(p^{\ell}) \ol{e^{-i\theta_{j_{2}}}b_{F_{j_{2}}}(p^{\ell})}}{p^{\ell}}
		\ll_{\bm{F}} \sum_{p}\sum_{2 \leq \ell \leq K_{1}}\frac{\|b_{\bm{F}}(p^{\ell})\|^2}{p^{\ell}} + 1,
	\end{align*}
	and by (S5), this is $\ll_{\bm{F}} 1$.
	Hence, assumptions (S4), (S5), (A1), and (A2) yield that
	\begin{align} \label{BDtau}
		\tau_{j_{1}, j_{2}}(Z)
		\ll_{\bm{F}} 1
	\end{align}
	for any $Z \geq 2$ when $j_{1} \not= j_{2}$.
	By Lemma \ref{SumbLargeP},
	\begin{align} \label{ESTsigma}
		\s_{F_{j}}(Z)^2 = \frac{n_{F_{j}}}{2} \log{\log{Z}} + O_{F_{j}}(1)
	\end{align}
	for any $Z \geq 2$.
	If $|z_{j}| \leq \|\bm{z}\|^{1/2}$, then $z_{j}^{2} \s_{F_{j}}(L)^2 \ll_{\bm{F}} \| \bm{z} \|^{2}$.
	If $|z_{j}| > \|\bm{z}\|^{1/2}$, then we use \eqref{ESTsigma} to obtain
	\begin{align*}
		\frac{z_{j}^2}{2}\s_{F_{j}}(L)^2
		= \frac{z_{j}^{2}}{2}\s_{F_{j}}(|z_{j}|)^2 + \frac{z_{j}^{2}}{4} \sum_{|z_{j}| < p \leq L}\frac{|a_{F_{j}}|^2}{p}
		= \frac{z_{j}^{2}}{2}\s_{F_{j}}(|z_{j}|)^2 + O_{F_{j}}\l(\| \bm{z} \|^{2}\r).
	\end{align*}
	From this observation and \eqref{ES_sI_PP_p6}, we find that
	\begin{align}
		\prod_{p \leq L}\exp\l( - K_{\bm{F}, \bm{\theta}}(p, \bm{z}) / 4 \r)
		\label{p_Psi_F_6}
		 & = \prod_{j = 1}^{r}\exp\l( -\frac{z_{j}^2}{2}\s_{F_{j}}(|z_{j}|)^2 + O_{\bm{F}}(|z_{j}|^2) \r).
	\end{align}
	Hence, we obtain
	\begin{align} \label{p_Psi_F_2}
		\bigg|\prod_{p \leq L}\frac{M_{p, \frac{1}{2}}(\bm{z})}{\exp\l( K_{\bm{F}, \bm{\theta}}(p, \bm{z}) / 4 \r)}\bigg|
		\leq \bigg|\prod_{j = 1}^{r}\exp\l( -\frac{z_{j}^2}{2}\s_{F_{j}}\l( |z_{j}| \r)^2
		+ O_{\bm{F}}\l( |z_{j}|^{2} + |z_{j}|^{\frac{2 - 2\vartheta_{\bm{F}}}{1 - 2\vartheta_{\bm{F}}}} \r) \r)\bigg|.
	\end{align}
	We also see that
	\begin{align} \label{ES_sI_PP_p7}
		\prod_{p > L}\frac{M_{p, \frac{1}{2}}(\bm{z})}{\exp\l( K_{\bm{F}, \bm{\theta}}(p, \bm{z})/4 \r)}
		= \exp\l( O_{\bm{F}}\l(\| \bm{z} \|^2\r) \r)
		= \prod_{j = 1}^{r}\exp\l( O_{\bm{F}}\l( |z_{j}|^2 \r) \r)
	\end{align}
	from \eqref{ES_sI_PP_p2}.
	These estimates yield inequality \eqref{Psi_F_1}.

	Next, we show \eqref{Psi_F_3}.
	It holds from the definition of $M_{p, \frac{1}{2}}(\bm{z})$ that
	\begin{align*}
		\exp\l( -\sum_{j = 1}^{r}|x_{j}|\sum_{k = 1}^{\infty}\frac{|b_{F_{j}}(p^{k})|}{p^{k/2}} \r)
		\leq M_{p, \frac{1}{2}}(\bm{x})
		\leq \exp\l( \sum_{j = 1}^{r}|x_{j}|\sum_{k = 1}^{\infty}\frac{|b_{F_{j}}(p^{k})|}{p^{k/2}} \r)
	\end{align*}
	for any $\bm{x} = (x_{1}, \dots, x_{r}) \in \RR^{r}$.
	Therefore, we have
	\begin{align*}
		\prod_{p \leq L}M_{p, \frac{1}{2}}(\bm{x})
		= \exp\l( O_{\bm{F}}\l( \| \bm{x} \| \sum_{p \leq L}\sum_{k = 1}^{\infty}\frac{\|b_{\bm{F}}(p^{k})\|}{p^{k/2}} \r) \r).
	\end{align*}
	Similarly to \eqref{p_Psi_F_5} and by this equation, we have
	\begin{align*}
		\prod_{p \leq L}M_{p, \frac{1}{2}}(\bm{x})
		= \prod_{j = 1}^{r}\exp\l( O_{\bm{F}}\l(|x_{j}|^{\frac{2 - 2\vartheta_{\bm{F}}}{1 - 2\vartheta_{\bm{F}}}} \r) \r).
	\end{align*}
	We can calculate the other parts similarly to the proof of \eqref{Psi_F_1}, and obtain \eqref{Psi_F_3}.

	Finally, we prove equation \eqref{Psi_F_2}.
	Since $\Psi$ is analytic on $\CC^{r}$, we can write
	\begin{align*}
		\Psi(x_{1} + iu_{1}, \dots, x_{r} + iu_{r})
		 & = \sum_{n = 0}^{\infty}\sum_{\substack{k_{1} + \dots + k_{r} = n \\ k_{1}, \dots, k_{r} \geq 0}}
		\frac{1}{k_{1}! \cdots k_{r}!}
		\frac{\partial^{n} \Psi(x_{1}, \dots, x_{r})}{\partial z_{1}^{k_{1}} \cdots \partial z_{r}^{k_{r}}}
		(iu_{1})^{k_{1}} \cdots (iu_{r})^{k_{r}}.
	\end{align*}
	It follows from estimates \eqref{Psi_F_1} and \eqref{Psi_F_3} that
	\begin{align*}
		|\Psi(z_{1}, \dots, z_{r})|
		\leq \Psi(x_{1}, \dots, x_{r})\exp\l( C\l( \| \bm{x} \|^2 + \| \bm{x} \|^{\frac{2 - 2\vartheta_{\bm{F}}}{1 - 2\vartheta_{\bm{F}}}}\r) \r)
	\end{align*}
	for some $C = C(\bm{F}) > 0$ when $|z_{1} - x_{1}| = \cdots |z_{r} - x_{r}| = 2$.
	Using this estimate and Cauchy's integral formula, we find that
	\begin{align*}
		\frac{\partial^{n} \Psi(x_{1}, \dots, x_{r})}{\partial z_{1}^{k_{1}} \cdots \partial z_{r}^{k_{r}}}
		 & = \frac{k_{1}! \cdots k_{r}!}{(2\pi i)^{r}}\int_{|z_{r} - x_{r}| = 2} \cdots \int_{|z_{1} - x_{1}| = 2}
		\frac{\Psi(z_{1}, \dots, z_{r})}{(z_{1} - x_{1})^{k_{1}} \cdots (z_{r} - x_{r})^{k_{r}}}dz_{1} \cdots dz_{r} \\
		 & \leq 2^{-(k_{1} + \cdots + k_{r})}k_{1}! \cdots k_{r}! \Psi(x_{1}, \dots, x_{r})
		\exp\l(C \l( \| \bm{x} \|^2 + \|\bm{x}\|^{\frac{2 - 2\vartheta_{\bm{F}}}{1 - 2\vartheta_{\bm{F}}}}\r) \r).
	\end{align*}
	Hence, when $\| \bm{u} \| \leq 1$, we have
	\begin{align*}
		\Psi(x_{1} + iu_{1}, \dots, x_{r} + iu_{r})
		 & = \Psi(x_{1}, \dots, x_{r})\l(1
		+ O_{\bm{F}}\l(\| \bm{u} \| \exp\l( C\| \bm{x} \|^{\frac{2 - 2\vartheta_{\bm{F}}}{1 - 2\theta_{\bm{F}}}} \r) \r)\r) \\
		 & = \Psi(x_{1}, \dots, x_{r})\prod_{j = 1}^{r}\l(1
		+ O_{\bm{F}}\l(|u_{j}| \exp\l( C \| \bm{x} \|^{\frac{2 - 2\vartheta_{\bm{F}}}{1 - 2\theta_{\bm{F}}}} \r) \r)\r),
	\end{align*}
	which completes the proof of \eqref{Psi_F_2}.
\end{proof}

\begin{lemma} \label{EST_Xi_lem}
  Let $\Xi_{X}$ be the function defined by \eqref{def_Xi}.
	For $\bm{x} = (x_{1}, \dots, x_{r}) \in (\RR_{\geq 0})^{r}$, we have
	\begin{align} \label{EST_Xi}
		\Xi_{X}(\bm{x})
		= \prod_{j = 1}^{r}\exp\l( -\frac{x_{j}^2}{2}\s_{F_{j}}(|x_{j}|)^{2}
		+ O_{\bm{F}}\l( x_{j}^2 + |x_{j}|^{\frac{2 - 2\vartheta_{\bm{F}}}{1 - 2\vartheta_{\bm{F}}}} \r) \r).
	\end{align}
\end{lemma}

\begin{proof}
	From \eqref{BDtau}, formula \eqref{Psi_F_3}, and the definition of $\Xi_{X}$, we have
	\begin{align*}
		\Xi_{X}(\bm{x})
		 & = \prod_{j = 1}^{r}\exp\l( -\frac{x_{j}^2}{2}\s_{F_{j}}(|x_{j}|)^{2}
		+ O_{\bm{F}}\l( \| \bm{x} \|^{2} + x_{j}^2 + |x_{j}|^{\frac{2 - 2\vartheta_{\bm{F}}}{1 - 2\vartheta_{\bm{F}}}} \r) \r) \\
		 & = \prod_{j = 1}^{r}\exp\l( -\frac{x_{j}^2}{2}\s_{F_{j}}(|x_{j}|)^{2}
		+ O_{\bm{F}}\l( x_{j}^2 + |x_{j}|^{\frac{2 - 2\vartheta_{\bm{F}}}{1 - 2\vartheta_{\bm{F}}}} \r) \r),
	\end{align*}
	which completes this lemma.
\end{proof}

\begin{lemma}	\label{ES_sI}
	Assume that $\bm{F}$ satisfies (S4), (A1), and (A2).
	For $\bm{z} = (z_1, \dots, z_r) \in \CC^{r}$, $X \geq C \|\bm{z}\|^{\frac{2}{1-2\vartheta_{\bm{F}}}} + 3$
	with $C = C(\bm{F})$ a sufficiently large positive constant, we have
	\begin{align}
		\label{ES_sI1}
		\bigg|\prod_{p \leq X}M_{p, \frac{1}{2}}(\bm{z})\bigg|
		\leq \bigg|\prod_{j = 1}^{r}
		\exp\l(\frac{z_{j}^2}{2} \l(\s_{F_{j}}(X)^2 - \s_{F_{j}}(|z_{j}| \r)^2
		+ O_{\bm{F}}\l( |z_{j}|^{2} + |z_{j}|^{\frac{2 - 2\vartheta_{\bm{F}}}{1 - 2\vartheta_{\bm{F}}}} \r) \r)\bigg|,
	\end{align}
	where $\s_{F_{j}}(X)$ is defined by \eqref{def_var}.
	Moreover, there exists a positive constant $b_{2} = b_{2}(\bm{F})$ such that,
	for any $X \geq 3$ and any $\bm{z} = (z_{1}, \dots, z_{r}) \in \CC^{r}$ with $\norm[]{\bm{z}} \leq b_{2}$, we have
	\begin{align}	\label{ES_sI2}
		\prod_{p \leq X}M_{p, \frac{1}{2}}(\bm{z})
		= \prod_{j = 1}^{r}\l( 1 + O_{\bm{F}}\l( |z_{j}|^2 \r) \r)\exp\l( \frac{z_{j}^2}{2} \s_{F_{j}}(X)^2 \r).
	\end{align}
	Furthermore, for any $\bm{z} = (x_{1} + iu_{1}, \dots, x_{r} + iu_{r}) \in \CC^{r}$ with $x_{j}, u_{j} \in \RR$ and $\| \bm{u} \| \leq 1$,
	and any $X \geq C \|\bm{x} \|^{\frac{2}{1-2\vartheta_{\bm{F}}}} + 3$
	with $C = C(\bm{F})$ a sufficiently large positive constant, we have
	\begin{align} \label{ES_sI3}
		 & \prod_{p \leq X}M_{p, \frac{1}{2}}(\bm{z}) \\
		 & = \Xi_{X}(\bm{x})
		\prod_{j = 1}^{r}\l( 1 + O_{\bm{F}}\l( |u_{j}|\exp\l( D_{1}\| \bm{x} \|^{\frac{2 - 2\vartheta_{\bm{F}}}{1 - 2\vartheta_{\bm{F}}}} \r)
		+ \frac{|z_{j}|^{4}}{\log{X}} \r) \r)\exp\l( \frac{z_{j}^2}{2}\s_{F_{j}}(X)^2 \r),
	\end{align}
	where $\Xi_{X}$ is the function defined by \eqref{def_Xi}, and $D_{1}$ is the same constant as in Lemma \ref{Prop_Psi_FPP}.
\end{lemma}

\begin{proof}
	First, we prove \eqref{ES_sI1}.
	It holds that
	\begin{align} \label{ES_sI_PP1}
		 & \prod_{p \leq X}M_{p, \frac{1}{2}}(\bm{z})
		= \Psi(\bm{z})\prod_{p \leq X}\exp\l( K_{\bm{F}, \bm{\theta}}(p, \bm{z}) / 4 \r)
		\times \prod_{p > X}\frac{\exp\l( K_{\bm{F}, \bm{\theta}}(p, \bm{z}) / 4 \r)}{M_{p, \frac{1}{2}}(\bm{z})}.
	\end{align}
	Using \eqref{ES_sI_PP_p2}, we have
	\begin{align} \label{ES_sI_PP2}
		\prod_{p > X}\frac{\exp\l( K_{\bm{F}, \bm{\theta}}(p, \bm{z}) / 4 \r)}{M_{p, \frac{1}{2}}(\bm{z})}
		 & = \prod_{j = 1}^{r}\exp\l( O_{\bm{F}}\l( \frac{|z_{j}|^3}{X^{\frac{1}{2} - \vartheta_{\bm{F}}}} \r) \r) \\
		 & = \prod_{j = 1}^{r}\exp\l( O_{\bm{F}}\l( |z_{j}|^2 \r) \r)
	\end{align}
	when $X \geq C \| \bm{z} \|^{\frac{2}{1 - 2\vartheta_{\bm{F}}}}$ with $C$ a suitably large constant.
	Also, as in the proof of \eqref{p_Psi_F_6}, we find that
	\begin{align} \label{ES_sI_PP3}
		\prod_{p \leq X}\exp\l( K_{\bm{F}, \bm{\theta}}(p, \bm{z}) / 4 \r)
		= \prod_{j = 1}^{r}\exp\l( \frac{z_{j}^2}{2}\s_{F_{j}}(X)^{2} + O_{\bm{F}}\l( |z_{j}|^{2} \r) \r).
	\end{align}
	Combing the above two estimates and inequality \eqref{Psi_F_1}, we have estimate \eqref{ES_sI1}.

	Next, we prove \eqref{ES_sI2}.
	When $\| \bm{z} \| \leq b_{2}$ with $b_{2} = b_{2}(\bm{F})$ sufficiently small, it holds from Lemma \ref{EQTM} that
	\begin{align*}
		|M_{p, \frac{1}{2}}(\bm{z}) - 1|
		\leq \frac{1}{2},
	\end{align*}
	and that
	\begin{align*}
		M_{p, \frac{1}{2}}(\bm{z})
		=1 + \frac{1}{4}K_{\bm{F}, \bm{\theta}}(p, \bm{z})
		+ O_{\bm{F}}\l( \|\bm{z}\|^3 \l(\sum_{\ell = 1}^{\infty}\frac{\| b_{\bm{F}}(p^{\ell}) \|}{p^{\ell/2}}\r)^{3} \r)
	\end{align*}
	with $b_{\bm{F}}(p^{\ell}) = (b_{F_{1}}(p^{\ell}), \dots, b_{F_{r}}(p^{\ell}))$.
	Using these and estimate \eqref{ES_sI_PP_p3},
	we obtain
	\begin{align*}
		\sum_{p \leq X}\log{M_{p, \frac{1}{2}}(\bm{z})}
		 & = \sum_{p \leq X}\l(\frac{1}{4}K_{\bm{F}, \bm{\theta}}(p, \bm{z})
		+ O_{\bm{F}}\l( \| \bm{z} \|^{3} \l(\sum_{\ell = 1}^{\infty}\frac{\| b_{\bm{F}}(p^{\ell}) \|}{p^{\ell/2}}\r)^{3} \r) \r) \\
		 & = \frac{1}{4}\sum_{p \leq X}K_{\bm{F}, \bm{\theta}}(p, \bm{z})
		+ O_{\bm{F}}\l(\norm{\bm{z}}^{3}\r).
	\end{align*}
	Similarly to \eqref{ES_sI_PP_p6}, we also have
	\begin{align}
		\label{EQKF2}
		 & \frac{1}{4}\sum_{p \leq X}K_{\bm{F}, \bm{\theta}}(p, \bm{z}) 
		= \sum_{j=1}^{r}\frac{z_{j}^2}{2}\s_{F_{j}}(X)^2
		+ O_{\bm{F}}\l( \| \bm{z} \|^2 \r).
	\end{align}
	Hence, it holds that
	\begin{align*}
		\sum_{p \leq X}\log{M_{p, \frac{1}{2}}(\bm{z})}
		 & = \sum_{j = 1}^{r}\frac{z_{j}^2}{2}\s_{F_{j}}(X)^2
		+ O_{\bm{F}}\l( \|\bm{z}\|^2 \r)                                                   \\
		 & = \sum_{j = 1}^{r}\frac{z_{j}^2}{2}\l(\s_{F_{j}}(X)^2 + O_{\bm{F}}\l( 1 \r)\r),
	\end{align*}
	which completes the proof of \eqref{ES_sI2}.

	Finally, we prove \eqref{ES_sI3}.
	It follows from equations \eqref{ES_sI_PP1} and \eqref{ES_sI_PP2} that
	\begin{align*}
		\prod_{p \leq X}M_{p, \frac{1}{2}}(\bm{z})
		= \Psi(\bm{z})\exp\l( \frac{1}{4}\sum_{p \leq X}K_{\bm{F}, \bm{\theta}}(p, \bm{z}) \r)
		\prod_{j = 1}^{r}\l( 1 + O_{\bm{F}}\l( \frac{| z_{j} |^3}{X^{\frac{1}{2} - \vartheta_{\bm{F}}}} \r) \r).
	\end{align*}
	for $X \geq C \|\bm x\|^{\frac{2}{1 - 2\vartheta_{\bm{F}}}} + 3$ with $C = C(\bm{F})$ sufficiently large.
	Moreover, by equation \eqref{ES_sI_PP_p6}, it holds that
	\begin{align}
		\exp\l( \frac{1}{4}\sum_{p \leq X}K_{\bm{F}, \bm{\theta}}(p, \bm{z}) \r)
		= \exp\l( \sum_{j = 1}^{r}\frac{z_{j}^2}{2}\s_{F_{j}}(X)
		+ \sum_{1 \leq j_{1} < j_{2} \leq r}z_{j_{1}} z_{j_{2}} \tau_{j_{1}, j_{2}}(X) \r).
	\end{align}
	In particular, from \eqref{BDtau}, the estimate $\tau_{{l_{1}}, {l_{2}}}(X) \ll_{\bm{F}} 1$ holds
	for all $1 \leq j_{1} < j_{2} \leq r$ by assumptions (S4), (S5), (A1), and (A2),
	and so the above is equal to
	\begin{align}
		\exp\l(\sum_{1 \leq l_{1} < l_{2} \leq r}x_{l_{1}} x_{l_{2}} \tau_{l_{1}, l_{2}}(X) \r)
		\prod_{j = 1}^{r}(1 + O_{\bm{F}}( |u_{j}| \cdot \| \bm{x} \| + u_{j}^2 ))\exp\l( \frac{z_{j}^2}{2}\s_{F_{j}}(X) \r).
	\end{align}
	Additionally, we have
	$
		\Psi(\bm{z}) = \Psi(\bm{x})\prod_{j = 1}^{r}\l(1 + O_{\bm{F}}\l(|u_{j}|
		\exp\l( D_{1} \| \bm{x} \|^{(2-2\vartheta_{\bm{F}}) / (1 - 2\vartheta_{\bm{F}})} \r) \r)\r)
	$
	by \eqref{Psi_F_2}.
	From the above estimates and the definition of $\Xi_{X}$ \eqref{def_Xi}, we also obtain formula \eqref{ES_sI3}.
	Thus, we complete the proof of this lemma.
\end{proof}

\subsection{Proofs of Propositions \ref{Main_Prop_JVD} and \ref{Main_Prop_JVD3}} \label{Proof_Props_JVD}

In this section, we prove Propositions \ref{Main_Prop_JVD} and Proposition \ref{Main_Prop_JVD3}.
Proposition \ref{Main_Prop_JVD3} is a generalization of Radziwi\l\l's result\footnote{That there are some incomplete arguments 
when truncating the integral in his proof, but we can fix his proof by a nontrivial modification.} \cite[Proposition 2]{Ra2011}.
Radziwi\l\l \ adopted Tenenbaum's formula \cite{Te1988} which relates the measure of the large values of the Dirichlet polynomial 
to a line integral. 
If we follow the same approach, we will need to assume Ramanujan's conjecture. 
The reason is that
we have to use formula \eqref{RKLJVDPP1} in a larger range of $\bm{z}$ when truncating the integral and
the estimate \eqref{ES_sI1} in Lemma \ref{ES_sI} is not enough when $\vartheta_{\bm{F}} > 0$.
We avoid these obstacles by probabilistic methods in large deviation theory (see Cramér \cite{Cramer}, Hwang \cite{HH1996})). 
We consider the associated distribution $\nu_{T, \bm{F}, \bm{x}}$ (see below)
and use the Selberg-Beurling functions to approximate regions in $\mathbb R^r$.
The associated distribution $\nu_{T, \bm{F}, \bm{x}}$ plays the role of localizing the measure by changing $\bm{x}$.
Eventually, we choose a suitable $\bm{x}$ in the proof of Propositions \ref{Main_Prop_JVD}, \ref{Main_Prop_JVD3}, 
and the $\bm{x}$ is the saddle point of the integrand in Lemma \ref{SPKLH}.
The Selberg-Beurling formula has the advantage that it only involves finite integrals. 

We introduce some notation.
Define the $\RR^{r}$-valued function $\bm{F}_{\bm{\theta}, X}(t)$ by
\begin{align*}
	\bm{F}_{\bm{\theta}, X}(t)
	= (\Re e^{-i\theta_{1}} P_{F_{1}}(\tfrac{1}{2}+it, X), \dots, \Re e^{-i\theta_{r}} P_{F_{r}}(\tfrac{1}{2}+it, X)),
\end{align*}
and $\mu_{T, \bm{F}}$ the measure on $\RR^{r}$ by
$
	\mu_{T, \bm{F}}(B)
	:= \frac{1}{T}\meas(\bm{F}_{\bm{\theta}, X}^{-1}(B) \cap \mathcal{A})
$
for $B \in \mathcal{B}(\RR^{r})$.
Put $y_{j} = V_{j} \s_{F_{j}}(X)$.
Then we find that
\begin{align}
	\label{RTmuS}
	 & \frac{1}{T}\meas(\S_{X}(T, \bm{V}; \bm{F}, \bm{\theta}))                \\
	 & = \mu_{T, \bm{F}}((y_{1}, \infty) \times \cdots \times (y_{r}, \infty))
	+ O_{\bm{F}}\l(\exp\l( -b_{0}(\log{\log{X}})^{4(r + 1)} \r)\r)
\end{align}
since $\meas([T, 2T] \setminus \mathcal{A}) \ll_{\bm{F}} T\exp( -b_{0}(\log{\log{X}})^{4(r + 1)} )$ by Lemma \ref{ESAEG_JVDPP}.
For $\bm{x} = (x_{1}, \dots, x_{r}) \in \RR^{r}$, set
\begin{align*}
	\nu_{T, \bm{F}, \bm{x}}(B)
	:= \int_{B}e^{x_{1} \xi_{1} + \cdots + x_{r} \xi_{r}}d\mu_{T, \bm{F}}(\bm{\xi})
\end{align*}
for $B \in \mathcal{B}(\RR^{r})$.
Note that $\nu_{T, \bm{F}, \bm{x}}$ is a measure on $\RR^{r}$, and has a finite value
for every $B \in \mathcal{B}(\RR^{r})$, $\bm{x} \in \RR^{r}$, $X \geq 3$
in the sense
\begin{align*}
	\nu_{T, \bm{F}, \bm{x}}(B)
	\leq \nu_{T, \bm{F}, \bm{x}}(\RR^{r})
	= \frac{1}{T}\int_{\mathcal{A}}\exp\l( \sum_{j = 1}^{r}x_{j}\Re e^{-i\theta_{j}}P_{F_{j}}(\tfrac{1}{2}+it, X) \r)dt
	< +\infty.
\end{align*}
This is an analogue of the associated distribution function of $\tilde{M}_{n}(w)$ in Hwang \cite[page 300]{HH1996}.
Under the above notation, we state and prove three lemmas.

\begin{lemma}	\label{SPKLH}
	For $x_{1}, \dots, x_{r} > 0$, we have
	\begin{align*}
		 & \mu_{T, \bm{F}}((y_{1}, \infty) \times \cdots \times (y_{r}, \infty))                               \\
		 & = \int_{x_{r} y_{r}}^{\infty} \cdots \int_{x_{1} y_{1}}^{\infty}e^{-(\tau_{1} + \cdots + \tau_{r})}
		\nu_{T, \bm{F}, \bm{x}}((y_{1}, \tau_{1}/x_{1}) \times \cdots \times (y_{r}, \tau_{r}/x_{r}))d\tau_{1} \cdots d\tau_{r}.
	\end{align*}
\end{lemma}

\begin{proof}
	For every $B \in \mathcal{B}(\RR^{r})$, it holds that
	\begin{align}
		\mu_{T, \bm{F}}(B)
		 & = \int_{B}e^{-(x_{1}v_{1} + \cdots + x_{r}v_{r})}d\nu_{T, \bm{F}, \bm{x}}(\bm{v}).
	\end{align}
	By Fubini's theorem, we find that
	\begin{align*}
		 & \int_{(y_{1}, \infty) \times \cdots \times (y_{r}, \infty)}e^{-(x_{1}v_{1} + \cdots + x_{r}v_{r})}
		d\nu_{T, \bm{F}, \bm{x}}(\bm{v})                                                                       \\
		 & = \int_{(y_{1}, \infty) \times \cdots \times (y_{r}, \infty)}
		\l(\int_{x_{r} v_{r}}^{\infty} \cdots \int_{x_{1} v_{1}}^{\infty}
		e^{-(\tau_{1} + \cdots + \tau_{r})}d\tau_{1} \cdots d\tau_{r}\r)
		d\nu_{T, \bm{F}, \bm{x}}(\bm{v})                                                                       \\
		 & = \int_{x_{r} y_{r}}^{\infty} \cdots \int_{x_{1} y_{1}}^{\infty}e^{-(\tau_{1} + \cdots + \tau_{r})}
		\l(\int_{(y_{1}, \tau_{1}/x_{1}) \times \cdots \times (y_{r}, \tau_{r}/x_{r})}1d\nu_{T, \bm{F}, \bm{x}}(\bm{v})\r)
		d\tau_{1} \cdots d\tau_{r}                                                                             \\
		 & = \int_{x_{r} y_{r}}^{\infty} \cdots \int_{x_{1} y_{1}}^{\infty}e^{-(\tau_{1} + \cdots + \tau_{r})}
		\nu_{T, \bm{F}, \bm{x}}((y_{1}, \tau_{1}/x_{1}) \times \cdots \times (y_{r}, \tau_{r}/x_{r}))d\tau_{1} \cdots d\tau_{r}.
	\end{align*}
\end{proof}

The next lemma is a generalization of \cite[Lemma 6.2]{LLR2019} in multidimensions.
Define
\begin{gather}
	G(u) = \frac{2u}{\pi} + \frac{2 (1 - u)u}{\tan{\pi u}}, \ \
	f_{c, d}(u) = \frac{e^{-2\pi i c u} - e^{-2 \pi i d u}}{2}.
\end{gather}
For a set $A$, we denote the indicator function of $A$ by $\bm{1}_{A}$.

\begin{lemma}	\label{Multi_BSF}
	Let $L$ be a positive number.
	Let $c_{1}, \dots, c_{r}, d_{1}, \dots, d_{r}$ be real numbers with $c_{j} < d_{j}$.
	Put $\mathscr{R} = (c_{1}, d_{1}) \times \cdots \times (c_{r}, d_{r}) \subset \RR^{r}$.
	For any $\bm{\xi} = (\xi_{1}, \dots, \xi_{r}) \in \RR^{r}$, we have
	\begin{align*}
		\bm{1}_{\mathscr{R}}(\bm{\xi})
		 & = W_{L, \mathscr{R}}(\bm{\xi})
		+ O_{r}\l( \sum_{j = 1}^{r}\l\{ \l( \frac{\sin(\pi L(\xi_{j} - c_{j}))}{\pi L(\xi_{j} - c_{j})} \r)^2
		+ \l( \frac{\sin(\pi L(\xi_{j} - d_{j}))}{\pi L(\xi_{j} - d_{j})} \r)^2 \r\} \r),
	\end{align*}
	where $W_{L, \mathscr{R}}(\bm{\xi})$ is defined as if $r$ is even,
	\begin{align*}
		\frac{i^{r}}{2^{r-1}}\sum_{j = 1}^{r}(-1)^{j-1}
		\Re \prod_{h = 1}^{r}\int_{0}^{L}G\l( \frac{u}{L} \r)
		e^{2\pi i \e_{j}(h) u \xi_{h}}f_{c_{h}, d_{h}}(\e_{j}(h)u)\frac{du}{u},
	\end{align*}
	if $r$ is odd,
	\begin{align*}
		\frac{i^{r+1}}{2^{r-1}}\sum_{j = 1}^{r}(-1)^{j-1}
		\Im \prod_{h = 1}^{r}\int_{0}^{L}G\l( \frac{u}{L} \r)
		e^{2\pi i \e_{j}(h) u \xi_{h}}f_{c_{h}, d_{h}}(\e_{j}(h)u)\frac{du}{u}.
	\end{align*}
	Here, $\e_{j}(h) = 1$ if $1 \leq h \leq j-1$, and $\e_{j}(h) = -1$ otherwise.
\end{lemma}

\begin{proof}
	We use the following formula (cf.  \cite[equation (6.1)]{LLR2019})
	\begin{align*}
		\bm{1}_{(c_{h}, d_{h})}(\xi_{h})
		 & = \Im \int_{0}^{L}G\l( \frac{u}{L} \r)e^{2\pi i u \xi_{h}}f_{c_{h}, d_{h}}(u)\frac{du}{u} \\
		 & \quad + O\l( \l( \frac{\sin(\pi L(\xi_{h} - c_{h}))}{\pi L(\xi_{h} - c_{h})} \r)^2
		+ \l( \frac{\sin(\pi L(\xi_{h} - d_{h}))}{\pi L(\xi_{h} - d_{h})} \r)^2 \r),
	\end{align*}
	which  leads to
	the estimate $\Im \int_{0}^{L}G\l( \frac{u}{L} \r)e^{2\pi i u \xi_{j}}f_{c_{j}, d_{j}}(u)\frac{du}{u} \ll 1$.
	Therefore, we obtain
	\begin{align}
		\label{Multi_BSF_1}
		\bm{1}_{\mathscr{R}}(\bm{\xi})
		 & = \prod_{h = 1}^{r}\Im \int_{0}^{L}G\l( \frac{u}{L} \r)
		e^{2\pi i u \xi_{h}}f_{c_{h}, d_{h}}(u) \frac{du}{u}                                                           \\
		 & \quad + O_{r}\l( \sum_{j = 1}^{r}\l\{ \l( \frac{\sin(\pi L(\xi_{j} - c_{j}))}{\pi L(\xi_{j} - c_{j})} \r)^2
		+ \l( \frac{\sin(\pi L(\xi_{j} - d_{j}))}{\pi L(\xi_{j} - d_{j})} \r)^2 \r\} \r).
	\end{align}
	For any complex numbers $w_{1}, \dots, w_{r}$, we observe that
	\begin{align*}
		\Im(w_{1}) \cdots \Im(w_{r})
		= \frac{i^{r}}{2^{r}}\sum_{j = 1}^{r}(-1)^{j - 1}
		\l(w_{1} \cdots w_{j - 1}\ol{w_{j} \cdots w_{r}} + (-1)^{r} \ol{w_{1} \cdots \ol{w_{j} \cdots w_{r}}}\r).
	\end{align*}
	In particular, if $r$ is even, then
	\begin{align*}
		\Im(w_{1}) \cdots \Im(w_{r})
		= \frac{i^{r}}{2^{r-1}}\Re\sum_{j = 1}^{r}(-1)^{j - 1}w_{1} \cdots w_{j - 1}\ol{w_{j} \cdots w_{r}},
	\end{align*}
	and if $r$ is odd, then
	\begin{align*}
		\Im(w_{1}) \cdots \Im(w_{r})
		= \frac{i^{r+1}}{2^{r-1}}\Im\sum_{j = 1}^{r}(-1)^{j - 1}w_{1} \cdots w_{j - 1}\ol{w_{j} \cdots w_{r}}.
	\end{align*}
	Substituting these to \eqref{Multi_BSF_1}, we obtain Lemma \ref{Multi_BSF}.
\end{proof}

\begin{lemma}	\label{FMnu}
	Assume that $\bm{F}$, $\bm{\theta}$ satisfy (S4), (A1), and (A2).
	Let $c_{1}, \dots, c_{r}, d_{1}, \dots, d_{r}$ be real numbers with $c_{j} < d_{j}$.
	Put $\mathscr{R} = (c_{1}, d_{1}) \times \cdots \times (c_{r}, d_{r})$.
	Let $T$, $X$ be large numbers depending on $\bm{F}$ and satisfying $X^{(\log{\log{X}})^{4(r + 1)}} \leq T$.
	Then for any $\bm{x} = (x_{1}, \dots, x_{r}) \in \RR^{r}$ satisfying $\|\bm{x}\| \leq (\log{\log{X}})^{2r}$, we have
	\begin{align} \label{FMnu1}
		\nu_{T, \bm{F}, \bm{x}}(\mathscr{R})
		 & = \Xi_{X}(\bm{x})\l( \prod_{h = 1}^{r}e^{\frac{x_{h}^2}{2}\s_{F_{h}}(X)^2} \r)
		\times \l\{\prod_{j = 1}^{r}\int_{x_{j} \s_{F_{j}}(X) - \frac{d_{j}}{\s_{F_{j}}(X)}}^{
			x_{j}\s_{F_{j}}(X) - \frac{c_{j}}{\s_{F_{j}}(X)}}e^{-v^2/2}\frac{dv}{\sqrt{2\pi}} + E_{1}\r\},
	\end{align}
	where the error term $E_{1}$ satisfies
	\begin{align*}
		 & E_{1}
		\ll_{\bm{F}}
		\frac{\exp\l( C \| \bm{x} \|^{\frac{2 - 2\vartheta_{\bm{F}}}{1 - 2\vartheta_{\bm{F}}}} \r)}{(\log{\log{X}})^{\a_{\bm{F}} + \frac{1}{2}}}
		+ \frac{\exp\l( C \| \bm{x} \|^{\frac{2 - 2\vartheta_{\bm{F}}}{1 - 2\vartheta_{\bm{F}}}} \r)}{\sqrt{\log{\log{X}}}}
		\prod_{h = 1}^{r}\frac{d_{h} - c_{h}}{\s_{F_{h}}(X)}
	\end{align*}
	for some constant $C = C(\bm{F}) > 0$.
	Moreover, if $\|\bm{x}\| \leq b_{3}$ with $b_{3} = b_{3}(\bm{F}) > 0$ sufficiently small,
	we have
	\begin{align} \label{FMnu2}
		\nu_{T, \bm{F}, \bm{x}}(\mathscr{R})
		=\l( \prod_{h = 1}^{r}e^{\frac{x_{h}^2}{2}\s_{F_{h}}(X)^2} \r)
		\times \l\{\prod_{j = 1}^{r}\int_{x_{j} \s_{F_{j}}(X) - \frac{d_{j}}{\s_{F_{j}}(X)}}^{
			x_{j}\s_{F_{j}}(X) - \frac{c_{j}}{\s_{F_{j}}(X)}}e^{-v^2/2}\frac{dv}{\sqrt{2\pi}} + E_{2}\r\},
	\end{align}
	where the error term $E_{2}$ satisfies
	\begin{align*}
		 & E_{2}
		\ll_{\bm{F}} \frac{1}{(\log{\log{X}})^{\a_{\bm{F}} + \frac{1}{2}}}
		+ \sum_{k = 1}^{r}\l(\frac{x_{k}^2 (d_{k} - c_{k})}{\s_{F_{k}}(X)} + \frac{1}{\s_{F_{k}}(X)^2} \r)
		\prod_{\substack{h = 1 \\ h \not= k}}^{r}\frac{d_{h} - c_{h}}{\s_{F_{h}}(X)}.
	\end{align*}
\end{lemma}

\begin{proof}
	We show formula \eqref{FMnu1}.
	Put $L = b_{4} (\log{\log{X}})^{\a_{\bm{F}}}$ with $b_{4} = b_{4}(\bm{F})$ a small positive constant to be chosen later.
  Recall that $\a_{\bm{F}} = \min\l\{ 2r, \frac{1 - 2\vartheta_{\bm{F}}}{2\vartheta_{\bm{F}}} \r\}$.
	It follows from Lemma \ref{Multi_BSF} that
	\begin{align}
		\label{MPnp1}
		 & \nu_{T, \bm{F}, \bm{x}}(\mathscr{R})
		= \int_{\RR^{r}}W_{L, \mathscr{R}}(\bm{\xi})e^{x_{1}\xi_{1} + \cdots + x_{r}\xi_{r}}
		d\mu_{T, \bm{F}}(\bm{\xi}) + E,
	\end{align}
	where the error term $E$ satisfies the estimate
	\begin{align*}
		E
		\ll \sum_{j = 1}^{r}\int_{\RR^{r}}
		\l\{ \l( \frac{\sin(\pi L(\xi_{j} - c_{j}))}{\pi L(\xi_{j} - c_{j})} \r)^2
		+ \l( \frac{\sin(\pi L(\xi_{j} - d_{j}))}{\pi L(\xi_{j} - d_{j})} \r)^2 \r\}
		e^{x_{1}\xi_{1} + \cdots + x_{r}\xi_{r}}
		d\mu_{T, \bm{F}}(\bm{\xi}).
	\end{align*}
	First, we estimate $E$.
	For $\bm{z} = (z_{1}, \dots, z_{r}) \in \CC^{r}$, define
	\begin{align*}
		\tilde{M}_{T}(\bm{z})
		= \int_{\RR^{r}}e^{z_{1} \xi_{1} + \cdots + z_{r} \xi_{r}}d\mu_{T, \bm{F}}(\bm{\xi}).
	\end{align*}
	Then, it holds that
	\begin{align*}
		\tilde{M}_{T}(\bm{z})
		=\frac{1}{T}\int_{\mathcal{A}}\exp\l( \sum_{j = 1}^{r}z_{j}\Re e^{-i\theta_{j}}P_{F_{j}}(\tfrac{1}{2}+it, X) \r)dt.
	\end{align*}
	Put $\bm{w} = (w_{1}, \dots, w_{r}) = (x_{1} + i u_{1}, \dots, x_{r} + i u_{r})$ with $u_{j} \in \RR$.
	When $\norm[]{(u_{1}, \dots, u_{r})} \leq L$ holds,
	we have
	\begin{align*}
		 & |\tilde{M}_{T}(\bm{w})|                                                                                                     \\
		 & \leq \bigg| \prod_{j = 1}^{r}\exp\l( \frac{(x_{j} + i u_{j})^2}{2}
		\l( \s_{F_{j}}(X)^2 - \s_{F_{j}}(|w_{j}|)^2 \r)
		+ O_{\bm{F}}\l(|x_{j} + i u_{j}|^2 + |x_{j} + i u_{j}|^{\frac{2 - 2\vartheta_{\bm{F}}}{1 - 2\vartheta_{\bm{F}}}}\r) \r) \bigg| \\
		 & \qqqquad + O_{\bm{F}}\l( \exp\l( -b_{1}(\log{\log{X}})^{4(r + 1)} \r) \r)                                                   \\
		 & \ll_{\bm{F}} \exp\l( C \| \bm{x} \|^{\frac{2 - 2\vartheta_{\bm{F}}}{1 - 2\vartheta_{\bm{F}}}} \r)
		\prod_{j = 1}^{r}\exp\l( \frac{x_{j}^2}{2}\l(\s_{F_{j}}(X)^2 - \s_{F_{j}}(|w_{j}|)^{2} \r) - \frac{u_{j}^2}{3}\s_{F_{j}}(X)^2
		+ O_{\bm{F}}\l(u_{j}^2L^{\frac{2\vartheta_{\bm{F}}}{1 - 2\vartheta_{\bm{F}}}}\r)  \r)
	\end{align*}
	by Proposition \ref{RKLJVDPP} and \eqref{ES_sI1}, where $C = C(\bm{F})$ is some positive constant.
	Additionally, by \eqref{EST_Xi}, we find that
	\begin{align*}
		\prod_{j = 1}^{r}\exp\l( -\frac{x_{j}^2}{2}\s_{F_{j}}(|w_{j}|)^{2}\r)
		\leq \prod_{j = 1}^{r}\exp\l( -\frac{x_{j}^2}{2}\s_{F_{j}}(|x_{j}|)^{2} \r)
		\ll_{\bm{F}} \Xi_{X}(\bm{x})\exp\l( C\| \bm{x} \|^{\frac{2 - 2\vartheta_{\bm{F}}}{1 - 2\vartheta_{\bm{F}}}} \r)
	\end{align*}
	for some $C = C(\bm{F}) > 0$.
	By the definition of $\a_{\bm{F}}$, the inequality
	$
		L^{\frac{2\vartheta_{\bm{F}}}{1 - 2\vartheta_{\bm{F}}}}
		\leq (2 b_{4})^{^{\frac{2\vartheta_{\bm{F}}}{1 - 2\vartheta_{\bm{F}}}}}\log{\log{X}}
	$ holds.
	Therefore, when $b_{4}$ is sufficiently small, we have
	\begin{align} \label{ESTMT1}
		 & |\tilde{M}_{T}(x_{1} + iu_{1}, \dots, x_{r} + iu_{r})|                                                           \\
		 & \ll_{\bm{F}} \Xi_{X}(\bm{x})\exp\l( C \| \bm{x} \|^{\frac{2 - 2\vartheta_{\bm{F}}}{1 - 2\vartheta_{\bm{F}}}} \r)
		\prod_{j = 1}^{r}\exp\l( \l(\frac{x_{j}^2}{2} - \frac{u_{j}^2}{4} \r)\s_{F_{j}}(X)^2 \r)
	\end{align}
	for $\|(u_{1}, \dots, u_{r})\| \leq L$.
	For any $\ell, \xi \in \RR$, we can write
	\begin{align}
		\l( \frac{\sin(\pi L (\xi - \ell))}{\pi L (\xi - \ell)} \r)^{2}
		 & = \frac{2}{L^2}\int_{0}^{L}(L - u)\cos(2\pi (\xi - \ell) u)du      \\
		\label{STTBS}
		 & = \frac{2}{L^2}\Re \int_{0}^{L}(L - u)e^{2\pi i (\xi - \ell) u}du.
	\end{align}
	Thus
	\begin{align*}
		 & \int_{\RR^{r}}\l( \frac{\sin(\pi L(\xi_{j} - \ell))}{\pi L(\xi_{j} - \ell)} \r)^2
		e^{x_{1}\xi_{1} + \cdots + x_{r}\xi_{r}}d\mu_{T, \bm{F}}(\bm{\xi})                                                           \\
		 & = \frac{2}{L^2}\Re \int_{0}^{L}(L - u)\int_{\RR^{r}}e^{2\pi i (\xi_{j} - \ell) u}e^{x_{1}\xi_{1} + \cdots + x_{r}\xi_{r}}
		d\mu_{T, \bm{F}}(\bm{\xi}) du                                                                                                \\
		 & = \frac{2}{L^2}\Re \int_{0}^{L}e^{-2\pi i \ell u}
		(L - u)\tilde{M}_{T}(x_{1}, \dots, x_{j-1}, x_{j} + 2\pi i u, x_{j+1}, \dots, x_{r})du,
	\end{align*}
	which, by  \eqref{ESTMT1}, is
	\begin{align*}
		 & \ll_{\bm{F}} \Xi_{X}(\bm{x})\frac{\exp\l( C \| \bm{x} \|^{\frac{2 - 2\vartheta_{\bm{F}}}{1 - 2\vartheta_{\bm{F}}}} \r)}{L^2}\l(\prod_{k = 1}^{r}
		\exp\l( \frac{x_{k}^2}{2}\s_{F_{k}}(X) \r)\r)
		\int_{0}^{L}(L - u)\exp\l( - (\pi \s_{F_{j}}(X) u)^2 \r)du                                                                                          \\
		 & \ll \Xi_{X}(\bm{x})\frac{\exp\l( C \| \bm{x} \|^{\frac{2 - 2\vartheta_{\bm{F}}}{1 - 2\vartheta_{\bm{F}}}} \r)}{L\s_{F_{j}}(X)}
		\prod_{k = 1}^{r}\exp\l( \frac{x_{k}^2}{2}\s_{F_{k}}(X)^2 \r).
	\end{align*}
	It then follows that equation \eqref{MPnp1} satisfies
	\begin{align}
		\label{EqnubS}
		\nu_{T, \bm{F}, \bm{x}}(\mathscr{R})
		 & = \int_{\RR^{r}}W_{L, \mathscr{R}}(\bm{\xi})e^{x_{1}\xi_{1} + \cdots + x_{r}\xi_{r}}
		d\mu_{T, \bm{F}}(\bm{\xi})+E
	\end{align}
	with
	\begin{align*}
		E
		\ll_{\bm{F}} \Xi_{X}(\bm{x})\frac{\exp\l( C \| \bm{x} \|^{\frac{2 - 2\vartheta_{\bm{F}}}{1 - 2\vartheta_{\bm{F}}}} \r)}{L\sqrt{\log{\log{X}}}}
		\prod_{k = 1}^{r}\exp\l( \frac{x_{k}^2}{2}\s_{F_{k}}(X)^2 \r).
	\end{align*}

	For the main term in \eqref{EqnubS}, it is enough to calculate
	\begin{align}
		\label{MPnp2}
		\int_{\RR^{r}}\l(\prod_{h = 1}^{r}\int_{0}^{L}G\l( \frac{u}{L} \r)
		e^{2\pi i \e_{j}(h) u \xi_{h}}f_{c_{h}, d_{h}}(\e_{j}(h) u)\frac{du}{u}\r)
		e^{x_{1}\xi_{1} + \cdots + x_{r}\xi_{r}}
		d\mu_{T, \bm{F}}(\bm{\xi})
	\end{align}
	for every fixed $1 \leq j \leq r$.
	Using Fubini's theorem, we find that \eqref{MPnp2} is equal to
	\begin{multline*}
		\int_{0}^{L} \cdots \int_{0}^{L}
		\l(\prod_{h = 1}^{r}G\l( \frac{u_{h}}{L} \r)\frac{f_{c_{h}, d_{h}}(\e_{j}(h) u_{h})}{u_{h}}\r)\\
		\times \tilde{M}_{T}\l( x_{1} + 2\pi i \e_{j}(1) u_{1}, \dots, x_{r} + 2 \pi i \e_{j}(r) u_{r} \r)
		du_{1} \cdots du_{r}.
	\end{multline*}
	Here we divide the range of this integral as
	\begin{align*}
		\int_{0}^{L} \cdots \int_{0}^{L}
		= \int_{0}^{1} \cdots \int_{0}^{1} +
		\sum_{k = 0}^{r-1} \int \cdots \int_{D_{k}},
	\end{align*}
	where
	\begin{align*}
		\int \cdots \int_{D_{k}}
		= \int_{0}^{1} \cdots \int_{0}^{1} \int_{1}^{L}
		\overbrace{\int_{0}^{L} \cdots \int_{0}^{L}}^{k}.
	\end{align*}
	By estimate \eqref{ESTMT1} and the estimates $\frac{f_{c, d}(\pm u)}{u} \ll d - c$,
	$G(u / L) \ll 1$ for $0 \leq u \leq L$, the integral over $D_{r-k}$ for $1 \leq k \leq r$ is
	\begin{align*}
		 & \ll_{\bm{F}} \Xi_{X}(\bm{x})\exp\l( C \| \bm{x} \|^{\frac{2 - 2\vartheta_{\bm{F}}}{1 - 2\vartheta_{\bm{F}}}} \r)
		\l(\prod_{h = 1}^{k-1}(d_{h} - c_{h})\int_{0}^{1}\exp\l( \l(\frac{x_{h}^2}{2} - (\pi u)^2\r)\s_{F_{h}}(X)^2 \r)du\r) \\
		 & \qqquad \times (d_{k} - c_{k})\int_{1}^{L}\exp\l( \l(\frac{x_{k}^2}{2} - (\pi u)^2 \r)\s_{F_{k}}(X)^2 \r)du       \\
		 & \qqquad \times \l(\prod_{h = k+1}^{r}(d_{h} - c_{h})\int_{0}^{L}
		\exp\l( \l(\frac{x_{h}^2}{2}  - (\pi u)^2 \r)\s_{F_{h}}(X)^2 \r)du\r)                                                \\
		 & \ll_{\bm{F}} \Xi_{X}(\bm{x})\exp\l( C \| \bm{x} \|^{\frac{2 - 2\vartheta_{\bm{F}}}{1 - 2\vartheta_{\bm{F}}}} \r)
		e^{-\s_{F_{k}}(X)^2}\prod_{h = 1}^{r}\frac{d_{h} - c_{h}}{\s_{F_{h}}(X)}
		\exp\l( \frac{x_{h}^2}{2}\s_{F_{h}}(X)^2 \r).
	\end{align*}
	Hence, integral \eqref{MPnp2} is equal to
	\begin{multline}  \label{MPnp3}
		\int_{0}^{1} \cdots \int_{0}^{1}
		\l(\prod_{h = 1}^{r}G\l( \frac{u_{h}}{L} \r)\frac{f_{c_{h}, d_{h}}(\e_{j}(h) u_{h})}{u_{h}}\r)
		\tilde{M}_{T}\l( x_{1} + 2\pi i \e_{j}(1) u_{1}, \dots, x_{r} + 2 \pi i \e_{j}(r) u_{r} \r)du_{1} \cdots du_{r}\\
		+ O_{\bm{F}}\l(\sum_{k=1}^r \Xi_{X}(\bm{x})\exp\l( C \| \bm{x} \|^{\frac{2 - 2\vartheta_{\bm{F}}}{1 - 2\vartheta_{\bm{F}}}} \r)
		e^{-\s_{F_{k}}(X)^2}\prod_{h = 1}^{r}\frac{d_{h} - c_{h}}{\s_{F_{h}}(X)}
		\exp\l( \frac{x_{h}^2}{2}\s_{F_{h}}(X)^2 \r) \r).
	\end{multline}
	When $\|(u_{1}, \dots, u_{r})\| \leq 1$, it follows from Proposition \ref{RKLJVDPP} and equation \eqref{ES_sI3} that
	\begin{align*}
		 & \tilde{M}_{T}\l( x_{1} + 2\pi i \e_{j}(1) u_{1}, \dots, x_{r} + 2 \pi i \e_{j}(r) u_{r} \r)                 \\
		 & = \Xi_{X}(\bm{x})\prod_{h = 1}^{r}\l(1
		+ O_{\bm{F}}\l(|u_{h}|\exp\l( D_{1}\| \bm{x} \|^{\frac{2 - 2\vartheta_{\bm{F}}}{1 - 2\vartheta_{\bm{F}}}} \r)
		+ \frac{x_{j}^{4} + u_{j}^{4}}{\log{X}}\r)\r)                                                                  \\
		 & \qquad \times\exp\l( \frac{x_{h}^2 + 4 \pi i \e_{j}(h) x_{h} u_{h} - 4 \pi^2 u_{h}^2}{2}\s_{F_{h}}(X)^2 \r)
		+ O_{\bm{F}}\l( \exp\l( -b_{1}(\log{\log{X}})^{4(r + 1)} \r) \r).
	\end{align*}
	Therefore, the integral of \eqref{MPnp3} is equal to
	\begin{multline}
		\label{MPnp4}
		\Xi_{X}(\bm{x})\prod_{h = 1}^{r}\int_{0}^{1}\l(1
		+ O_{\bm{F}}\l(u\exp\l( D_{1}\| \bm{x} \|^{\frac{2 - 2\vartheta_{\bm{F}}}{1 - 2\vartheta_{\bm{F}}}} \r)
		+ \frac{x_{h}^{4} + u_{h}^{4}}{\log{X}} \r)\r)
		G\l( \frac{u}{L} \r)f_{c_{h}, d_{h}}(\e_{j}(h)u)\\
		\times \exp\l( \frac{x_{h}^2 + 4 \pi i \e_{j}(h) x_{h} u - 4 \pi^2 u^2}{2}\s_{F_{h}}(X)^2 \r) \frac{du}{u}\\
		+ O_{\bm{F}}\l( \exp\l( -c_{1}(\log{\log{X}})^{4(r + 1)} \r)
		\prod_{h = 1}^{r}\int_{0}^{1} \bigg|G\l( \frac{u}{L} \r) \frac{f_{c_{h}, d_{h}}(u)}{u}\bigg|du \r).
	\end{multline}
	Since $G(u / L) \ll 1$
	and $\frac{f_{c_{h}, d_{h}}(\pm u)}{u} \ll d_{h} - c_{h}$,
	we find that
	\begin{align*}
		 & \int_{0}^{1}\l(u\exp\l( D_{1}\| \bm{x} \|^{\frac{2 - 2\vartheta_{\bm{F}}}{1 - 2\vartheta_{\bm{F}}}} \r)
		+ \frac{x_{h}^{4} + u^{4}}{\log{X}} \r)G\l( \frac{u}{L} \r)f_{c_{h}, d_{h}}(\e_{j}(h)u)                                     \\
		 & \qqquad \times \exp\l( \frac{x_{h}^2 + 4 \pi i \e_{j}(h) x_{h} u - 4 \pi^2 u^2}{2}\s_{F_{h}}(X)^2 \r) \frac{du}{u}       \\
		 & \ll \exp\l( \frac{x_{h}^{2}}{2}\s_{F_{h}}(X)^2 \r)(d_{h} - c_{h})                                                        \\
		 & \qqquad \times \int_{0}^{1} \l(u \exp\l( D_{1}\| \bm{x} \|^{\frac{2 - 2\vartheta_{\bm{F}}}{1 - 2\vartheta_{\bm{F}}}} \r)
		+ \frac{x_{h}^{4} + u^{4}}{\log{X}} \r)\exp\l(- 2 \pi^2 u^2\s_{F_{h}}(X)^2 \r)du                                            \\
		 & \ll_{\bm{F}} \exp\l( D_{1} \| \bm{x} \|^{\frac{2 - 2\vartheta_{\bm{F}}}{1 - 2\vartheta_{\bm{F}}}} \r)
		\times \exp\l( \frac{x_{h}^{2}}{2}\s_{F_{h}}(X)^2 \r)\frac{d_{h} - c_{h}}{\s_{F_{h}}(X)^{2}},
	\end{align*}
	and that
	\begin{align*}
		 & \int_{0}^{1}G\l( \frac{u}{L} \r)f_{c_{h}, d_{h}}(\e_{j}(h) u)
		\exp\l( \frac{x_{h}^2 + 4 \pi i \e_{j}(h) x_{h} u - 4 \pi^2 u^2}{2}\s_{F_{h}}(X)^2 \r) \frac{du}{u} \\
		 & \ll (d_{h} - c_{h})\exp\l( \frac{x_{h}^{2}}{2}\s_{F_{h}}(X)^2 \r)
		\int_{0}^{1} \exp\l(- 2 \pi^2 u^2\s_{F_{h}}(X)^2 \r)du                                              \\
		 & \ll \exp\l( \frac{x_{h}^{2}}{2}\s_{F_{h}}(X)^2 \r)
		\frac{d_{h} - c_{h}}{\s_{F_{h}}(X)}.
	\end{align*}
	Moreover, we find that
	\begin{align*}
		 & \int_{1}^{L}G\l( \frac{u}{L} \r)f_{c_{h}, d_{h}}( \e_{j}(h) u)
		\exp\l( \frac{x_{h}^2 + 4 \pi i \e_{j}(h) x_{h} u - 4 \pi^2 u^2}{2}\s_{F_{h}}(X)^2 \r) \frac{du}{u}                      \\
		 & \ll_{\bm{F}} \exp\l( \frac{x_{h}^{2}}{2}\s_{F_{h}}(X)^{2} \r)\frac{d_{h} - c_{h}}{\s_{F_{h}}(X)}e^{-\s_{F_{h}}(X)^2},
	\end{align*}
	and that
	\begin{align*}
		\int_{0}^{1}\bigg|G\l( \frac{u}{L} \r)\frac{f_{c_{h}, d_{h}}( \e_{j}(h) u)}{u}\bigg|du
		\ll d_{h} - c_{h}.
	\end{align*}
	Also, by \eqref{EST_Xi}, it holds that
	\begin{align*}
		&\Xi_{X}(\bm{x})\exp\l( D_{2} \| \bm{x} \|^{\frac{2 - 2\vartheta_{\bm{F}}}{1 - 2\vartheta_{\bm{F}}}} \r)
		\prod_{j = 1}^{r} \exp\l( \frac{x_{j}^2}{2}\s_{F_{j}}(X)^2 \r) \times \frac{1}{\sqrt{\log{\log{X}}}}\prod_{h = 1}^{r}\frac{1}{\s_{F_{h}}(X)}\\
		&\geq \frac{1}{\sqrt{\log{\log{X}}}}\prod_{h = 1}^{r}\frac{1}{\s_{F_{h}}(X)}
		\geq \exp\l( -b_{1}(\log{\log{X}})^{4(r + 1)} \r)
	\end{align*}
	for some constant $D_{2} = D_{2}(\bm{F}) > 0$.
	From these estimates and \eqref{MPnp4}, integral \eqref{MPnp3} is equal to
	\begin{multline*}
		\Xi_{X}(\bm{x})\prod_{h = 1}^{r}\int_{0}^{L}G\l( \frac{u}{L} \r)f_{c_{h}, d_{h}}(\e_{j}(h)u)
		\exp\l( \frac{x_{h}^2 + 4 \pi i \e_{j}(h) x_{h} u - 4 \pi^2 u^2}{2}\s_{F_{h}}(X)^2 \r) \frac{du}{u}\\
		+ O_{\bm{F}}\l( \Xi_{X}(\bm{x})\exp\l( D_{3} \| \bm{x} \|^{\frac{2 - 2\vartheta_{\bm{F}}}{1 - 2\vartheta_{\bm{F}}}} \r)
		\prod_{j = 1}^{r} \exp\l( \frac{x_{j}^2}{2}\s_{F_{j}}(X)^2 \r) \times
		\frac{1}{\sqrt{\log{\log{X}}}}\prod_{h = 1}^{r}\frac{d_{h} - c_{h}}{\s_{F_{h}}(X)}\r)
	\end{multline*}
	with $D_{3} = \max\{ D_{1}, D_{2} \}$.
	Using the well known formula
	\begin{align}
		\label{FTGF}
		\frac{1}{\sqrt{2\pi}}\int_{\RR}e^{-iv\xi}e^{-\eta v^2}dv
		= \frac{1}{\sqrt{2\eta}}\exp\l( -\frac{\xi^2}{4\eta} \r),
	\end{align}
	we can rewrite the above main term as
	\begin{multline*}
		\Xi_{X}(\bm{x})\l(\prod_{h = 1}^{r} \frac{\exp\l( \frac{x_{h}^2}{2}\s_{F_{h}}(X)^2 \r)}{\sqrt{2\pi}} \r)\\
		\times \int_{\RR^{r}} e^{-(v_{1}^2 + \cdots + v_{r}^2)/2}
		\bigg\{
		\prod_{h = 1}^{r}\int_{0}^{L}G\l( \frac{u}{L} \r)
		e^{2\pi i \e_{j}(h) u (x_{h}\s_{F_{h}}(X)^2 - v_{h} \s_{F_{h}}(X))}f_{c_{h}, d_{h}}(\e_{j}(h) u)\frac{du}{u}\bigg\}
		d\bm{v}.
	\end{multline*}
	Combining this with \eqref{MPnp3}, we see that integral \eqref{MPnp2} is equal to
	\begin{align*}
		 & \Xi_{X}(\bm{x})\l(\prod_{h = 1}^{r} \frac{\exp\l( \frac{x_{h}^2}{2}\s_{F_{h}}(X)^2 \r)}{\sqrt{2\pi}} \r)                \\
		 & \qquad \times \int_{\RR^{r}} e^{-(v_{1}^2 + \cdots + v_{r}^2)/2}
		\bigg\{
		\prod_{h = 1}^{r}\int_{0}^{L}G\l( \frac{u}{L} \r)
		e^{2\pi i \e_{j}(h) u (x_{h}\s_{F_{h}}(X)^2 - v_{h} \s_{F_{h}}(X))}f_{c_{h}, d_{h}}(\e_{j}(h) u)\frac{du}{u}\bigg\}
		d\bm{v}                                                                                                                    \\
		 & + O_{\bm{F}}\l( \Xi_{X}(\bm{x})\exp\l( D_{1} \| \bm{x} \|^{\frac{2 - 2\vartheta_{\bm{F}}}{1 - 2\vartheta_{\bm{F}}}} \r)
		\prod_{j = 1}^{r} \exp\l( \frac{x_{j}^2}{2}\s_{F_{j}}(X)^2 \r) \times \frac{1}{\sqrt{\log{\log{X}}}}
		\prod_{h = 1}^{r}\frac{d_{h} - c_{h}}{\s_{F_{h}}(X)} \r).
	\end{align*}
	Substituting this equation to the definition of $W_{L, \mathscr{R}}$ and using Lemma \ref{Multi_BSF} and equation \eqref{EqnubS},
	we obtain
	\begin{align}
		\label{MPnp5}
		 & \nu_{T, \bm{F}, \bm{x}}(\mathscr{R})                                                                       \\
		 & = \Xi_{X}(\bm{x})\l(\prod_{h = 1}^{r} \frac{\exp\l( \frac{x_{h}^2}{2}\s_{F_{h}}(X)^2 \r)}{\sqrt{2\pi}} \r) \\ \nonumber
		 & \qquad\times\bigg\{\int_{\RR^{r}} e^{-\frac{v_{1}^2 + \cdots + v_{r}^2}{2}}
		\bm{1}_{\mathscr{R}}\l( x_{1}\s_{F_{1}}(X)^2 - v_{1}\s_{F_{1}}(X), \dots, x_{r}\s_{F_{r}}(X)^2 - v_{r}\s_{F_{r}}(X)  \r)
		d\bm{v}
		+ E_{3} + E_{4}\bigg\},
	\end{align}
	where $E_{3}$ and $E_{4}$ satisfy
	\begin{align*}
		 & E_{3} \ll_{\bm{F}}
		\sum_{j = 1}^{r}\int_{\RR^{r}}\bigg\{ \l(
		\frac{\sin(\pi L (x_{j}\s_{F_{j}}(X)^2 - v_{j}\s_{F_{j}}(X) - c_{j}))}
		{\pi L (x_{j}\s_{F_{j}}(X)^2 - v_{j}\s_{F_{j}}(X) - c_{j})} \r)^{2}                    \\
		 & \qqqquad+ \l(\frac{\sin(\pi L (x_{j}\s_{F_{j}}(X)^2 - v_{j}\s_{F_{j}}(X) - d_{j}))}
		{\pi L (x_{j}\s_{F_{j}}(X)^2 - v_{j}\s_{F_{j}}(X) - d_{j})} \r)^{2} \bigg\}
		\times e^{-(v_{1}^2 + \cdots + v_{r}^2)/2}d\bm{v},
	\end{align*}
	and
	\begin{align*}
		E_{4}
		\ll_{\bm{F}}
		\frac{\exp\l( C \| \bm{x} \|^{\frac{2 - 2\vartheta_{\bm{F}}}{1 - 2\vartheta_{\bm{F}}}} \r)}{(\log{\log{X}})^{\a_{\bm{F}} + \frac{1}{2}}}
		+ \frac{\exp\l( C \| \bm{x} \|^{\frac{2 - 2\vartheta_{\bm{F}}}{1 - 2\vartheta_{\bm{F}}}} \r)}{\sqrt{\log{\log{X}}}}
		\prod_{h = 1}^{r}\frac{d_{h} - c_{h}}{\s_{F_{h}}(X)}
	\end{align*}
	for some constant $C = C(\bm{F}) > 0$.
	By equation \eqref{STTBS}, it holds that, for any $\ell \in \RR$,
	\begin{align*}
		 & \int_{\RR^{r}}\l(
		\frac{\sin(\pi L (x_{j}\s_{F_{j}}(X)^2 - v_{j}\s_{F_{j}}(X) - \ell))}
		{\pi L (x_{j}\s_{F_{j}}(X)^2 - v_{j}\s_{F_{j}}(X) - \ell)} \r)^{2}
		\times e^{-(v_{1}^2 + \cdots + v_{r}^2)/2}d\bm{v}                                                                         \\
		 & = \frac{2}{L^2}\Re\int_{0}^{L}(L - \a)\l(\int_{\RR^{r}}e^{2\pi i (x_{j}\s_{F_{j}}(X)^2 - v_{j}\s_{F_{j}}(X) - \ell)\a}
		e^{-(v_{1}^2 + \cdots + v_{r}^2)/2}d\bm{v}\r) d\a                                                                         \\
		 & = \frac{2(2\pi)^{(r-1)/2}}{L^2}\Re\int_{0}^{L}(L - \a)e^{2\pi i (x_{j}\s_{F_{j}}(X)^2 - \ell)\a}
		\l(\int_{\RR}e^{-2\pi i v\s_{F_{j}}(X) \a}e^{-v^2/2}dv \r)d\a,
	\end{align*}
	which, by \eqref{FTGF},  becomes
	\begin{align*}
		 & = \frac{2(2\pi)^{r/2}}{L^2}\Re\int_{0}^{L}(L - \a)e^{2\pi i (x_{j}\s_{F_{j}}(X)^2 - \ell)\a}
		\exp\l( -2 \pi^2 \a^2 \s_{F_{j}}(X)^{2} \r)d\a                                                  \\
		 & \ll_{\bm{F}} \frac{1}{L \s_{F_{j}}(X)}
		\ll_{\bm{F}} \frac{1}{(\log{\log{X}})^{\a_{\bm{F}} + \frac{1}{2}}}.
	\end{align*}
	Hence, we have $E_{3} \ll_{\bm{F}} \frac{1}{(\log{\log{X}})^{\a_{\bm{F}} + \frac{1}{2}}}$.
	Finally, by simple calculations, we can write
	\begin{align*}
		 & \int_{\RR^{r}} e^{-\frac{v_{1}^2 + \cdots + v_{r}^2}{2}}
		\bm{1}_{\mathscr{R}}\l( x_{1}\s_{F_{1}}(X)^2 - v_{1}\s_{F_{1}}(X), \dots, x_{r}\s_{F_{r}}(X)^2 - v_{r}\s_{F_{r}}(X)  \r)
		d\bm{v}                                                                                                                          \\
		 & =\prod_{j = 1}^{r}\int_{x_{j}\s_{F_{j}}(X) - \frac{d_{j}}{\s_{F_{j}}(X)}}^{x_{j} \s_{F_{j}}(X) - \frac{c_{j}}{\s_{F_{j}}(X)}}
		e^{-v^2/2}\frac{dv}{\sqrt{2\pi}}
	\end{align*}
	and this completes the proof of \eqref{FMnu1}.

	Next, we consider \eqref{FMnu2}.
	Using Proposition \ref{RKLJVDPP} and equation \eqref{ES_sI2}, we have
	\begin{align*}
		 & M_{T}\l( x_{1} + 2\pi i \e_{j}(1) u_{1}, \dots, x_{r} + 2 \pi i \e_{j}(r) u_{r} \r)         \\
		 & = \prod_{h = 1}^{r}\l(1 + O_{\bm{F}}\l(|x_{h} + iu_{h}|^2\r)\r)
		\exp\l( \frac{x_{h}^2 + 4 \pi i \e_{j}(h) x_{h} u_{h} - 4 \pi^2 u_{h}^2}{2}\s_{F_{h}}(X)^2 \r) \\
		 & \qqqquad+ O_{\bm{F}}\l( \exp\l( -b_{1}(\log{\log{X}})^{4(r + 1)} \r) \r)
	\end{align*}
	when $\| \bm{x} \|$, $\| \bm{u} \|$ are sufficiently small.
	By using this equation, we can prove \eqref{FMnu2} similarly to the proof of \eqref{FMnu1}.
\end{proof}

\begin{proof}[Proof of Proposition \ref{Main_Prop_JVD}]
	We firstly prove Proposition \ref{Main_Prop_JVD} in the case $V_{j}$'s are nonnegative.
	Let $\bm{x} = (x_{1}, \dots, x_{r}) \in (\RR_{> 0})^{r}$ satisfying
	$\|\bm{x}\| \leq b_{4}$ with $b_{4}$ the same number as in Lemma \ref{FMnu}.
	By Lemma \ref{SPKLH} and equation \eqref{FMnu2}, we have
	\begin{align*}
		 & \mu_{T, \bm{F}}((y_{1}, \infty) \times \cdots \times (y_{r}, \infty))                       \\
		 & = \prod_{j = 1}^{r}e^{\frac{x_{j}^2}{2}\s_{F_{j}}(X)^2}\int_{x_{j} y_{j}}^{\infty}e^{-\tau}
		\int_{x_{j}\s_{F_{j}}(X) - \frac{\tau / x_{j}}{\s_{F_{j}}(X)}}^{x_{j}\s_{F_{j}}(X) - \frac{y_{j}}{\s_{F_{j}}(X)}}
		e^{-v^2 / 2}\frac{dv}{\sqrt{2\pi}} d\tau                                                       \\
		 & \qqquad+ O_{\bm{F}}\l( \frac{1}{(\log{\log{X}})^{\a_{\bm{F}} + \frac{1}{2}}}
		\prod_{j = 1}^{r}e^{\frac{x_{j}^2}{2}\s_{F_{j}}(X)^2 -x_{j} y_{j}}
		+ E \times \prod_{j = 1}^{r}e^{\frac{x_{j}^2}{2}\s_{F_{j}}(X)^2} \r),
	\end{align*}
	where
	\begin{align*}
		 & E =
		\sum_{k = 1}^{r}\int_{x_{r} y_{r}}^{\infty} \cdots \int_{x_{1} y_{1}}^{\infty}
		e^{-(\tau_{1} + \cdots + \tau_{r})}
		\l(\frac{x_{k}^2 (\tau_{k}/x_{k} - y_{k})}{\s_{F_{k}}(X)} + \frac{1}{\s_{F_{k}}(X)^2} \r)
		\prod_{\substack{h = 1 \\ h \not= k}}^{r}\frac{\tau_{h} / x_{h} - y_{h}}{\s_{F_{h}}(X)}d\tau_{1} \cdots d\tau_{r}.
	\end{align*}
	Now, simple calculations lead that
	\begin{align*}
		\int_{x_{j} y_{j}}^{\infty}e^{-\tau}
		\int_{x_{j}\s_{F_{j}}(X) - \frac{\tau / x_{j}}{\s_{F_{j}}(X)}}^{x_{j}\s_{F_{j}}(X) - \frac{y_{j}}{\s_{F_{j}}(X)}}
		e^{-v^2 / 2}\frac{dv}{\sqrt{2\pi}}d\tau
		= \exp\l( -\frac{x_{j}^2}{2}\s_{F_{j}}(X)^2 \r)\int_{V_{j}}^{\infty}e^{-u^2/2}\frac{du}{\sqrt{2\pi}}
	\end{align*}
	since $y_{j} = V_{j} \s_{F_{j}}(X)$.
	Therefore, we obtain
	\begin{align*}
		 & \mu_{T, \bm{F}}((y_{1}, \infty) \times \cdots \times (y_{r}, \infty))     \\
		 & =\prod_{j = 1}^{r}\int_{V_{j}}^{\infty}e^{-u^2 / 2}\frac{du}{\sqrt{2\pi}}
		+ O_{\bm{F}}\l( \frac{1}{(\log{\log{X}})^{\a_{\bm{F}} + \frac{1}{2}}}
		\prod_{j = 1}^{r}e^{\frac{x_{j}^2}{2}\s_{F_{j}}(X)^2 -x_{j} y_{j}}
		+ E \prod_{j = 1}^{r}e^{\frac{x_{j}^2}{2}\s_{F_{j}}(X)^2} \r).
	\end{align*}
	Here, we decide $x_{j}$'s as $x_{j} = \max\{ 1, V_{j} \} / \s_{F_{j}}(X)$, where
	$V_{j}$'s must satisfy the inequality $V_{j} \leq R \s_{F_{j}}(X) $.
	Then, we see that
	\begin{align}
		\label{pMPJVD1}
		e^{-x_{j} y_{j}}
		= e^{-\frac{x_{j}^2}{2}\s_{F_{j}}(X)^2}e^{\frac{x_{j}^2}{2}\s_{F_{j}}(X)^2 - x_{j} y_{j}}
		\ll e^{-\frac{x_{j}^2}{2}\s_{F_{j}}(X)^2} e^{-V_{j}^2/2}.
	\end{align}
	This estimate leads that
	\begin{align*}
		\frac{1}{(\log{\log{X}})^{\a_{\bm{F}} + \frac{1}{2}}}
		\prod_{j = 1}^{r}e^{\frac{x_{j}^2}{2}\s_{F_{j}}(X)^2 -x_{j} y_{j}}
		 & \ll_{r} \frac{e^{-(V_{1}^2 + \cdots + V_{r}^2)/2}}{(\log{\log{X}})^{\a_{\bm{F}} + \frac{1}{2}}} \\
		 & \ll_{r} \frac{1}{(\log{\log{X}})^{\a_{\bm{F}} + \frac{1}{2}}}
		\prod_{j = 1}^{r}(1 + V_{j})\int_{V_{j}}^{\infty}e^{-u^2/2}\frac{du}{\sqrt{2\pi}}.
	\end{align*}
	Moreover, since it holds that
	\begin{align}     \label{BFpte}
		\int_{x_{j}y_{j}}^{\infty}\l( \frac{\tau}{x_{j}} - y_{j} \r)e^{-\tau}d\tau
		= \frac{e^{-x_{j}y_{j}}}{x_{j}},
	\end{align}
	we have
	\begin{align}
		 & \int_{x_{r} y_{r}}^{\infty} \cdots \int_{x_{1} y_{1}}^{\infty}e^{-(\tau_{1} + \cdots + \tau_{r})}
		\l(\frac{x_{k}^2 (d_{k} - c_{k})}{\s_{F_{k}}(X)} + \frac{1}{\s_{F_{k}}(X)^2} \r)
		\prod_{\substack{h = 1                                                                               \\ h \not= k}}^{r}\frac{\tau_{h} / x_{h} - y_{h}}{\s_{F_{h}}(X)}
		d\tau_{1} \cdots d\tau_{r}                                                                           \\ \nonumber
		 & = x_{k}^{2}\prod_{j = 1}^{r}
		\frac{1}{\s_{F_{j}}(X)}\int_{x_{j} y_{j}}^{\infty}\l( \frac{\tau}{x_{j}} - y_{j} \r)e^{-\tau}d\tau
		+ \frac{e^{-x_{k} y_{k}}}{\s_{F_{k}}(X)^2}\prod_{\substack{j = 1                                     \\ j \not= k}}^{r}\frac{1}{\s_{F_{j}}(X)}
		\int_{x_{j} y_{j}}^{\infty}\l( \frac{\tau}{x_{j}} - y_{j} \r)e^{-\tau}d\tau                          \\
		\label{pMPJVD2}
		 & = x_{k}^2\prod_{j = 1}^{r}\frac{e^{-x_{j} y_{j}}}{x_{j}\s_{F_{j}}(X)}
		+ \frac{x_{k}}{\s_{F_{k}}(X)}
		\prod_{j = 1}^{r}\frac{e^{-x_{j} y_{j}}}{x_{j} \s_{F_{j}}(X)}.
	\end{align}
	for every $1 \leq k \leq r$.
	By estimate \eqref{pMPJVD1} and $x_{j}\s_{F_{j}}(X) \asymp 1 + V_{j}$,
	we can write
	\begin{align*}
		\frac{e^{-x_{j} y_{j}}}{x_{j}\s_{F_{j}}(X)}
		\ll e^{-\frac{x_{j}^2}{2}\s_{F_{j}}(X)^2}\frac{e^{-V_{j}^2 / 2}}{1 + V_{j}}
		\ll e^{-\frac{x_{j}^2}{2}\s_{F_{j}}(X)^2}\int_{V_{j}}^{\infty}e^{-u^2/2}\frac{du}{\sqrt{2\pi}}.
	\end{align*}
	By this observation, \eqref{pMPJVD2} is
	\begin{align*}
		 & \ll_{r} \l(x_{k}^2 + \frac{x_{k}}{\s_{F_{k}}(X)}\r)
		\prod_{j = 1}^{r}e^{-\frac{x_{j}^2}{2}\s_{F_{j}}(X)^2}\int_{V_{j}}^{\infty}e^{-u^2/2}\frac{du}{\sqrt{2\pi}} \\
		 & \ll_{\bm{F}} \frac{1 + V_{k}^2}{\log{\log{X}}}
		\prod_{j = 1}^{r}e^{-\frac{x_{j}^2}{2}\s_{F_{j}}(X)^2}\int_{V_{j}}^{\infty}e^{-u^2/2}\frac{du}{\sqrt{2\pi}}.
	\end{align*}
	Hence, we have
	\begin{align*}
		E \prod_{j = 1}^{r}e^{\frac{x_{j}^2}{2}\s_{F_{j}}(X)^2}
		\ll_{\bm{F}} \frac{1 + \norm[]{\bm{V}}^2}{\log{\log{X}}}\prod_{j = 1}^{r}\int_{V_{j}}^{\infty}e^{-u^2/2}\frac{du}{\sqrt{2\pi}}.
	\end{align*}
	From the above estimations, we obtain
	\begin{align*}
		 & \mu_{T, \bm{F}}((y_{1}, \infty) \times \cdots \times (y_{r}, \infty))    \\
		 & = \prod_{j = 1}^{r}\int_{V_{j}}^{\infty}e^{-u^2/2}\frac{du}{\sqrt{2\pi}}
		+ O_{\bm{F}}\l( \l\{\frac{\prod_{k = 1}^{r}(1 + V_{k})}{(\log{\log{X}})^{\a_{\bm{F}} + \frac{1}{2}}}
		+ \frac{1 + \norm[]{\bm{V}}^2}{\log{\log{X}}}\r\}
		\times \prod_{j = 1}^{r}\int_{V_{j}}^{\infty}e^{-u^2/2}\frac{du}{\sqrt{2\pi}} \r)
	\end{align*}
	for $0 \leq V_{j} \leq b_{2} \s_{F_{j}}(X)$.
	Thus, by this formula and \eqref{RTmuS}, we complete the proof of Proposition \ref{Main_Prop_JVD}
	in the case $V_{j}$'s are nonnegative.

	In order to finish the proof of Proposition \ref{Main_Prop_JVD}, we consider the negative cases.
	It suffices to show that,
	for the case $- b\s_{F_{1}}(X) \leq V_{1} \leq 0$ and $0 \leq V_{j} \leq b \s_{F_{j}}(X)$,
	\begin{align*}
		 & \frac{1}{T}\meas(\S_{X}(T, (-V_{1}, V_2, \dots, V_{r}); \bm{F}, \bm{\theta}))                              \\
		 & = \l( 1 + O_{\bm{F}}\l( \frac{\prod_{k = 1}^{r}(1 + |V_{k}|)}{(\log{\log{X}})^{\a_{\bm{F}} + \frac{1}{2}}}
			+ \frac{1 + \norm[]{\bm{V}}^2}{\log{\log{X}}} \r) \r)
		\prod_{j = 1}^{r}\int_{V_{j}}^{\infty}e^{-\frac{u^2}{2}}\frac{du}{\sqrt{2\pi}}
	\end{align*}
	since other cases can be shown similarly by induction.
	By the definition of the set $\S_{X}(T, \bm{V}; \bm{F}, \bm{\theta})$, it holds that
	\begin{multline*}
		\S_{X}(T, \bm{V}; \bm{F}, \bm{\theta})
		= \S_{X}(T, (V_{2}, \dots, V_{r}); (F_{2}, \dots, F_{r}), (\theta_{2}, \dots, \theta_{r}))\\
		\setminus \S_{X}(T, (-V_{1}-0, V_{2}, \dots, V_{r}); \bm{F}, (\pi - \theta_{1}, \theta_2, \dots, \theta_{r})),
	\end{multline*}
	where we regard that if $r = 1$, the first set on the right hand side is $[T, 2T]$.
	Therefore, from the nonnegative cases, we have
	\begin{align}
		\label{PNMPJVD1}
		 & \frac{1}{T}\meas(\S_{X}(T, (-V_{1}, V_2, \dots, V_{r}); \bm{F}, \bm{\theta}))                                         \\
		 & = \frac{1}{T}\meas(\S_{X}(T, (V_{2}, \dots, V_{r}); (F_{2}, \dots, F_{r}), (\theta_{2}, \dots, \theta_{r})))          \\
		 & \qqqquad
		-\frac{1}{T}\meas(\S_{X}(T, (-V_{1} - 0, V_{2}, \dots, V_{r}); \bm{F}, (\pi - \theta_{1}, \theta_2, \dots, \theta_{r}))) \\
		 & = \l( 1 + E_{1} \r)\prod_{j = 2}^{r}\int_{V_{j}}^{\infty}e^{-\frac{u^2}{2}}\frac{du}{\sqrt{2\pi}}
		- (1 + E_{2})\prod_{j = 1}^{r}
		\int_{|V_{j}|}^{\infty}e^{-\frac{u^2}{2}}\frac{du}{\sqrt{2\pi}}.
	\end{align}
	Here, $E_{1}$ and $E_{2}$ satisfy
	\begin{gather*}
		E_{1}
		\ll_{\bm{F}} \frac{\prod_{k = 2}^{r}(1 + V_{k})}{(\log{\log{X}})^{\a_{\bm{F}} + \frac{1}{2}}}
		+ \frac{1 + \norm[]{(V_{2}, \dots, V_{r})}^2}{\log{\log{X}}},\\
		E_{2}
		\ll_{\bm{F}} \frac{\prod_{k = 1}^{r}(1 + V_{k})}{(\log{\log{X}})^{\a_{\bm{F}} + \frac{1}{2}}}
		+ \frac{1 + \norm[]{(V_{1}, \dots, V_{r})}^2}{\log{\log{X}}}.
	\end{gather*}
	Hence, we find that \eqref{PNMPJVD1} is equal to
	\begin{align*}
		 & \l(1 - \int_{-V_{1}}^{\infty}e^{-u^2/2}\frac{du}{\sqrt{2\pi}}\r)
		\prod_{j = 2}^{r}\int_{V_{j}}^{\infty}e^{-u^2/2}\frac{du}{\sqrt{2\pi}}                                         \\
		 & \qqqquad \qqquad+ E_{1}\prod_{j = 2}^{r}\int_{V_{j}}^{\infty}e^{-\frac{u^2}{2}}\frac{du}{\sqrt{2\pi}}
		+ E_{2}\prod_{j = 1}^{r}\int_{V_{j}}^{\infty}e^{-\frac{u^2}{2}}\frac{du}{\sqrt{2\pi}}                          \\
		 & =  \l( 1 + O_{\bm{F}}\l( \frac{\prod_{k = 1}^{r}(1 + |V_{k}|)}{(\log{\log{X}})^{\a_{\bm{F}} + \frac{1}{2}}}
			+ \frac{1 + \norm[]{\bm{V}}^2}{\log{\log{X}}} \r) \r)\prod_{j = 1}^{r}\int_{V_{j}}^{\infty}e^{-u^2/2}\frac{du}{\sqrt{2\pi}}.
	\end{align*}
	Thus, we also obtain the negative cases of Proposition \ref{Main_Prop_JVD}.
\end{proof}

\begin{proof}[Proof of Proposition \ref{Main_Prop_JVD3}]
	Let $\bm{V} = (V_{1}, \dots, V_{r}) \in (\RR_{\geq 0})^{r}$ satisfying
	$\| \bm{V} \| \leq (\log{\log{X}})^{2r}$,
	and put $x_{j} = \max\{ 1, V_{j} \} / \s_{F_{j}}(X)$.
	Similarly to the proof of Proposition \ref{Main_Prop_JVD} by using \eqref{FMnu1} instead of \eqref{FMnu2}, we obtain
	\begin{align*}
		 & \mu_{T, \bm{F}}((y_{1}, \infty) \times \cdots \times (y_{r}, \infty))                                                                         \\
		 & = \Xi_{X}(\bm{x})\Bigg\{\prod_{j = 1}^{r}\int_{V_{j}}^{\infty}e^{-u^2 / 2}\frac{du}{\sqrt{2\pi}}                                              \\
		 & \qqquad+ O_{\bm{F}}\l( \exp\l( C \l(\frac{\| \bm{V} \|}{\sqrt{\log{\log{X}}}}\r)^{\frac{2 - 2\vartheta{\bm{F}}}{1 - 2\vartheta_{\bm{F}}}} \r)
		\frac{\prod_{k =1}^{r}(1 + V_{k})}{(\log{\log{X}})^{\a_{\bm{F}} + \frac{1}{2}}}
		\prod_{j = 1}^{r}\int_{V_{j}}^{\infty}e^{-u^2 / 2}\frac{du}{\sqrt{2\pi}} + E\r)\Bigg\}
	\end{align*}
	for $0 \leq V_{j} \leq (\log{\log{X}})^{2r}$, where
	\begin{align*}
		 & E =
		\exp\l( C \l(\frac{\| \bm{V} \|}{\sqrt{\log{\log{X}}}}\r)^{\frac{2 - 2\vartheta_{\bm{F}}}{1 - 2\vartheta_{\bm{F}}}} \r)
		\frac{1}{\sqrt{\log{\log{X}}}} \times \prod_{j = 1}^{r}\frac{e^{\frac{x_{j}^2}{2}\s_{F_{j}}(X)^2}}{\s_{F_{j}}(X)}
		\int_{x_{j}}^{y_{j}}\l( \frac{\tau}{x_{j}} - y_{j} \r)e^{-\tau}d\tau.
	\end{align*}
	Here, $C = C(\bm{F})$ is a positive constant.
	Moreover, using \eqref{BFpte} we have
	\begin{align*}
		E
		= \exp\l( C \l(\frac{\| \bm{V} \|}{\sqrt{\log{\log{X}}}}\r)^{\frac{2 - 2\vartheta_{\bm{F}}}{1 - 2\vartheta_{\bm{F}}}} \r)
		\frac{1}{\sqrt{\log{\log{X}}}} \prod_{j = 1}^{r}\frac{e^{\frac{x_{j}^2}{2}\s_{F_{j}}(X)^2 - x_{j} y_{j}}}{x_{j}\s_{F_{j}}(X)}
	\end{align*}
	By estimate \eqref{pMPJVD1} and $x_{j}\s_{F_{j}}(X) \asymp 1 + V_{j}$,
	we can write
	\begin{align*}
		\frac{e^{\frac{x_{j}^2}{2}\s_{F_{j}}(X)^2-x_{j} y_{j}}}{x_{j}\s_{F_{j}}(X)}
		\ll \frac{e^{-V_{j}^2 / 2}}{1 + V_{j}}
		\ll \int_{V_{j}}^{\infty}e^{-u^2/2}\frac{du}{\sqrt{2\pi}}.
	\end{align*}
	By this observation, we have
	\begin{align*}
		E
		\ll_{r} \exp\l( C \l(\frac{\| \bm{V} \|}{\sqrt{\log{\log{X}}}}\r)^{\frac{2 - 2\vartheta_{\bm{F}}}{1 - 2\vartheta_{\bm{F}}}} \r)
		\frac{1}{\sqrt{\log{\log{X}}}}\prod_{j = 1}^{r}\int_{V_{j}}^{\infty}e^{-u^2/2}\frac{du}{\sqrt{2\pi}}.
	\end{align*}
	From the above estimations, we obtain
	\begin{align*}
		 & \mu_{T, \bm{F}}((y_{1}, \infty) \times \cdots \times (y_{r}, \infty))                   \\
		 & = \Xi_{X}(\bm{x})\prod_{j = 1}^{r}\int_{V_{j}}^{\infty}e^{-u^2/2}\frac{du}{\sqrt{2\pi}} \\
		 & \qqquad\times \Bigg\{ 1
		+ O_{\bm{F}}\l( \exp\l( C \l(\frac{\| \bm{V} \|}{\sqrt{\log{\log{X}}}}\r)^{\frac{2 - 2\vartheta_{\bm{F}}}{1 - 2\vartheta_{\bm{F}}}} \r)
		\l\{\frac{\prod_{k = 1}^{r}(1 + V_{k})}{(\log{\log{X}})^{\a_{\bm{F}} + \frac{1}{2}}}
		+ \frac{1}{\sqrt{\log{\log{X}}}}\r\} \r) \Bigg\}.
	\end{align*}
	In particular, by the definition of $\Xi_{X}$ \eqref{def_Xi}, assumptions (A1), (A2),
	and the analyticity of $\Xi_{X}$ coming from Lemma \ref{Prop_Psi_FPP}, it holds that
	\begin{align*}
		\Xi_{X}(\bm{x})
		= \l( 1 + O_{\bm{F}}\l( \frac{1}{\sqrt{\log{\log{X}}}} \r) \r)
		\Xi_{X}\l( \tfrac{V_{1}}{\s_{F_{1}}(X)}, \dots, \tfrac{V_{r}}{\s_{F_{r}}(X)} \r).
	\end{align*}
	Thus, by these formulas and \eqref{RTmuS}, we obtain Proposition \ref{Main_Prop_JVD3} when $\| \bm{V} \| \leq (\log{\log{X}})^{2r}$.
\end{proof}

\subsection{Proof of Theorem \ref{GMDP}}

We prove the following lemma as a preparation.

\begin{lemma} \label{LDESRDP}
	Let $F$ be a Dirichlet series satisfying (S4), (S5), and \eqref{SNC}.
	There exists a positive constant $c_{0} = c_{0}(F)$ such that for any large numbers $X$, $T$
  with $X \leq T^{1/2}$
	and for any $\Delta > 0$, $10 \Delta \s_{F}(X)^2 + \Delta^{-1} \leq V \leq \frac{\log{(T / \log{T})}}{\Delta \log{X}}$
	\begin{align*}
		\frac{1}{T}\meas\set{t \in [T, 2T]}{P_{F}(\tfrac{1}{2} + it, X) > V}
		\ll \exp\l( -\Delta V \log\l(\frac{c_{0} V}{\Delta \s_{F}(X)^2}\r) \r).
	\end{align*}
\end{lemma}

\begin{proof}
	Similarly to the proof of \eqref{MVEDPCLPP}, we can obtain
	\begin{align*}
		\frac{1}{T}\int_{T}^{2T}|P_{F}(\tfrac{1}{2} + it, X)|^{2k}dt
		\ll (C k \s_{F}(X)^2)^{k}
	\end{align*}
	for any $1 \leq k \leq \frac{\log{(T / \log{T})}}{\log{X}}$ with $C = C(F)$ a positive constant.
	It follows from this estimate that
	\begin{align*}
		\frac{1}{T}\meas\set{t \in [T, 2T]}{P_{F}(\tfrac{1}{2} + it, X) > V}
		\ll \l( \frac{C k \s_{F}(X)^2}{V^2} \r)^{k}.
	\end{align*}
	Hence, choosing $k = \lfloor \Delta V \rfloor$, we obtain
	\begin{align*}
		\frac{1}{T}\meas\set{t \in [T, 2T]}{P_{F}(\tfrac{1}{2} + it, X) > V}
		\ll \exp\l( -\Delta V \log\l( \frac{V}{4 \Delta \s_{F}(X)^2} \r) \r)
	\end{align*}
	for any $V \leq \frac{\log{(T / \log{T})}}{\Delta \log{X}}$.
\end{proof}

\begin{proof}[Proof of Theorem \ref{GMDP}]
To make the dependency of the limit on $X$ and $T$ explicit, 
we introduce three more parameters $n_{\bm{F}}^{*}$, $\vartheta_{\bm{F}}^{*}$, and $b_{\bm{F}}^{*}$, where 
\begin{align*}
  n_{\bm{F}}^{*} = \min_{1 \leq j \leq r}n_{F_{j}},\  
  \vartheta_{\bm{F}}^{*} = \min_{1 \leq j \leq r}\vartheta_{F_{j}}, \ 
  b_{\bm{F}}^{*}
  =\min_{\substack{1\leq i\leq r\\\vartheta_{F_i}=\vartheta_{\bm F}^*} }\sup_{p \text{ prime}}\{\frac{|b_{F_i}(p)|}{p^{\vartheta_{F_i}}}  \}.
\end{align*}
  Let $\e > 0$ be any small fixed constant and let 
	$ X
		\leq \max\{c_{1}(\log{T} \log{\log{T}})^{\frac{2}{1 + 2\vartheta_{\bm{F}}^{*}}}, c_{2}\frac{(\log{T})^2 \log{\log{T}}}{\log_{3}{T}}\}$,
	where $c_{1} = c_{1}(\bm{F}, k, \e)$ and $c_{2} = c_{2}(\bm{F}, k, \e)$ are some positive constants to be chosen later.
	Denote $\phi_{\bm{F}}(t, X) = \min_{1 \leq j \leq r}\Re e^{-i\theta_{j}}P_{F_{j}}(\tfrac{1}{2} + it, X)$ and
	\begin{align}
		\Phi_{\bm{F}}(T, V, X)
		= \meas\set{t \in [T, 2T]}{\phi_{\bm{F}}(t, X) > V}.
	\end{align}
	Then we have
	\begin{align*}
		\int_{T}^{2T}\exp\l( 2k \phi_{\bm{F}}(t, X) \r)dt
		= \int_{-\infty}^{\infty}2k e^{2k V} \Phi_{\bm{F}}(T, V, X)dV.
	\end{align*}
	We divide the range of the integral on the right hand side as
	\begin{align*}
		 & I_{1} + I_{2} + I_{3} + I_{4} + I_{5}                       \\
		 & := \bigg(\int_{-\infty}^{\sqrt{\log{\log{X}}}}
		+ \int_{\sqrt{\log{\log{X}}}}^{C k \log{\log{X}}}
    + \int_{C k \log{\log{X}}}^{\frac{\log{T}}{\log{\log{T}}}} + \int_{\frac{\log{T}}{\log{\log{T}}}}^{c_{3} \log{T}} 
    + \int_{c_{3} \log{T}}^{\infty}\bigg) 2k e^{2k V} \Phi_{\bm{F}}(T, V, X)dV
	\end{align*}
	with $c_{3} = \frac{1}{4k + \sqrt{\e} / 4}$ and $C = C(\bm{F})$ a suitably large constant.
	The trivial bound $\Phi_{\bm{F}}(T, V, X) \leq T$ yields the inequality $I_{1} \leq T e^{2k\sqrt{\log{\log{X}}}}$.
  Applying Lemma \ref{LDESRDP} with $\Delta = 4k$, we have $I_{3} \ll T$ when $C$ is suitably large depending only on $\bm{F}$.
  Similarly, applying Lemma \ref{LDESRDP} with $\Delta = \frac{2k + \sqrt{\e}/10}{\log{\log{T}}}$, we have $I_{4} \ll_{k, \e} T$.

  Next, we consider $I_{5}$.
  We have
	\begin{align*}
		\phi_{\bm F}(t, X)\leq \min_{1\leq j\leq r} |P_{F_{j}}(\tfrac{1}{2} + it, X)|
		\leq\min_{1\leq j\leq r} \sum_{p \leq X}\frac{|a_{F_{j}}(p)|}{p^{1/2}}
		+ \max_{1\leq j\leq r} \sum_{2 \leq \ell \leq \frac{\log{X}}{\log{2}}}\sum_{p \leq X^{1/\ell}}\frac{|b_{F_{j}}(p^{\ell})|}{p^{\ell/2}}.
	\end{align*}
	By the Cauchy-Schwarz inequality, the prime number theorem, and assumption \eqref{SNC}, the first sum is
	\begin{align*}
		\leq \min_{1\leq j\leq r} \l( \sum_{p \leq X}1 \r)^{1/2} \l( \sum_{p \leq X}\frac{|a_{F_{j}}(p)|^2}{p} \r)^{1/2}
    = (1 + o(1))\sqrt{\frac{X}{\log{X}} \times n_{\bm F}^* \log{\log{X}}}.
  \end{align*}
	On the other hand, we also obtain, using the prime number theorem and partial summation, that the first sum is
	\begin{align*}
		\leq \min_{1\leq j\leq r}\sup_{p}\frac{|b_{F_{j}}(p)|}{p^{\vartheta_{F_j}}}\sum_{p \leq X}\frac{p^{\vartheta_{F_{j}}}}{p^{1/2}}
		=  (1 + o(1))\frac{2b_{\bm F}^*}{1 + 2\vartheta_{\bm F}^*}\frac{X^{1/2 + \vartheta_{\bm F}^*}}{\log{X}}
	\end{align*}
	By the assumption (S5) and the Cauchy-Schwarz inequality, the second sum is
	\begin{align*}
		\ll_{F_{j}} \sum_{2 \leq \ell \leq \frac{\log{X}}{\log{2}}}\l( \sum_{p \leq X^{1/\ell}}\frac{1}{(\log{p^{\ell}})^2} \r)^{1/2}
		\l( \sum_{p \leq X^{1/\ell}}\frac{|b_{F_{j}}(p^{\ell}) \log{p^{\ell}}|^2}{p^{\ell}} \r)^{1/2}
		\ll X^{1/4}.
	\end{align*}
  Therefore, we have
  \begin{align*}
    \phi_{\bm{F}}(t, X)
    \leq (1 + o(1))
    \min\l\{ \frac{2b_{\bm F}^*}{1 + 2\vartheta_{\bm F}^*}\frac{X^{1/2 + \vartheta_{\bm{F}}^{*}}}{\log{X}}, \sqrt{n_{\bm{F}}^{*} \frac{X \log{\log{X}}}{\log{X}}} \r\}.
  \end{align*}
  Choosing $c_{1} = (b_{\bm{F}}^{*}(4k + \sqrt{\e}/2)^{-2 / (1 + 2\vartheta_{\bm{F}}^{*})}$ 
  and $c_{2} = ( (8k^2  + \e)n_{\bm{F}}^{*})^{-1}$, 
  we have $\Phi_{\bm{F}}(T, V, X) = 0$ for $V \geq c_{3}\log{T}$ and thus $I_{5} = 0$.

  Finally, we consider $I_{2}$.
	From Proposition \ref{Main_Prop_JVD3}, it holds that
	\begin{align*}
		\Phi_{\bm{F}}(T, V, X)
		= T(1 + E_{1}) \times \Xi_{X}\l( \tfrac{V}{\s_{F_{1}}(X)^2}, \dots, \tfrac{V}{\s_{F_{r}}(X)^2} \r)
		\prod_{j = 1}^{r}\int_{V / \s_{F_{j}}(X)}^{\infty}e^{-u^2/2}\frac{du}{\sqrt{2\pi}}
	\end{align*}
	for $\sqrt{\log{\log{X}}} \leq V \leq C k \log{\log{X}}$.
	Here, $E_{1}$ satisfies the estimate
	\begin{align*}
		E_{1} \ll_{\bm{F}, k}
		\frac{V^{r}}{(\log{\log{X}})^{\a_{\bm{F}} + \frac{r + 1}{2}}}
		+ \frac{1}{\sqrt{\log{\log{X}}}}.
	\end{align*}
	We also have
	\begin{align*}
		\int_{x}^{\infty}e^{-u^2/2}\frac{du}{\sqrt{2\pi}}
		= \frac{e^{-x^{2}/2}}{\sqrt{2\pi}x}\l(1 + O_{\e}\l(\frac{1}{x^2} \r)\r)
	\end{align*}
	for $x \geq \e$ with any $\e > 0$.
	Therefore, we obtain
	\begin{align*}
		 & \Phi_{\bm{F}}(T, V, X)                                                                                       \\
		 & = T\l(\prod_{j = 1}^{r}\s_{F_{j}}(X)\r)\frac{1 + E_{1} + O_{\bm{F}}(\log{\log{X}} / V^2)}{(2\pi)^{r/2}V^{r}}
		\times \Xi_{X}\l( \tfrac{V}{\s_{F_{1}}(X)^2}, \dots, \tfrac{V}{\s_{F_{r}}(X)^2} \r)
		\exp\l( -H_{\bm{F}}(X)^{-1}V^2 \r)
	\end{align*}
	for $\sqrt{\log{\log{X}}} \leq V \leq C k \log{\log{X}}$.
	Hence, we have
	\begin{align} \label{PGMDP1}
		 & 2k e^{2kV}\Phi_{\bm{F}}(T, V, X)                                                                           \\
		 & = T 2k \l( \prod_{j = 1}^{r}\s_{F_{j}}(X) \r)\frac{1 + E_{1} + O(\log{\log{X}} / V^2)}{(2\pi)^{r/2} V^{r}}
		\exp\l( k^{2}H_{\bm{F}}(X) \r)                                                                                \\
		 & \qqqquad \times \Xi_{X}\l( \tfrac{V}{\s_{F_{1}}(X)^2}, \dots, \tfrac{V}{\s_{F_{r}}(X)^2} \r)
		\exp\l( -\frac{1}{H_{\bm{F}}(X)}\l( V - k H_{\bm{F}}(X) \r)^2 \r)
	\end{align}
	for $\sqrt{\log{\log{X}}} \leq V \leq C k \log{\log{X}}$.
	Put $H = D\sqrt{\log{\log{X}} \log_{3}{X}}$ with $D = D(\bm{F})$ a suitably large constant.
	Here, we choose $X_{0}(\bm{F}, k)$ large such that $k H_{\bm{F}}(X) \geq 2H$ if necessary.
	We write
	\begin{align*}
		I_{2}
		= \l(\int_{\sqrt{\log{\log{X}}}}^{kH_{\bm{F}}(X) - H} + \int_{kH_{\bm{F}}(X) - H}^{kH_{\bm{F}}(X) + H}
		+ \int_{kH_{\bm{F}}(X) + H}^{C k \log{\log{X}}} \r)2ke^{2kV}\Phi_{\bm{F}}(T, V, X)dV
		=: I_{2, 1} + I_{2, 2} + I_{2, 3}.
	\end{align*}
	Applying \eqref{PGMDP1}, we find that
	\begin{align*}
		I_{2, 1}
		 & \ll_{\bm{F}, k} T \exp\l( k^2 H_{\bm{F}}(X) \r)
		\int_{\sqrt{\log{\log{X}}}}^{kH_{\bm{F}}(X) - H} \exp\l( -\frac{1}{H_{\bm{F}}(X)}\l( V - kH_{\bm{F}}(X) \r)^2 \r)dV \\
		 & \leq T \exp\l( k^2 H_{\bm{F}}(X) \r)\int_{H}^{\infty}\exp\l( -\frac{V^2}{H_{\bm{F}}(X)} \r)dV
		\ll T \frac{H_{\bm{F}}(X)}{H}\exp\l( k^2 H_{\bm{F}}(X) - \frac{H^2}{H_{\bm{F}}(X)} \r),
	\end{align*}
	and that
	\begin{align*}
		I_{2, 3}
		 & \ll_{\bm{F}, k} T \frac{\exp(k^2 H_{\bm{F}}(X))}{(\log{\log{X}})^{r/2}}
		\int_{kH_{\bm{F}}(X) + H}^{C k \log{\log{X}}}\exp\l( -\frac{1}{H_{\bm{F}}(X)}\l( V - k H_{\bm{F}}(X) \r)^2 \r) dV \\
		 & \leq T \frac{\exp(k^2 H_{\bm{F}}(X))}{(\log{\log{X}})^{r/2}}
		\int_{H}^{\infty}\exp\l( -\frac{V^2}{H_{\bm{F}}(X)} \r)dV
		\ll T \frac{H_{\bm{F}}(X)}{H (\log{\log{X}})^{r/2}}\exp\l( k^2 H_{\bm{F}}(X) - \frac{H^2}{H_{\bm{F}}(X)} \r).
	\end{align*}
	Hence, we have
	\begin{align*}
		I_{2, 1} + I_{2, 3}
		\ll_{\bm{F}, k} T\frac{\exp\l( k^2 H_{\bm{F}}(X)\r)}{(\log{\log{X}})^{2r}}
	\end{align*}
	Our main term comes from $I_{2, 2}$, which is equal to
	\begin{multline*}
		2k \frac{\exp\l( k^2 H_{\bm{F}}(X) \r)}{(\sqrt{2\pi} k H_{\bm{F}}(X))^{r}} \l( \prod_{j = 1}^{r}\s_{F_{j}}(X)\r)
		\l( 1 + O_{\bm{F}}\l( \frac{1}{(\log{\log{X}})^{\a_{\bm{F}} - \frac{r - 1}{2}}}
			+ \frac{1}{\sqrt{\log{\log{X}}}} + \frac{H}{H_{\bm{F}}(X)} \r) \r)\\
		\times \Xi_{X}\l( \frac{k H_{\bm{F}}(X) + O(H)}{\s_{F_{1}}(X)^2}, \dots, \frac{k H_{\bm{F}}(X) + O(H)}{\s_{F_{r}}(X)^2} \r)
		\int_{k H_{\bm{F}}(X) - H}^{k H_{\bm{F}}(X) + H}\exp\l( -\frac{1}{H_{\bm{F}}(X)}(V - kH_{\bm{F}}(X))^2 \r).
	\end{multline*}
	The analyticity of $\Xi_{X}$ yields that
	\begin{align*}
		 & \Xi_{X}\l( \frac{k H_{\bm{F}}(X) + O(H)}{\s_{F_{1}}(X)^2}, \dots, \frac{k H_{\bm{F}}(X) + O(H)}{\s_{F_{r}}(X)^2} \r) \\
		 & = \l( 1 + O_{k, \bm{F}}\l( \frac{H}{\log{\log{X}}} \r) \r)
		\Xi_{X}\l( \frac{k H_{\bm{F}}(X)}{\s_{F_{1}}(X)^2}, \dots, \frac{k H_{\bm{F}}(X)}{\s_{F_{r}}(X)^2} \r).
	\end{align*}
	We also see that
	\begin{align*}
		\int_{k H_{\bm{F}}(X) - H}^{k H_{\bm{F}}(X) + H}\exp\l( -\frac{1}{H_{\bm{F}}(X)}(V - kH_{\bm{F}}(X))^2 \r)
		 & = \sqrt{\pi H_{\bm{F}}(X)} + O\l( \int_{H}^{\infty}\exp\l( -\frac{V^2}{H_{\bm{F}}(X)} \r) \r)        \\
		 & = \sqrt{\pi H_{\bm{F}}(X)} + O\l( \frac{H_{\bm{F}}(X)}{H}\exp\l( -\frac{H^2}{H_{\bm{F}}(X)} \r) \r).
	\end{align*}
	Therefore, we have
	\begin{align*}
		I_{2, 2}
		= T\frac{\exp\l( k^2 H_{\bm{F}}(X) \r) \prod_{j = 1}^{r}\s_{F_{j}}(X)}{\l( \sqrt{2\pi} k H_{\bm{F}}(X) \r)^{r-1}
		\sqrt{\frac{1}{2}H_{\bm{F}}(X)}}
		\Xi_{X}\l( \frac{k H_{\bm{F}}(X)}{\s_{F_{1}}(X)^2}, \dots, \frac{k H_{\bm{F}}(X)}{\s_{F_{r}}(X)^2} \r)\l( 1 + E_{2} \r),
	\end{align*}
	where
	\begin{align*}
		E_{2}
		\ll_{\bm{F}, k}
		\frac{1}{(\log{\log{X}})^{\a_{\bm{F}} - \frac{r - 1}{2}}}
		+ \frac{\log_{3}{X}}{\sqrt{\log{\log{X}}}}.
	\end{align*}
	From the above estimates of $I_{1}$, $I_{2}$, and $I_{3}$, we obtain Theorem \ref{GMDP}.
\end{proof}

\begin{proof}[Proof of Corollary \ref{jointDirichlet}]
  Assume the Selberg Orthonormality Conjecture, \eqref{SSOC} and $\vartheta_{\bm{F}} < \frac{1}{r + 1}$.
  Note that the error term $E$ in Theorem \ref{GMDP} is $= o(1)$ as $X \rightarrow + \infty$ when $\vartheta_{\bm{F}} < \frac{1}{r + 1}$.
  Hence, it suffices to show that for $k > 0$
  \begin{align*}
    &\frac{\exp\l( k^2 H_{\bm{F}}(X) \r) \prod_{j = 1}^{r}\s_{F_{j}}(X)}{\l( \sqrt{2\pi} k H_{\bm{F}}(X) \r)^{r-1}
		\sqrt{\frac{1}{2}H_{\bm{F}}(X)}}
    \Xi_{X}\l( \frac{k H_{\bm{F}}(X)}{\s_{F_{1}}(X)^2}, \dots, \frac{k H_{\bm{F}}(X)}{\s_{F_{r}}(X)^2} \r)
    \sim C(\bm{F}, k, \bm{\theta})\frac{(\log{X})^{k^2 / h_{\bm{F}}}}{(\log{\log{X}})^{(r - 1)/2}}
  \end{align*}
  as $X \rightarrow + \infty$.
  Using \eqref{SSOC}, (S4), (S5), we find that
  \begin{align} \label{pGGHK1}
    \s_{F_{j}}(X)^2
    = \frac{n_{F_{j}}}{2}\log{\log{X}} + c_{j} + o(1)
  \end{align}
  as $X \rightarrow + \infty$. Here, $c_{j}$ is a constant depending only on $\bm{F}$.
  From this equation, we also see that $\s_{F_{j}}(X) = (1 + o(1))\sqrt{\frac{n_{F_{j}}}{2}\log{\log{X}}}$.
  Moreover, it holds that as $X \rightarrow + \infty$
  \begin{align} \label{pGGHK2}
    H_{\bm{F}}(X)
    = \frac{\prod_{j = 1}^{r}(n_{F_{j}}\log{\log{X}} + 2c_{j})}{\sum_{h = 1}^{r}\us{j \not= h}{\prod_{j = 1}^{r}}(n_{F_{j}}\log{\log{X}} + 2c_{j})}
    + o(1)
    = \frac{\log{\log{X}}}{h_{\bm{F}}} + g_{\bm{F}} + o(1)
  \end{align}
  for a constant $g_{\bm{F}}$.
  Hence, we have
  \begin{align*}
    \frac{\exp\l( k^2 H_{\bm{F}}(X) \r) \prod_{j = 1}^{r}\s_{F_{j}}(X)}{\l( \sqrt{2\pi} k H_{\bm{F}}(X) \r)^{r-1}
    \sqrt{\frac{1}{2}H_{\bm{F}}(X)}}
    \sim C_{1}(\bm{F}, k)\frac{(\log{X})^{k^2 / h_{\bm{F}}}}{(\log{\log{X}})^{\frac{r - 1}{2}}}
  \end{align*}
  as $X \rightarrow + \infty$.
  Also, by \eqref{SSOC}, (S4), and (S5), we can prove that, for $j_{1} \not= j_{2}$,
  \begin{align*}
    \tau_{j_{1}, j_{2}}(X) = c_{j_{1}, j_{2}} + o(1)
  \end{align*}
  as $X \rightarrow + \infty$, where $c_{j_{1}, j_{2}}$ is a constant depending only on $\bm{F}$, $\bm{\theta}$.
  From this equation and the definition of $\Psi_{\bm{F}}$ \eqref{def_Psi_F}, we can write
  \begin{align*}
    &\Xi_{X}\l( \frac{k H_{\bm{F}}(X)}{\s_{F_{1}}(X)^2}, \dots, \frac{k H_{\bm{F}}(X)}{\s_{F_{r}}(X)^2} \r)\\
    &= (1 + o(1))\exp\l( k^2\sum_{1 \leq j_{1} < j_{2} \leq r}\frac{H_{\bm{F}}(X)^2}{(\s_{F_{j_{1}}}(X) \s_{F_{j_{2}}}(X))^2}c_{j_{1}, j_{2}} \r)
    \Psi_{\bm{F}}\l( \frac{k H_{\bm{F}}(X)}{\s_{F_{1}}(X)^2}, \dots, \frac{k H_{\bm{F}}(X)}{\s_{F_{r}}(X)^2} \r)
  \end{align*}
  as $X \rightarrow + \infty$.
  Equations \eqref{pGGHK1}, \eqref{pGGHK2} yield that for some constant $C_{2}(\bm{F}, k, \bm{\theta})$
  \begin{align*}
    \exp\l( k^2\sum_{1 \leq j_{1} < j_{2} \leq r}\frac{H_{\bm{F}}(X)^2}{(\s_{F_{j_{1}}}(X) \s_{F_{j_{2}}}(X))^2}c_{j_{1}, j_{2}} \r)
    \sim C_{2}(\bm{F}, k, \bm{\theta}).
  \end{align*}
  Moreover, it follows from the analyticity shown in Lemma \ref{Prop_Psi_FPP} and equations \eqref{pGGHK1}, \eqref{pGGHK2} that
  \begin{align*}
    \Psi_{\bm{F}}\l( \frac{k H_{\bm{F}}(X)}{\s_{F_{1}}(X)^2}, \dots, \frac{k H_{\bm{F}}(X)}{\s_{F_{r}}(X)^2} \r)
    &= \Psi_{\bm{F}}\l( \frac{2k}{n_{F_{1}}h_{\bm{F}}} + o(1), \dots, \frac{2k}{n_{F_{r}}h_{\bm{F}}} + o(1) \r)\\
    &\sim \Psi_{\bm{F}}\l( \frac{2k}{n_{F_{1}}h_{\bm{F}}}, \dots, \frac{2k}{n_{F_{r}}h_{\bm{F}}} \r).
  \end{align*}
  Thus, taking 
  $C(\bm{F}, k, \bm{\theta}) = C_{1}(\bm{F}, k) \cdot C_{2}(\bm{F}, k, \bm{\theta}) 
  \cdot \Psi_{\bm{F}}\l( \frac{2k}{n_{F_{1}}h_{\bm{F}}}, \dots, \frac{2k}{n_{F_{r}}h_{\bm{F}}} \r)$, we obtain \eqref{GGHK}.
\end{proof}

\section{\textbf{Proof of Theorems \ref{Main_Thm_LD_JVD}--Theorem \ref{New_MVT}}} \label{uncond}

\begin{lemma}{\label{diff}}
	Under the same situation as in Proposition \ref{KLI}.
	Let $r \in \ZZ_{\geq 1}$ be given.
	There exists a positive constant $A_{4} = A_{4}(F, r)$ such that
	for $X = T^{1 / (\log{\log{T}})^{4(r + 1)}}$, $Y = T^{1 / k}$, $k \in \ZZ_{\geq 1}$
	with $k \leq (\log{\log{T}})^{{4}(r + 1)}$,
	\begin{align*}
		 & \frac{1}{T}\int_{T}^{2T}\bigg| \log{F(\tfrac{1}{2} +it)} - P_{F}(\tfrac{1}{2}+it, X)
		- \sum_{|\frac{1}{2}+it-\rho_{F}| \leq \frac{1}{\log{Y}}}\log{((\tfrac{1}{2}+it-\rho_{F})\log{Y})} \bigg|^{2k}dt \\
		 & \leq A_{4}^{k} k^{2k} + A_{4}^{k}k!(\log_3T)^{k},
	\end{align*}
	and
	\begin{align*}
		\frac{1}{T}\int_{T}^{2T}\bigg| \log{F(\tfrac{1}{2} +it)} - P_{F}(\tfrac{1}{2}+it, X) \bigg|^{2k}dt
		\leq A_{4}^{k} k^{4k} + A_{4}^{k}k!(\log_3T)^{k}.
	\end{align*}
\end{lemma}

\begin{proof}
	By Proposition \ref{KLI}, it suffices to show that
	\begin{align}	\label{diffp1}
		\sum_{X < p \leq Y}\frac{|a_{F}(p)|^2}{p}
		\ll_{F, r} \log_{3}T,
	\end{align}
	where $Y = T^{1/ k}$.
	Using formula \eqref{SNC}, we find that
	\begin{align*}
		\sum_{X < p < Y}\frac{|a_{F}(p)|^2}{p}
		= n_{F}\l( \log{\log{Y^2}} - \log{\log{X}} \r) + O_{F}(1)
		\ll_{F, r} \log_3T.
	\end{align*}
	Thus, we obtain estimate \eqref{diffp1}.
\end{proof}

\noindent
\textit{Proof of Theorem \ref{Main_Thm_LD_JVD}}.
We consider \eqref{Main_Thm_LD_JVD1} and \eqref{Main_Thm_LD_JVD2}--\eqref{Main_Thm_LD_JVD3} separately.
\begin{proof}[Proof of \eqref{Main_Thm_LD_JVD1}]
	Let $T$ be large.
	Put $X = T^{1 / (\log{\log{T}})^{4(r + 1)}}$.
	Let $A \geq 1$ be a fixed arbitrary constant.
	Let the set $\mathcal{E}_j$ be
	\begin{align*}
		\mathcal{E}_j:=\set{t \in [T, 2T]}{\bigg|\log F_{j}(\tfrac{1}{2} + it) - P_{F_{j}}(\tfrac{1}{2} + it, X)\bigg|
			\geq \mathcal{L}}.
	\end{align*}
	From Lemma \ref{diff}, we have
	\begin{align}
		\meas(\mathcal{E}_j)
		\ll T\mathcal{L}_{j}^{-2k}A_{5}^{k}(k^{4k} + k^{k}(\log_3 T)^k)
	\end{align}
	for all $j$ with $A_{5} := A_{5}(\bm{F}) = \max_{1 \leq j \leq r}A_{4}(F_{j}, r)+1$,
	where $A_{4}(F_{j}, r)$ has the same meaning as in Lemma \ref{diff}.
	Here the parameter $\mathcal{L}$ satisfying $\mathcal{L} \geq (2A_5\log_{3}{T})^{2/3}$ will be chosen later.
	Set $k=\lfloor \mathcal{L}^{1/2}/e A_{5}^{1/4}\rfloor$ so that
	$\operatorname{meas}(\mathcal{E}_j)\ll T \exp(-c_{1}\mathcal{L}^{1/2})$ for some $c_1>0$.
	Therefore except on the set $\mathcal{E}:=\bigcup_{j=1}^r\mathcal{E}_j$ with measure $O_r(T \exp(-c_{1}\mathcal{L}^{1/2}))$, we have
	\begin{align}\label{logFjVj}
		\Re e^{-i\theta_j}\log F_j(\tfrac{1}{2}+it)
		=\Re e^{-i\theta_j}P_{F_{j}}(\tfrac{1}{2} + it, X) + \beta_j(t)\mathcal{L}
	\end{align}
	with $|\beta_j(t)|\leq 1$ for all $j=1, \dots, r$.
	By \eqref{logFjVj} and Proposition \ref{Main_Prop_JVD}, the measure of $t\in[T, 2T]\backslash \mathcal{E}$ such that
	\begin{align}\label{large dev}
		\Re e^{-i\theta_j}\log F_j(\tfrac{1}{2}+it)
		\geq V_{j} \sqrt{\frac{n_{F_{j}}}{2}\log{\log{T}}}, \; \forall j=1, \dots, r
	\end{align}
	is at least (since $\beta_j(t)\leq 1$
	\begin{multline}\label{lower meas}
		T\l( 1 + O_{\bm{F}}\l(
			\frac{\prod_{j = 1}^{r}(1 + |V_{k}| + \frac{\mathcal{L}}{\sqrt{\log{\log{T}}}})}{ (\log{\log{T}})^{\alpha_F+\frac{1}{2}}}
			+ \frac{1 + \norm[]{\bm{V}}^{2}+\frac{\mathcal{L}^2}{\log{\log{T}}}}{\log{\log{T}}}\r) \r)\\
		\times \prod_{j = 1}^{r}\int_{\s_{F_{j}}(X)^{-1}(V_{j}\sqrt{(n_{F_{j}}/2)\log{\log{T}}} + \mathcal{L})}^{\infty}
		e^{-u^2/2}\frac{du}{\sqrt{2\pi}}
	\end{multline}
	for $\frac{\mathcal{L}}{\sqrt{\log{\log{T}}}}, \norm[]{\bm{V}}
		\leq c\sqrt{\log{\log{T}}}$ with $c$ sufficiently small.
	Similarly, the measure of $t\in [T, 2T]\backslash \mathcal{E}$ such that \eqref{large dev} holds is at most
	\begin{multline}\label{upper meas}
		T\l( 1 + O_{\bm{F}}\l(
			\frac{\prod_{k = 1}^{r}(1 + |V_{k}| + \frac{\mathcal{L}}{\sqrt{\log{\log{T}}}})}{(\log{\log{T}})^{\a_{\bm{F}} + \frac{1}{2}}}
			+ \frac{1 + \norm[]{\bm{V}}^{2}+\frac{\mathcal{L}^2}{\log{\log{T}}}}{\log{\log{T}}}\r) \r)\\
		\times \prod_{j = 1}^{r}\int_{\s_{F_{j}}(X)^{-1}(V_{j}\sqrt{(n_{F_{j}}/2)\log{\log{T}}} - \mathcal{L})}^{\infty}
		e^{-u^2/2}\frac{du}{\sqrt{2\pi}}
	\end{multline}
	for $\frac{\mathcal{L}}{\sqrt{\log{\log{T}}}}, \norm[]{\bm{V}}
		\leq c\sqrt{\log{\log{T}}}$.
	By using equation \eqref{SNC}, we find that
	\begin{align*}
		\s_{F_{j}}(X)
		= \sqrt{\frac{n_{F_{j}}}{2} \log{\log{T}}} + O_{r, F_{j}}\l( \frac{\log_{3}{T}}{\sqrt{\log{\log{T}}}} \r),
	\end{align*}
	and so we also have
	\begin{align}
		\label{sigaminv}
		\s_{F_{j}}(X)^{-1}
		= \frac{1}{\sqrt{\frac{n_{F_{j}}}{2} \log{\log{T}}}} \l(1+ O_{r, F_{j}}\l( \frac{\log_{3}{T}}{\log{\log{T}}} \r)\r).
	\end{align}
	Therefore, when $|V_{j}| \leq \sqrt{\frac{\log{\log{T}}}{\log_{3}{T}}}$,
	$(|V_{j}| + 1)\mathcal{L} \leq B_{1}\sqrt{\log{\log{T}}}$ with $B_{1} > 0$ a constant to be chosen later,
	we find that
	\begin{align*}
		 & \int_{\s_{F_{j}}(X)^{-1}(V_{j}\sqrt{(n_{F_{j}}/2)\log{\log{T}}} \pm \mathcal{L})}^{\infty}
		e^{-u^2/2}\frac{du}{\sqrt{2\pi}}                                                              \\
		 & = \int_{V_{j}}^{\infty}e^{-u^2/2}\frac{du}{\sqrt{2\pi}}
		+ \int_{V_{j}+ O_{r, F_{j}}( |V_{j}|\frac{\log_{3}{T}}{\log{\log{T}}} + \frac{\mathcal{L}}{\sqrt{\log{\log{T}}}})}^{V_{j}}
		e^{-u^2/2}\frac{du}{\sqrt{2\pi}}                                                              \\
		 & = \int_{V_{j}}^{\infty}e^{-u^2/2}\frac{du}{\sqrt{2\pi}}
		+ O_{r, F_{j}, B_{1}}\l( \l(\frac{|V_{j}| \log_{3}{T}}{\log{\log{T}}} + \frac{\mathcal{L}}{\sqrt{\log{\log{T}}}} \r)
		e^{-V_{j}^2/2}\r)                                                                             \\
		 & = \l( 1 + O_{r, F_{j}, B_{1}}
		\l( \frac{|V_{j}|(|V_{j}| + 1)\log_{3}{T} + \mathcal{L}(|V_{j}| + 1)\sqrt{\log{\log{T}}}}{\log{\log{T}}} \r) \r)
		\int_{V_{j}}^{\infty}e^{-u^2 / 2}\frac{du}{\sqrt{2\pi}}.
	\end{align*}
	Hence, choosing $\mathcal{L} = 2c_{1}^{-2}r^{2}(\norm[]{\bm{V}}^{4} + (\log_{3}{T})^{2})$,
	we find that \eqref{lower meas} and \eqref{upper meas} become
	\begin{align*} 
		T \l( 1 + O_{\bm{F}, B_{1}}\l(\frac{(\norm[]{\bm{V}}^{4} + (\log_{3}{T})^2)(\norm[]{\bm{V}} + 1)}{\sqrt{\log{\log{T}}}}
			+ \frac{\prod_{k = 1}^{r}(1 + |V_{k}|)}{(\log{\log{T}})^{\a_{\bm{F}} + \frac{1}{2}}}\r)\r)
		\prod_{j = 1}^{r}\int_{V_{j}}^\infty e^{-\frac{u^2}{2}}\frac{du}{\sqrt{2\pi}},
	\end{align*}
	and
	\begin{align*}
		T \exp(-c_{1} \mathcal{L}^{1/2})
		\ll T \exp(-r(\|\bm V\|^2+\log_3T))\ll T \frac{1}{\log{\log{T}}}\prod_{j = 1}^{r}\int_{V_{j}}^{\infty}e^{-u^2/2}\frac{du}{\sqrt{2\pi}}
	\end{align*}
	when $\norm[]{\bm{V}} \leq c_{1}^{2}r^{-2}B_{1}(\log{\log{T}})^{1/10}$.
	Choosing $B_{1} = A c_{1}^{-2} r^{2}$, we have
	\begin{align*}
		 & \frac{1}{T}\operatorname{meas}\left( \bigcap_{j = 1}^{r}\set{t \in [T, 2T]}
		{\frac{\Re e^{-i\theta_{j}}\log{F_{j}(1/2 + it)}}{\sqrt{\frac{n_{F_{j}}}{2}\log{\log{T}}}} \geq V_j}\right)           \\
		 & =\l( 1 + O_{\bm{F}, A}\l(\frac{(\norm[]{\bm{V}}^{4} + (\log_{3}{T})^2)(\norm[]{\bm{V}} + 1)}{\sqrt{\log{\log{T}}}}
			+ \frac{\prod_{k = 1}^{r}(1 + |V_{k}|)}{(\log{\log{T}})^{\alpha_F+1/2}}\r)\r)\prod_{j = 1}^{r}
		\int_{V_j}^{\infty} e^{-\frac{u^2}{2}}\frac{du}{\sqrt{2\pi}}
	\end{align*}
	for $\norm[]{\bm{V}} \leq A(\log{\log{T}})^{1/10}$.
	Thus, we complete the proof of \eqref{Main_Thm_LD_JVD1}.
\end{proof}

\begin{proof}[Proof of \eqref{Main_Thm_LD_JVD2} and \eqref{Main_Thm_LD_JVD3}]
	Let $X = T^{1 / (\log{\log{T}})^{{4}(r + 1)}}$, $Y=T^{1/k}$ and
	let $\mathcal{B}_j$ be the set of $t\in [T, 2T]$ such that
	\begin{align}
		\left| \log F_j(\tfrac{1}{2}+it)-P_{F_{j}}(\tfrac{1}{2} + it, X)
		- \sum_{|1/2+it-\rho_{F_j}| \leq \frac{1}{\log Y}}\log ((\tfrac{1}{2}+it-\rho_{F_j})\log{Y})\right|
		\geq \mathcal{L}.
	\end{align}
	By Lemma \ref{diff}, we know
	\begin{align*}
		\meas(\mathcal{B}_j)
		\leq T\mathcal{L}^{-2k}A_{5}^k (k^{2k}+k^k(\log_{3}T)^k),
	\end{align*}
	where $A_5=\max_{1\leq j\leq r}{A_4(F_j,r)}+2$ and  $A_{4}(F_{j}, r)$ has the same meaning as in Lemma \ref{diff}.
	By taking $k=\lfloor \mathcal{L}/\sqrt{A_{5}e}\rfloor$, we have that
	$\meas(\mathcal{B}_j)\leq T \exp(-c_2 \mathcal{L} )$ for some $c_2>0$ as long as
	$\mathcal{L} \geq 2A_5\log_3T$.
	Therefore, it follows that for $t \in [T, 2T] \setminus \bigcup_{j=1}^r\mathcal{B}_j$
	\begin{align}
		 & \Re e^{-i\theta_{j}}  \log F_j(\tfrac{1}{2}+it)      \\
		 & =\Re e^{-i\theta_{j}}P_{F_{j}}(\tfrac{1}{2} + it, X)
		+ \sum_{|1/2+it-\rho_{F_j}|\leq \frac{1}{\log Y}}
		\Re e^{-i\theta_{j}} \log((\tfrac{1}{2}+it-\rho_{F_j})\log{Y}) + \b_j(t)\mathcal{L}
	\end{align}
	holds for all $1\leq j\leq r$ with some $|\beta_j(t)|\leq 1$.
	Let $\mathcal{C}_j$ be the set of $t\in [T, 2T]$ such that
	\begin{align}
		\sum_{|1/2-it-\rho_{F_j}|\leq \frac{1}{\log{Y}}}\Re e^{-\theta_{j}} \log((\tfrac{1}{2}+it-\rho_{F_j})\log{Y})
		\geq \mathcal{L}.
	\end{align}
	When $\theta_{j} \in [-\frac{\pi}{2}, \frac{\pi}{2}]$ and $|1/2+it-\rho_{F_j}| \leq \frac{1}{\log{Y}}$,
	we find that
	\begin{align}
		 & \Re e^{-i\theta_{j}} \log((\tfrac{1}{2}+it-\rho_{F_j})\log{Y})                                                             \\
		\label{KPlt}
		 & = \cos{\theta_{j}}\log|(\tfrac{1}{2}+it-\rho_{F_j})\log{Y}| + \sin{\theta_{j}} \arg((\tfrac{1}{2}+it-\rho_{F_j})\log{Y})
		\leq \pi.
	\end{align}
	Hence, we have
	\begin{align*}
		\sum_{|1/2+it-\rho_{F_j}|\leq \frac{1}{\log{Y}}}\Re e^{-i\theta_{j}} \log((\tfrac{1}{2}+it-\rho_{F_j})\log{Y})
		\leq \pi\sum_{|1/2+it-\rho_{F_j}|\leq \frac{1}{\log{Y}}}1,
	\end{align*}
	and thus by Lemma \ref{ESRSIZ}
	\begin{align*}
		\meas(\mathcal{C}_j)
		\leq C^k k^{2k} T \mathcal{L}^{-2k}
	\end{align*}
	for some constant $C = C(F_{j}) > 0$.
	By choosing $k = \lfloor \mathcal{L} / \sqrt{C e} \rfloor$, we have $\operatorname{meas}(\mathcal{C}_j)\leq T \exp(-c_3\mathcal{L})$ for some $c_3>0$. Now we have that the measure of
	$t\in[T, 2T] \setminus \bigcup_{j=1}^r(\mathcal{B}_j\cup \mathcal{C}_j)$ such that
	\begin{align*}
		\Re e^{-\theta_j}\log{F_{j}(\tfrac{1}{2}+it)}
		\geq V_j \sqrt{\frac{n_{F_j}}{2}\log\log T}
	\end{align*}
	is bounded by the measure of the set $t\in[T, 2T]$ such that
	\begin{align}\label{lower V-L-D}
		\Re e^{-i\theta_{j}} P_{F_{j}}(\tfrac{1}{2}+it, X)
		\geq V_j\sqrt{\frac{n_{F_j}}{2}\log\log T}- 2\mathcal{L}.
	\end{align}
	From Proposition \ref{Main_Prop_JVD}, we know \eqref{lower V-L-D} holds with measure
	\begin{multline}\label{upper bound V-L-D}
		T\l( 1 + O_{\bm{F}}\l(
			\frac{\prod_{k = 1}^{r}(1 + |V_{k}| + \frac{\mathcal{L}}{\sqrt{\log{\log{T}}}})}{ (\log{\log{T}})^{\alpha_F+\frac{1}{2}}}
			+ \frac{1 + \norm[]{\bm{V}}^{2}+\frac{\mathcal{L}^2}{\log{\log{T}}}}{\log{\log{T}}}\r) \r)\\
		\times \prod_{j = 1}^{r}\int_{\s_{F_{j}}(X)^{-1}(V_{j}\sqrt{(n_{F_{j}}/2)\log{\log{T}}} - 2\mathcal{L})}^{\infty}
		e^{-u^2/2}\frac{du}{\sqrt{2\pi}}
	\end{multline}
	for $\frac{\mathcal{L}}{\sqrt{\log{\log{T}}}}, \norm[]{\bm{V}}
		\leq c\sqrt{\log{\log{T}}}$ with $c$ sufficiently small. Since we have $\operatorname{meas}(\bigcup_{j=1}^r(\mathcal{B}_j\cup \mathcal{C}_j))\ll T\exp(-c_4\mathcal L)$ for some $c_4>0$, we choose  $\mathcal{L} =c_4^{-1}r( \norm[]{\bm{V}}^{2} + 2A_5\log_{3}{T})$ so that that \eqref{upper bound V-L-D} becomes
	\begin{multline} 
		T \l( 1 + O_{\bm{F}, A}\l(\frac{(\norm[]{\bm{V}}^{2} + \log_{3}{T})(\norm[]{\bm{V}} + 1)}{\sqrt{\log{\log{T}}}}
			+ \frac{\prod_{k = 1}^{r}(1 + |V_{k}|)}{(\log{\log{T}})^{\a_{\bm{F}} + \frac{1}{2}}}\r)\r)
		\prod_{j = 1}^{r}\int_{V_{j}}^\infty e^{-\frac{u^2}{2}}\frac{du}{\sqrt{2\pi}},
	\end{multline}
	for $\| \bm V\|\leq A (\log \log T)^{1/6}$.
	This completes the proof of \eqref{Main_Thm_LD_JVD2}.

	The proof of \eqref{Main_Thm_LD_JVD3} is similar by noting that
	when $\theta_{j} \in [\frac{\pi}{2}, \frac{3\pi}{2}]$ and $|\tfrac{1}{2}+it-\rho_{F_j}| \leq \frac{1}{\log{Y}}$,
	\begin{align*}
		 & \Re e^{-i\theta_{j}} \log((\tfrac{1}{2}+it-\rho_{F_j})\log{Y})                                                             \\
		 & = \cos{\theta_{j}}\log|(\tfrac{1}{2}+it-\rho_{F_j})\log{Y}| + \sin{\theta_{j}} \arg((\tfrac{1}{2}+it-\rho_{F_j})\log{Y})
		\geq -\pi.
	\end{align*}
	and thus the set of $t\in[T, 2T]$ such that
	\begin{align}
		\sum_{|1/2+it-\rho_{F_j}|\leq \frac{1}{\log{Y}}}\Re e^{-\theta_{j}} \log((\tfrac{1}{2}+it-\rho_{F_j})\log{Y})
		\leq - \mathcal{L}
	\end{align}
	has measure bounded by $C^k k^{2k} T \mathcal{L}^{-2k}$ for some constant $C=C(F_j)$.
\end{proof}

\begin{proof}[Proof of Theorem \ref{GJu}]
	Let $X = T^{1 / (\log{\log{T}})^{4(r + 1)}}$.
	Let $a_{1} = a_{1}(\bm{F}) > 0$ be a sufficiently small constant to be chosen later,
	Let $\bm{V} = (V_{1}, \dots, V_{r}) \in (\RR_{\geq 0})^{r}$ such that
	$\norm[]{\bm{V}} \leq a_{1}(1 + V_{m}^{{1/2}})(\log{\log{T}})^{{1/4}}$ with $V_{m} := \min_{1 \leq j \leq r}V_{j}$.

	We consider the case when $\bm \theta\in[-\frac{\pi}{2}, \frac{\pi}{2}]^r$ first.
	Similarly to the proof of \eqref{Main_Thm_LD_JVD2} (see \eqref{lower V-L-D}),
	we find that the measure of the set of $t \in [T, 2T]$ except for a set of measure
	$T\exp(-c_{4}\mathcal{L})$ ($\mathcal{L} \gg \log_{3}{T}$) such that
	$
		\Re e^{-i\theta_{j}} \log{F_{j}\l( \frac{1}{2} + it \r)}
		\geq V_{j} \sqrt{\frac{n_{F_{j}}}{2} \log{\log{T}}}
	$
	is at most the measure of the set $t \in [T, 2T]$ such that
	\begin{align*}
		\Re e^{-i\theta_{j}}P_{F_{j}}(\tfrac{1}{2} + it, X)
		\geq V_{j} \sqrt{\frac{n_{F_{j}}}{2} \log{\log{T}}} - 2\mathcal{L}.
	\end{align*}
	From Proposition \ref{Main_Prop_JVD}, the measure of $t \in [T, 2T]$ satisfying this inequality for all $j = 1, \dots, r$
	is equal to
	\begin{multline}
		\label{pGJu1}
		T\l( 1 + O_{\bm{F}}\l(
			\frac{\prod_{k = 1}^{r}(1 + V_{k} + \frac{\mathcal{L}}{\sqrt{\log{\log{T}}}})}{ (\log{\log{T}})^{\alpha_F+\frac{1}{2}}}
			+ \frac{1 + \norm[]{\bm{V}}^{2}+\frac{\mathcal{L}^2}{\log{\log{T}}}}{\log{\log{T}}}\r) \r)\\
		\times \prod_{j = 1}^{r}\int_{\s_{F_{j}}(X)^{-1}(V_{j}\sqrt{(n_{F_{j}}/2)\log{\log{T}}} - 2\mathcal{L})}^{\infty}
		e^{-u^2/2}\frac{du}{\sqrt{2\pi}}
	\end{multline}
	for $\frac{\mathcal{L}}{\sqrt{\log{\log{T}}}}, \norm[]{\bm{V}} \leq c \sqrt{\log{\log{T}}}$ with $c$ sufficiently small.

	Now, we choose $\mathcal{L} = 2 r c_{4}^{-1} \norm[]{\bm{V}}^2 + \log_{3}{T}$
	and $a_{1}$ small enough so that we have the inequalities $\norm[]{\bm{V}} \leq a_{9} \sqrt{\log{\log{T}}}$ and
	$4\mathcal{L} \leq (1 + V_{j})\sqrt{\frac{n_{F_{j}}}{2}\log{\log{T}}}$
	for all $j = 1, \dots, r$, where $a_{9}$ is the same constant as in Proposition \ref{Main_Prop_JVD}.
	Then, by equation \eqref{sigaminv} and the estimate $\int_{V}^{\infty}e^{-u^2/2}du \ll \frac{1}{1 + V}e^{-V^2/2}$
	for $V \geq 0$,
	we obtain
	\begin{align*}
		 & \int_{\s_{F_{j}}(X)^{-1}(V_{j}\sqrt{(n_{F_{j}}/2)\log{\log{T}}} - 2\mathcal{L})}^{\infty}
		e^{-u^2/2}\frac{du}{\sqrt{2\pi}}                                                                                  \\
		 & \ll_{r, F_{j}} \frac{1}{1 + V_{j}}
		\exp\l( -\frac{1}{2}\l( V_{j}
		+ O_{r, F_{j}}\l( \frac{\mathcal{L}}{\sqrt{\log{\log{T}}}} + V_{j}\frac{\log_{3}{T}}{\log{\log{T}}} \r) \r)^2 \r) \\
		 & \ll_{\bm F} \frac{1}{1 + V_{j}}
		\exp\l( -\frac{V_{j}^2}{2} + O_{\bm{F}}\l( \frac{V_{j}\norm[]{\bm{V}}^2}{\sqrt{\log{\log{T}}}}
		+ \frac{\norm[]{\bm{V}}^{4}}{\log{\log{T}}} \r) \r).
	\end{align*}
	Hence, when $0 \leq V_{1}, \dots, V_{r} \leq a \sqrt{\log{\log{T}}}$ with $a$ sufficiently small,
	\eqref{pGJu1} is
	\begin{align*}
		\ll_{\bm{F}} T \l\{\l(\prod_{j = 1}^{r}\frac{1}{1 + V_{j}}\r) + \frac{1}{(\log{\log{T}})^{\a_{\bm{F}} + \frac{1}{2}}} \r\}
		\exp\l( -\frac{V_{1}^2 + \cdots + V_{r}^2}{2}
		+ O_{\bm{F}}\l( \frac{\norm[]{\bm{V}}^{3}}{\sqrt{\log{\log{T}}}} \r) \r).
	\end{align*}
	Moreover, we have
	\begin{align*}
		T \exp(-c_{4} \mathcal{L})
		\leq T \exp(-2 r \norm[]{\bm{V}}^2)
		 & \leq T \prod_{j = 1}^{r}\exp\l( -2(V_{1}^2 + \cdots + V_{r}^2) \r)      \\
		 & \ll T \prod_{j = 1}^{r}\frac{1}{1 + V_{j}}\exp\l(-\frac{V_{j}^2}{2}\r).
	\end{align*}
	Similarly when $\bm \theta\in[\frac{\pi}{2}, \frac{3\pi}{2}]^r$, except for a set of measure $T \exp(-c_4\mathcal L)$ ($\mathcal L \gg \log_3T$), the measure of $t\in[T, 2T]$  such that
	\begin{align}
		\Re e^{-i\theta_j} \log F_j(\tfrac{1}{2}+it) \geq V_j \sqrt{\tfrac{n_{F_j}}{2}\log\log T}
	\end{align}
	is at least the measure of $t\in[T, 2T]$ such that
	\begin{align}
		\Re e^{-i\theta_j}  P_{F_j}(\tfrac{1}{2}+it) \geq V_j \sqrt{\tfrac{n_{F_j}}{2}\log\log T} + 2\mathcal L
	\end{align}
	When $4 \mathcal L\leq V_j\sqrt{\frac{n_{F_j}}{2}\log\log T}$, we have the measure of $t$ satisfying the above inequality for all $j = 1, \dots, r$ is (by Proposition \ref{Main_Prop_JVD})
	\begin{multline}\label{lowden}
		T\l( 1 + O_{\bm{F}}\l(
			\frac{\prod_{k = 1}^{r}(1 + V_{k} + \frac{\mathcal{L}}{\sqrt{\log{\log{T}}}})}{ (\log{\log{T}})^{\alpha_F+\frac{1}{2}}}
			+ \frac{1 + \norm[]{\bm{V}}^{2}+\frac{\mathcal{L}^2}{\log{\log{T}}}}{\log{\log{T}}}\r) \r)\\
		\times \prod_{j=1}^r\int_{\s_{F_{j}}(X)^{-1}(V_{j}\sqrt{(n_{F_{j}}/2)\log{\log{T}}} + 2\mathcal{L})}^{\infty}e^{-u^2/2}\frac{du}{\sqrt{2\pi}},
	\end{multline}
	which can be bounded by
	\begin{align}
		 & \gg_{\bm{F}} T \l(\prod_{j = 1}^{r}\frac{1}{1 + V_{j}}\r)
		\exp\l( -\frac{V_{1}^2 + \cdots + V_{r}^2}{2} + O_{\bm{F}}\l( \frac{\norm[]{\bm{V}}\mathcal L}{\sqrt{\log{\log{T}}}} \r) \r)
	\end{align}
	when $\prod_{j = 1}^{r}(1 + V_{j}) \leq c(\log{\log{T}})^{\a_{\bm{F}} + \frac{1}{2}}$
	with $c = c(\bm{F}) > 0$ a suitably small constant.
	Choose $\mathcal{L}=2r\|\bm V\|^2+ \log_3 T$ and $a_1$ small enough so that $\|\bm V\|\leq a_6 \sqrt{\log\log T}$ and $4\mathcal L \leq V_j\sqrt{\frac{n_{F_j}}{2}\log\log T}$ hold for all $j=1, \dots, r$, where $a_6$ is the same constant as in Proposition \ref{Main_Prop_JVD}.  Then \eqref{lowden} is
	\begin{align}
		 & \gg_{\bm{F}} T \l(\prod_{j = 1}^{r}\frac{1}{1 + V_{j}}\r)
		\exp\l( -\frac{V_{1}^2 + \cdots + V_{r}^2}{2} + O_{\bm{F}}\l( \frac{\norm[]{\bm{V}}^{3}}{\sqrt{\log{\log{T}}}} \r) \r),
	\end{align}
	which completes the proof of Theorem \ref{GJu}.
\end{proof}



We prepare a lemma to prove Theorem \ref{New_MVT}.

\begin{lemma}	\label{LLDSL}
	Let $\theta \in \l[ -\frac{\pi}{2}, \frac{\pi}{2} \r]$, and $F \in \Sc^{\dagger}$ satisfying \eqref{SNC} and (A3).
	There exist positive constants $a_{8} = a_{8}(F)$ such that for any large $V$,
	\begin{align*}
		\frac{1}{T}\meas\set{t \in [T, 2T]}{\Re e^{-i\theta} \log{F(\tfrac{1}{2}+it)} > V}
		\leq \exp\l(-a_{8} \frac{V^2}{\log{\log{T}}}\r)
		+ \exp\l( -a_{8} V \r).
	\end{align*}
\end{lemma}

\begin{proof}
	We can show that, for $t \in [T, 2T]$, the inequality $\Re e^{-i\theta} \log F(1/2+it) \leq C_{1}\log{T}$
	with $C_{1} = C_{1}(F) > 0$ a suitable constant by using
	Theorem \ref{Main_F_S} in the case $X = 3$, $H = 1$ and estimate \eqref{SIZDS}.
	Hence, this lemma holds when $V \geq C_{1} \log{T}$ with $C_{1} = C_{1}(F) > 0$.
	In the following, we consider the case $V \leq C_{1}\log{T}$.
	Similarly to the proof of Lemma \ref{diff}, we obtain
	\begin{align*}
		 & \frac{1}{T}\int_{T}^{2T}\bigg| \log{F(\tfrac{1}{2} +it)} - P_{F}(\tfrac{1}{2}+it, X)
		- \sum_{|\frac{1}{2}+it-\rho_{F}| \leq \frac{1}{\log{X}}}\log{((\tfrac{1}{2}+it-\rho_{F})\log{X})} \bigg|^{2k}dt \\
		 & \leq A_{4}^{k}k^{2k} + A_{4}^{k} k! (\log\log T)^{k}
	\end{align*}
	for $X = T^{\delta_F/ k}$, where $\delta_F$ has the same meaning as in Lemma \ref{ESRSIZ}. 
	Moreover, by Lemma \ref{SLL} and Lemma \ref{ESRSIZ}, we obtain
	\begin{align*}
		\frac{1}{T}\int_{T}^{2T}|P_{F}(\tfrac{1}{2}+it, X)|^{2k} dt
		\leq \l( Ck \log{\log{T}} \r)^{k},
	\end{align*}
	and
	\begin{align*}
		\frac{1}{T}\int_{T}^{2T}\l( \sum_{|\frac{1}{2}+it-\rho_{F}| \leq \frac{1}{\log{X}}} 1\r)^{2k}dt
		\leq C^{k} k^{2k}.
	\end{align*}
	When $V \leq \log{\log{T}}$, we choose $k = \lfloor c V^2 / \log{\log{T}} \rfloor$,
	and when $V \geq \log{\log{T}}$, we choose $k = \lfloor c V \rfloor$.
	Here, $c$ is a suitably small constant depending only on $F$.
	Then, by the above calculations and the inequality $\Re e^{-i\theta}\log{((1/2+it-\rho_{F})\log{X})} \leq \pi$,
	we obtain
	\begin{align*}
		\frac{1}{T}\meas\set{t \in [T, 2T]}{\Re e^{-i\theta} \log{F(\tfrac{1}{2}+it)} > V}
		\leq \exp\l(-a_{8} \frac{V^2}{\log{\log{T}}}\r)
		+ \exp\l( -a_{8} V \r),
	\end{align*}
	which completes the proof of Lemma \ref{LLDSL}.
\end{proof}

\begin{proof}[Proof of Theorem \ref{New_MVT}]
	Let $0 \leq k \leq a_{6}$ with $a_{6} = a_{6}(\bm{F}) > 0$ suitably small to be chosen later.
	Put $\phi_{\bm{F}}(t) = \min_{1 \leq j \leq r}\Re e^{-i\theta_{j}} \log{F_{j}(1/2+it)}$ and
	\begin{align*}
		\Phi_{\bm{F}}(T, V)
		:= \meas\set{t \in [T, 2T]}{\phi_{\bm{F}}(t) > V}.
	\end{align*}
	Then we have
	\begin{align}
		\label{MINEF}
		\int_{T}^{2T}\exp\l( 2k \phi_{\bm{F}}(t) \r)dt
		= \int_{-\infty}^{\infty}2k e^{2k V}\Phi_{\bm F}(T, V)dV.
	\end{align}
	We consider the case when $\bm \theta \in [-\frac{\pi}{2},\frac{\pi}{2}]^r$ first.
	From Theorem \ref{GJu}, it follows that, for any $0 \leq V \leq a_{3}\log{\log{T}}$ with $a_{3} = a_{3}(\bm{F})$ a suitable constant,
	\begin{align}
		\label{APGJu}
		 & \Phi_{\bm{F}}(T, V) \\
		 & \ll_{\bm{F}}
		T\l(\frac{1}{1 + (V/\sqrt{\log{\log{T}}})^{r}} + \frac{1}{(\log{\log{T}})^{\a_{\bm{F}} + \frac{1}{2}}} \r)
		\exp\l( -h_{\bm{F}}\frac{V^2}{\log{\log{T}}} + C_{1}\frac{V^3}{(\log{\log{T}})^2} \r)
	\end{align}
	for some constant $C_{1} = C_{1}(\bm{F}) > 0$
	Moreover, by Lemma \ref{LLDSL}, it holds that
	\begin{align}
		\label{TBPhiF}
		\Phi_{\bm{F}}(T, V)
		\leq T\exp\l(-a_{8} \frac{V^2}{\log{\log{T}}}\r) + T\exp\l( -a_{8} V \r)
	\end{align}
	for any large $V$.
	Now we choose $a_{6} = \min\{a_{3}a_{8} / 4, a_{8}/4\}$.
	Put $D_{1} = 4a_{8}^{-1}$.
	We divide the integral on the right hand side of \eqref{MINEF} to
	\begin{align*}
		\l(\int_{-\infty}^{0} + \int_{0}^{D_{1}k\log{\log{T}}} + \int_{D_{1}k\log{\log{T}}}^{\infty}\r)2k e^{2k V}\Phi_{\bm{F}}(T, V)dV
		=: I_{1} + I_{2} + I_{3},
	\end{align*}
	say. We use the trivial bound $\Phi_{\bm{F}}(T, V) \leq T$ to obtain $I_{1} \leq T$.
	Also, by inequality \eqref{TBPhiF}, it follows that
	\begin{align*}
		I_{3}
		 & \leq T\int_{D_{1} k \log{\log{T}}}^{\infty}2k \l\{\exp\l( \l(-a_{8}\frac{V}{\log{\log{T}}} + 2k\r)V \r)
		+ e^{(-a_{8} + 2k)V}\r\}dV                                                                                 \\
		 & \leq T\int_{0}^{\infty}4k e^{-2k V}dV
		\leq 2T.
	\end{align*}
	Moreover, using inequality \eqref{APGJu}, we find that
	\begin{align*}
		I_{2}
		 & \ll_{\bm{F}} T\int_{0}^{D_{1} k \log{\log{T}}}
		\l(  E_{1} + E_{2}\r)
		\exp\l(2kV - h_{\bm{F}}\frac{V^2}{\log{\log{T}}} + C_{1}\frac{V^3}{(\log{\log{T}})^2} \r)dV \\
		 & \ll T(\log{T})^{k^2/h_{\bm{F}} + C_{1}D_{1}^{3}k^{3}}
		\int_{0}^{D_{1} k \log{\log{T}}}(E_{1} + E_{2})
		\exp\l( - \frac{h_{\bm{F}}}{\log{\log{T}}}\l( V - \frac{k}{h_{\bm{F}}}\log{\log{T}} \r)^2 \r)dV,
	\end{align*}
	where $E_{1} = \frac{k}{1 + (V/\sqrt{\log{\log{T}}})^r}$ and $E_{2} = \frac{k}{(\log{\log{T}})^{\a_{\bm{F}} + \frac{1}{2}}}$.
	We see that
	\begin{align*}
		 & \int_{0}^{D_{1} k \log{\log{T}}}E_{2}
		\exp\l( - \frac{h_{\bm{F}}}{\log{\log{T}}}\l( V - \frac{k}{h_{\bm{F}}}\log{\log{T}} \r)^2 \r)dV \\
		 & \leq \frac{k}{(\log{\log{T}})^{\a_{\bm{F}} + \frac{1}{2}}}
		\int_{-\infty}^{\infty}\exp\l( - \frac{h_{\bm{F}}}{\log{\log{T}}}V^2\r)dV
		\ll_{\bm{F}} \frac{k}{(\log{\log{T}})^{\a_{\bm{F}}}}.
	\end{align*}
	Also, we write
	\begin{align*}
		 & \int_{0}^{D_{1} k \log{\log{T}}}E_{1}
		\exp\l( - \frac{h_{\bm{F}}}{\log{\log{T}}}\l( V - \frac{k}{h_{\bm{F}}}\log{\log{T}} \r)^2 \r)dV                              \\
		 & = \l( \int_{0}^{\frac{k}{2h_{\bm{F}}}\log{\log{T}}} + \int_{\frac{k}{2h_{\bm{F}}}\log{\log{T}}}^{D_{1}k\log{\log{T}}} \r)
		k\frac{\exp\l( - \frac{h_{\bm{F}}}{\log{\log{T}}}\l( V - \frac{k}{h_{\bm{F}}}\log{\log{T}} \r)^2 \r)}{
			1 + (V/\sqrt{\log{\log{T}}})^{r}}dV
		=: I_{2, 1} + I_{2, 2},
	\end{align*}
	say.
	We find that
	\begin{align*}
		I_{2, 2}
		\ll_{\bm{F}} \frac{k}{1 + (k\sqrt{\log{\log{T}}})^{r}}
		\int_{-\infty}^{\infty}\exp\l( -\frac{h_{\bm{F}}}{\log{\log{T}}}V^2 \r)dV
		\ll_{\bm{F}} \frac{k \sqrt{\log{\log{T}}}}{1 + (k\sqrt{\log{\log{T}}})^{r}},
	\end{align*}
	and that
	\begin{align*}
		I_{2, 1}
		\leq \int_{\frac{k}{2h_{\bm{F}}}\log{\log{T}}}^{\frac{k}{h_{\bm{F}}}\log{\log{T}}}
		k\exp\l( -\frac{h_{\bm{F}}}{\log{\log{T}}}V^2 \r)dV
		\leq \sqrt{\frac{\log{\log{T}}}{h_{\bm{F}}}}
		\int_{\frac{k}{2\sqrt{h_{\bm{F}}}}\sqrt{\log{\log{T}}}}^{\infty}ke^{-u^2}du.
	\end{align*}
	If $k \leq (\log{\log{T}})^{-1/2}$, the last is clearly $\ll_{\bm{F}} 1$.
	If $k \geq (\log{\log{T}})^{-1/2}$, we apply the estimate $\int_{x}^{\infty}e^{-u^2}du \ll x^{-1}e^{-x^2}$ to obtain
	\begin{align*}
		I_{2, 1}
		\leq \sqrt{\frac{\log{\log{T}}}{h_{\bm{F}}}}
		\int_{\frac{k}{2\sqrt{h_{\bm{F}}}}\sqrt{\log{\log{T}}}}^{\infty}ke^{-u^2}du
		\ll 1.
	\end{align*}
	Hence, we obtain
	\begin{align}
		\label{pNew_MVT2_1_I2}
		I_{2}
		\ll_{\bm{F}} T + kT(\log{T})^{k^2/h_{\bm{F}} + C_{1}D_{1}^{3}k^{3}}
		\l( \frac{\sqrt{\log{\log{T}}}}{1 + (k\sqrt{\log{\log{T}}})^{r}} + \frac{1}{(\log{\log{T}})^{\a_{\bm{F}} + \frac{1}{2}}} \r).
	\end{align}
	Combing this estimate and the estimates for $I_{1}$, $I_{3}$,
	we complete the proof of \eqref{New_MVT1}.

	Next, we consider the case $\bm \theta \in [\frac{\pi}{2},\frac{3\pi}{2}]^r$.
	By equation \eqref{MINEF}, estimate \eqref{NNGJu2}, and positivity of $\Phi_{\bm{F}}$, we have
	\begin{align}
		\label{pNew_MVT2_1}
		 & \int_{T}^{2T}\exp\l( 2k \phi_{\bm{F}}(t) \r)dt    \\
		 & \geq \int_{0}^{1}2k e^{2kV} \Phi_{\bm{F}}(T, V)dV
		+ \int_{\frac{k}{h_{\bm F}}\log\log{T}}^{\frac{k}{h_{\bm{F}}} \log{\log{T}} + \sqrt{\log{\log{T}}}}
		2k e^{2k V}\Phi_{\bm F}(T, V)dV.
	\end{align}
	By estimate \eqref{NNGJu2}, the first integral on the right hand side is $\gg_{\bm{F}} T$,
	and the second integral on the right hand side is
	\begin{align}
		 & \gg_{\bm{F}}
		\frac{kT}{1 + (k\sqrt{\log\log T})^{r}}
		\int_{\frac{k}{h_{\bm F}}\log\log T}^{\frac{k}{h_{\bm{F}}} \log{\log{T}} + \sqrt{\log{\log{T}}}}
		\exp\l(2kV - h_{\bm{F}}\frac{V^2}{\log{\log{T}}} - C_{1}\frac{V^3}{(\log{\log{T}})^2} \r)dV                                   \\
		 & \geq \frac{kT (\log{T})^{\frac{k^2}{h_{\bm{F}}} - C_{2}k^{3}}}{1 + (k\sqrt{\log\log T})^{r}}
		\int_{\frac{k}{h_{\bm F}}\log\log T}^{\frac{k}{h_{\bm{F}}} \log{\log{T}} + \sqrt{\log{\log{T}}}}
		\exp\l( -\frac{h_{\bm{F}}}{\log{\log{T}}}\l( V - \frac{k}{h_{\bm{F}}}\log{\log{T}} \r)^2 \r)                                  \\
		 & \gg _{\bm{F}} kT (\log{T})^{\frac{k^2}{h_{\bm{F}}} - C_{2}k^{3}}\frac{\sqrt{\log{\log{T}}}}{1 + (k\sqrt{\log\log T})^{r}},
	\end{align}
	where $C_2 \geq 0$ is some constant depending on $\bm F$.
	Hence, we also obtain Theorem \ref{New_MVT} in the case $\bm \theta \in [\frac{\pi}{2},\frac{3\pi}{2}]^r$.
\end{proof}

\section{\textbf{Proofs of Theorem \ref{RH_LD_JVD} and Theorem \ref{New_MVT_RH}}} \label{cond}

\begin{proof}[Proof of Theorem \ref{RH_LD_JVD}]
	Let $\bm{F} \in (\Sc^{\dagger})^{r}$ and $\bm{\theta} \in [-\tfrac{\pi}{2}, \tfrac{3\pi}{2}]^{r}$ satisfying $\mathscr{A}$.
	Let $T$ be a sufficiently large constant depending on $\bm F$.
	Set $Y = T^{K_{1} / \mathcal{L}}$ where
	$K_{1} = K_{1}(\bm{F}) > 0$ is a suitably large constant and $\mathcal{L} \geq (\log_{3}{T})^{2}$ is a large parameter to be chosen later.
	Let $f$ be a fixed function satisfying the condition of this paper (see Notation) and $D(f) \geq 2$.
	Assuming the Riemann Hypothesis for $F_{1}, \dots, F_{r}$, we apply Theorem \ref{Main_F_S} with $X = Y$, $H = 1$ to obtain
	\begin{align*}
		\log{F_{j}(\tfrac{1}{2}+it)}
		 & = \sum_{2 \leq n \leq Y^2}\frac{\Lam_{F_{j}}(n) v_{f, 1}(e^{\log{n} / \log{Y}})}{n^{1/2+it} \log{n}}      \\
		 & \qquad+ \sum_{|1/2+it-\rho_{F_{j}}| \leq \frac{1}{\log{Y}}}\log((\tfrac{1}{2} + it -\rho_{F_{j}})\log{Y})
		+ R_{F_{j}}(\tfrac{1}{2}+it, Y, 1),
	\end{align*}
	where
	\begin{align*}
		\l| R_{F_{j}}(\tfrac{1}{2}+it, Y, 1)\r|
		\leq C_{0}\l(\frac{1}{\log{Y}}\bigg| \sum_{n \leq Y^3}\frac{\Lam_{F_{j}}(n) w_{Y}(n)}{n^{\frac{1}{2}+\frac{4}{\log{Y}}+it}} \bigg|
		+ \frac{d_{F_{j}}\log{T}}{\log{Y}}\r)
	\end{align*}
	for any $t \in [T, 2T]$.
	Here $C_{0}$ is a positive constant depending only on $f$.
	Moreover, when $\theta_{j} \in \l[ -\frac{\pi}{2}, \frac{\pi}{2} \r]$, it holds that
	\begin{align*}
		\Re e^{-i\theta_{j}} \sum_{|1/2+it-\rho_{F_{j}}| \leq \frac{1}{\log{Y}}}\log((\tfrac{1}{2} + it -\rho_{F_{j}})\log{Y})
		\leq \pi \sum_{|1/2+it-\rho_{F_{j}}| \leq \frac{1}{\log{Y}}}1,
	\end{align*}
	and when $\theta_{j} \in \l[ \frac{\pi}{2}, \frac{3\pi}{2} \r]$, it holds that
	\begin{align*}
		\Re e^{-i\theta_{j}} \sum_{|1/2+it-\rho_{F_{j}}| \leq \frac{1}{\log{Y}}}\log((\tfrac{1}{2} + it -\rho_{F_{j}})\log{Y})
		\geq -\pi \sum_{|1/2+it-\rho_{F_{j}}| \leq \frac{1}{\log{Y}}}1.
	\end{align*}
	Hence, there exists some positive constant $C_{1} > 0$ such that we have (by \eqref{NZVDP}),
	\begin{align*}
		 & \Re e^{-i\theta_{j}} \sum_{|1/2+it-\rho_{F_{j}}| \leq \frac{1}{\log{Y}}}\log((\tfrac{1}{2} + it -\rho_{F_{j}})\log{Y}) \\
		 & \leq C_{1} \l( \frac{1}{\log{Y}}\bigg|
		\sum_{n \leq Y^3}\frac{\Lam_{F_{j}}(n) w_{Y}(n)}{n^{\frac{1}{2}+\frac{4}{\log{Y}}+it}} \bigg|
		+ \frac{d_{F_{j}}\log{T}}{\log{Y}} \r)
	\end{align*}
	when $\theta_{j} \in \l[ -\frac{\pi}{2}, \frac{\pi}{2} \r]$,
	and
	\begin{align*}
		 & \Re e^{-i\theta_{j}} \sum_{|1/2+it-\rho_{F_{j}}| \leq \frac{1}{\log{Y}}}\log((\tfrac{1}{2} + it -\rho_{F_{j}})\log{Y}) \\
		 & \geq -C_{1} \l( \frac{1}{\log{Y}}\bigg|
		\sum_{n \leq Y^3}\frac{\Lam_{F_{j}}(n) w_{Y}(n)}{n^{\frac{1}{2}+\frac{4}{\log{Y}}+it}} \bigg|
		+ \frac{d_{F_{j}}\log{T}}{\log{Y}} \r)
	\end{align*}
	when $\theta_{j} \in \l[ \frac{\pi}{2}, \frac{3\pi}{2} \r]$.
	Taking $K_{1} = 2(C_{0} + C_{1})\max_{1 \leq j \leq r}d_{F_{j}}$, we find that there exists some positive constant $C_{2}$ depending on $f$ such that for any $t \in [T, 2T]$ and all $j = 1, \dots, r$,
	\begin{align} \label{upRH}
		\Re e^{-i\theta_{j}} \log{F_{j}(\tfrac{1}{2}+it)}
		\leq & \Re e^{-i\theta_{j}}\sum_{2 \leq n \leq Y^2}\frac{\Lam_{F_{j}}(n) v_{f, 1}(e^{\log{n} / \log{Y}})}{n^{1/2+it} \log{n}}      \\
		     & + \frac{C_{2}}{\log{Y}}\bigg| \sum_{n \leq Y^3}\frac{\Lam_{F_{j}}(n) w_{Y}(n)}{n^{\frac{1}{2}+\frac{4}{\log{Y}}+it}} \bigg|
		+ \frac{\mathcal{L}}{2}
	\end{align}
	when $\theta_{j} \in \l[ -\frac{\pi}{2}, \frac{\pi}{2} \r]$,
	and
	\begin{align}\label{loRH}
		\Re e^{-i\theta_{j}} \log{F_{j}(\tfrac{1}{2}+it)}
		\geq & \Re e^{-i\theta_{j}}\sum_{2 \leq n \leq Y^2}\frac{\Lam_{F_{j}}(n) v_{f, 1}(e^{\log{n} / \log{Y}})}{n^{1/2+it} \log{n}}      \\
		     & - \frac{C_{2}}{\log{Y}}\bigg| \sum_{n \leq Y^3}\frac{\Lam_{F_{j}}(n) w_{Y}(n)}{n^{\frac{1}{2}+\frac{4}{\log{Y}}+it}} \bigg|
		- \frac{\mathcal{L}}{2}
	\end{align}
	when $\theta_{j} \in \l[ \frac{\pi}{2}, \frac{3\pi}{2} \r]$.


	Put $X = Y^{1 / (\log{\log{T}})^{4(r + 1)}}$.
	By Lemma \ref{SLL} and assumption (A1), we obtain
	\begin{align*}
		\int_{T}^{2T}\bigg| \sum_{X < p \leq Y^2}\frac{a_{F_{j}}(p) v_{f, 1}(e^{\log{p} / \log{Y}})}{p^{1/2+it}} \bigg|^{2k}dt
		 & \ll T k! \l( \sum_{X < p \leq Y^2}\frac{|a_{F_{j}}(p)|^2}{p} \r)^{k} \\
		 & \leq T C_{3}^{k}k^{k}\l(\log_{3}{T}\r)^{k}                           
	\end{align*}
	for some constant $C_{3} = C_{3}(F_{j}, r) > 0$.
	Similarly to the proofs of estimates \eqref{PKLI3},
	we can show that for any integer $k$ with $1 \leq k \leq \mathcal{L} / 4K_{1}$
	\begin{align*}
		\int_{T}^{2T}\bigg| \sum_{\substack{X < p^{\ell} \leq Y^2 \\ \ell \geq 2}}
		\frac{\Lam_{F_{j}}(p^{\ell}) v_{f, 1}(e^{\log{p^{\ell}} / \log{Y}})}{p^{\ell(1/2+it)} \log{p^{\ell}}} \bigg|^{2k}dt
		\leq T C_{4}^k k^{k}
	\end{align*}
	for some constant $C_{4} = C_{4}(F_{j}) > 0$.
	Moreover, by Lemma \ref{ESMDP}, we have
	\begin{align*}
		\int_{T}^{2T}\l( \frac{C_{1}}{\log{Y}}
		\bigg| \sum_{n \leq Y^3}\frac{\Lam_{F_{j}}(n) w_{Y}(n)}{n^{\frac{1}{2}+\frac{4}{\log{Y}}+it}} \bigg| \r)^{2k}dt
		\leq T C_{5}^{k} k^{k}
	\end{align*}
	for any integer $k$ with $1 \leq k \leq \mathcal{L} / 4K_{1}$ and for some constant $C_{5} = C_{5}(F_{j}) > 0$.
	Here the assumptions in Lemma \ref{ESMDP} is satisfied as we can take $\kappa_F$ arbitrarily large.
	Therefore, the set of $t \in [T, 2T]$ such that for all $j = 1, \dots, r$,
	\begin{align*}
		\frac{\mathcal{L}}{2}
		\leq \bigg| \sum_{X < n \leq Y^2}\frac{\Lam_{F}(n) v_{f, 1}(e^{\log{n} / \log{Y}})}{n^{1/2+it}\log{n}} \bigg|
		 & + \frac{C_{2}}{\log{Y}}\bigg| \sum_{n \leq Y^3}\frac{\Lam_{F_{j}}(n) w_{Y}(n)}{n^{\frac{1}{2}+\frac{4}{\log{Y}}+it}} \bigg|
	\end{align*}
	has a measure bounded by
	$
		T \mathcal{L}^{-2k} C_{6}^{k} k^{k} \l(\log_{3}{T}\r)^{k} 
	$
	with $C_{6} = C_{6}(\bm{F}) > 0$ a suitably large constant.
	Choosing $k = \lfloor c_{1} \mathcal{L} \rfloor$ with $c_{1}$ suitably small depending only on $\bm{F}$, we find that
	there exists a set $\mathcal{X} \subset [T, 2T]$ with
	\begin{align}\label{measX}
		\meas(\mathcal{X}) \leq T\exp\l( -c_{1}\mathcal{L}\log\l( \frac{\mathcal{L}}{\log_{3}{T}} \r) \r)
	\end{align}
	such that for any $t \in [T, 2T] \setminus \mathcal{X}$ and any $j = 1, \dots, r$,
	\begin{align}
		\label{RH_LD_JVDp1}
		\Re e^{-i\theta_{j}} \log{F_{j}(\tfrac{1}{2}+it)}
		\leq \Re e^{-i\theta_{j}}P_{F_{j}}(\tfrac{1}{2}+it, X)
		+ \mathcal{L} \quad \text{ when }  \theta_{j}  \in [-\tfrac{\pi}{2}, \tfrac{\pi}{2}],
	\end{align}
	and
	\begin{align}
		\label{RH_LD_JVDp2}
		\Re e^{-i\theta_{j}} \log{F_{j}(\tfrac{1}{2}+it)}
		\geq \Re e^{-i\theta_{j}}P_{F_{j}}(\tfrac{1}{2}+it, X)
		- \mathcal{L} \quad \text{ when } \theta_{j}  \in [\tfrac{\pi}{2}, \tfrac{3\pi}{2}].
	\end{align}
	Now we consider  the case when
	$\|\bm{V}\| \leq a_{5}V_{m}^{{1/2}}(\log{\log{T}})^{{1/4}} (\log_{3}{T})^{{1/2}}$, where $a_{5}$ is a sufficiently small positive constant.
	Set $\mathcal{L} = 4 r c_{1}^{-1}\l(\frac{\norm[]{\bm{V}}^2}{\log{\norm[]{\bm{V}}}} + (\log_{3}{T})^{{4}}\r)$.
	Then we can verify  from \eqref{measX}  that  $\meas(\mathcal{X}) \ll_{\bm{F}} T\exp\l( -2 r \norm[]{\bm{V}}^{2} \r)$.
	Moreover, when $\theta_i\in[-\frac{\pi}{2}, \frac{\pi}{2}]$, the measure of $t\in[T, 2T]\backslash \mathcal{X}$ such that
	\begin{align*}
		\Re e^{-i\theta_{j}} \log{F_{j}(\tfrac{1}{2}+it)} \geq V_{j} \sqrt{\frac{n_{F_{j}}}{2}\log{\log{T}}}
	\end{align*}
	is bounded above by the measure of $t\in [T, 2T]$ such that
	\begin{align*}
		\frac{\Re e^{-i\theta_{j}}P_{F_{j}}(\tfrac{1}{2}+it, X)}{\s_{F_{j}}(X)}
		\geq  V_{j} - C_{\bm{F}}\l( \frac{\mathcal{L}}{\sqrt{\log{\log{T}}}}+ \frac{V_j\log_3T}{\log\log T} \r).
	\end{align*}
	where $C_{\bm F}$ is some positive constant and we used \eqref{sigaminv} for $\sigma_{F_j}(X)^{-1}$.
	Similarly when $\theta_i\in[\frac{\pi}{2}, \frac{3\pi}{2}]$, the measure of $t\in[T, 2T]\backslash \mathcal{X}$ such that
	\begin{align*}
		\Re e^{-i\theta_{j}} \log{F_{j}(\tfrac{1}{2}+it)} \geq V_{j} \sqrt{\frac{n_{F_{j}}}{2}\log{\log{T}}}
	\end{align*}
	is bounded below by the measure of $t\in [T, 2T]$ such that
	\begin{align*}
		\frac{\Re e^{-i\theta_{j}}P_{F_{j}}(\tfrac{1}{2}+it, X)}{\s_{F_{j}}(X)}
		\geq  V_{j} + C_{\bm{F}}\l( \frac{\mathcal{L}}{\sqrt{\log{\log{T}}}}+ \frac{V_j\log_3T}{\log\log T} \r).
	\end{align*}
	Here, we choose $a_{3}$ so that
	$C_{\bm{F}}\l( \frac{\mathcal{L}}{\sqrt{\log{\log{T}}}}+ \frac{V_j\log_3T}{\log\log T} \r)\leq \frac{V_{m}}{2}$.
	From these observations, the estimate $\int_{V}^{\infty}e^{-u^2/2}du \asymp \frac{1}{1+V}e^{-V^2/2}$ for $V \geq 0$,
	estimate \eqref{EST_Xi}, and Proposition \ref{Main_Prop_JVD3}, we find that if $\theta_{j} \in [-\frac{\pi}{2}, \frac{\pi}{2}]$,
	\begin{align*}
		 & \frac{1}{T}\meas(\S(T, \bm{V}; \bm{F}, \bm{\theta}))                                                                             \\
		 & \ll_{\bm{F}} \frac{1}{T}\meas(\mathcal{X})                                                                                       \\
		 & + \l(\frac{1}{V_{1} \cdots V_{r}} + \frac{1}{(\log{\log{T}})^{\a_{\bm{F}}+\frac{1}{2}}}\r)
		\prod_{j = 1}^{r}\exp\Bigg( -\frac{V_{j}^2}{2}
		- \frac{V_{j}^2}{2 \s_{F_{j}}(X)^2}\s_{F_{j}}\l( \frac{V_{j}^2}{\s_{F_{j}}(X)^2} \r)^2 +                                            \\
		 & \qqqquad \qqqquad + O_{\bm{F}}\l( \frac{\norm[]{\bm{V}} \mathcal{L}}{\sqrt{\log{\log{T}}}} + \frac{\mathcal{L}^2}{\log{\log{T}}}
		+ \l( \frac{\norm[]{\bm{V}}}{\sqrt{\log{\log{T}}}} \r)^{\frac{2 - 2\vartheta_{\bm{F}}}{1 - 2\vartheta_{\bm{F}}}} \r) \Bigg)         \\
		 & \ll_{\bm{F}} \l(\frac{1}{V_{1} \cdots V_{r}} + \frac{1}{(\log{\log{T}})^{\a_{\bm{F}}+\frac{1}{2}}}\r)
		\prod_{j = 1}^{r}\exp\l( -\frac{V_{j}^2}{2}
		+ O_{\bm{F}}\l( \frac{\norm[]{\bm{V}}^{3}}{\sqrt{\log{\log{T}}} \log{\norm[]{\bm{V}}}} \r) \r)
	\end{align*}
	for $\|\bm{V}\| \leq a_{5} V_{m}^{{1/2}} (\log{\log{T}})^{{1/4}} (\log_{3}{T})^{{1/2}}$.
	Hence, we obtain estimate \eqref{RH_LD_JVD1u}.
	Similarly, we can also find that if $\theta_{j} \in [\frac{\pi}{2}, \frac{3\pi}{2}]$,
	\begin{align*}
		 & \frac{1}{T}\meas(\S(T, \bm{V}; \bm{F}, \bm{\theta}))                                                                             \\
		 & \gg_{\bm{F}}
		\l(\frac{1}{V_{1} \cdots V_{r}} + \frac{1}{(\log{\log{T}})^{\a_{\bm{F}}+\frac{1}{2}}}\r)
		\prod_{j = 1}^{r}\exp\Bigg( -\frac{V_{j}^2}{2}
		- \frac{V_{j}^2}{2 \s_{F_{j}}(X)^2}\s_{F_{j}}\l( \frac{V_{j}^2}{\s_{F_{j}}(X)^2} \r)^2 +                                            \\
		 & \qqqquad \qqqquad - O_{\bm{F}}\l( \frac{\norm[]{\bm{V}} \mathcal{L}}{\sqrt{\log{\log{T}}}} + \frac{\mathcal{L}^2}{\log{\log{T}}}
		+ \l( \frac{\norm[]{\bm{V}}}{\sqrt{\log{\log{T}}}} \r)^{\frac{2 - 2\vartheta_{\bm{F}}}{1 - 2\vartheta_{\bm{F}}}} \r) \Bigg)         \\
		 & \qqqquad \qqqquad - \frac{1}{T}\meas(\mathcal{X})                                                                                \\
		 & \gg_{\bm{F}} \frac{1}{V_{1} \cdots V_{r}}\exp\l( -\frac{V_{1}^2 + \cdots + V_{r}^2}{2}
		- O_{\bm{F}}\l( \frac{\norm[]{\bm{V}}^{3}}{\sqrt{\log{\log{T}}} \log{\norm[]{\bm{V}}}} \r) \r)
	\end{align*}
	for $\|\bm{V}\| \leq a_{5} V_{m}^{{1/2}} (\log{\log{T}})^{{1/4}} (\log_{3}{T})^{{1/2}}$ satisfying
	$\prod_{j = 1}^{r}V_{j} \leq a_{6}(\log{\log{T}})^{\a_{\bm{F}} + \frac{1}{2}}$ with $a_{6} = a_{6}(\bm{F}) > 0$ a suitably small constant.
	Hence, we also obtain \eqref{RH_LD_JVD1l}.


	Now we consider  \eqref{RH_LD_JVD2}, where $\theta_{j} \in \l[ -\frac{\pi}{2}, \frac{\pi}{2} \r]$.
	Putting $\mathcal{L} = \frac{4K_{1}\log{T}}{\log{\log{T}}}$,
	we see that $Y = (\log{T})^{1/4}$, and hence there exists a positive constant $A = A(\bm{F})$ such that
	the right hand side of \eqref{upRH} is $\leq A \frac{\log{T}}{\sqrt{\log{\log{T}}}}\sqrt{\frac{n_{F_{j}}}{2}\log{\log{T}}}$
	uniformly for any $t \in [T, 2T]$ and all $j = 1, \dots, r$.
	Thus, we may assume $\norm[]{\bm{V}} \leq A\frac{\log{T}}{\sqrt{\log{\log{T}}}}$.
	We first consider  the case when $\sqrt{\log{\log{T}}} \leq \norm[]{\bm{V}} \leq A\frac{\log{T}}{\sqrt{\log{\log{T}}}}$.
	Set $\mathcal{L} = b_{1}\frac{\norm[]{\bm{V}}}{2}\sqrt{\log{\log{T}}}$, where $b_{1}$ is some small positive constant such that the inequality $Y \geq 3$ holds.
	Then we see \eqref{measX}  becomes
	$$
		\meas(\mathcal{X})
		\ll_{\bm{F}} T \exp\l( -c_{2} \norm[]{\bm{V}}\sqrt{\log{\log{T}}} \log{\norm[]{\bm{V}}} \r)
	$$
	for some constant $c_{2} = c_{2}(\bm{F}) > 0$.
	Using Lemma \ref{SLL}, we have, uniformly for any $j = 1, \dots, r$,
	\begin{align}\label{PF_2k}
		\int_{T}^{2T}|P_{F_{j}}(\tfrac{1}{2}+it, X)|^{2k}dt
		\ll_{\bm{F}} T(C_{6} k \log{\log{T}})^{k}
	\end{align}
	for any integer $k$ with $1 \leq k \leq \mathcal{L}\log{\log{T}}$ and some $C_{6} = C_{6}(\bm{F})$.
	Combing \eqref{RH_LD_JVDp1} and \eqref{PF_2k}, we obtain
	\begin{align*}
		 & \frac{1}{T}\meas(\S(T, \bm{V}; \bm{F}, \bm{\theta}))                                                                    \\
		 & \ll \min_{1\leq j\leq r} \frac{1}{T}\meas\set{t \in [T, 2T]}{\Re e^{-i\theta_{j}} \log{F_{j}(\tfrac{1}{2}+it)} > V_{j}} \\
		 & \ll_{\bm{F}}
		\norm[]{\bm{V}}^{-2k} C_{7}^{k} k^{k}
		+ \exp\l( -c_{2} \norm[]{\bm{V}}\sqrt{\log{\log{T}}} \log{\norm[]{\bm{V}}} \r).
	\end{align*}
	When $\norm[]{\bm{V}} \leq \log{\log{T}}$, we choose $k = \lfloor c_{3} \norm[]{\bm{V}}^2 \rfloor$,
	and when $\norm[]{\bm{V}} > \log{\log{T}}$, we choose $k = \lfloor c_{3} \norm[]{\bm{V}}\sqrt{\log{\log{T}}} \rfloor$,
	where $c_{3} $ is a suitably small positive constant depending only on $\bm{F}$.
	Then, it follows that
	\begin{align*}
		\frac{1}{T}\meas(\S(T, \bm{V}; \bm{F}, \bm{\theta}))
		\ll_{\bm{F}}
		\exp\l( -c_{4}\|\bm{V}\|^2 \r) +  \exp\l( -c_{4} \|\bm{V}\|\sqrt{\log{\log{T}}} \log{\norm[]{\bm{V}}} \r),
	\end{align*}
	which completes the proof of \eqref{RH_LD_JVD2}.
\end{proof}

\begin{proof}[Proof of Theorem \ref{New_MVT_RH}]
	Let $T$ be large, and put $\e(T) = (\log_{3}{T})^{-1}$.
	Let $k \geq 0$.
	We recall equation \eqref{MINEF}, which is
	\begin{align*}
		\int_{T}^{2T}\exp\l( 2k \phi_{\bm{F}}(t) \r)dt
		= \int_{-\infty}^{\infty}2k e^{2k V}\Phi_{\bm{F}}(T, V)dV.
	\end{align*}
	We divide the integral on the right hand side to
	\begin{align*}
		\l(\int_{-\infty}^{0}+ \int_{0}^{D_{2}k\log{\log{T}}}
		+ \int_{D_{2}k\log{\log{T}}}^{\infty}\r)2k e^{2k V}\Phi_{\bm{F}}(T, V)dV
		=: I_{4} + I_{5} + I_{6},
	\end{align*}
	say. Here, $D_{2} = D_{2}(\bm{F})$ is a suitably large positive constant.
	Now we consider the case when $\theta_{j} \in [-\frac{\pi}{2}, \frac{\pi}{2}]$.
	We use the trivial bound $\Phi_{\bm{F}}(T, V) \leq T$ to obtain $I_{4} \leq T$.
	Applying estimate \eqref{RH_LD_JVD2}, we find that the estimate
	\begin{align*}
		\Phi_{\bm{F}}(T, V)
		\ll_{\bm{F}} T\exp\l( - 4k V \r)
	\end{align*}
	holds for $V \geq D_{2} k \log{\log{T}}$ when $D_{2}$ is suitably large.
	Therefore, we have
	\begin{align*}
		I_{6}
		\ll_{\bm{F}} T \int_{D\log{\log{T}}}^{\infty}2k  e^{-2kV}dV
		\ll T.
	\end{align*}
	By estimate \eqref{RH_LD_JVD1u}, we find that
	\begin{align*}
		 & \Phi_{\bm{F}}(T, V)                                                                                                            \\
		 & \ll_{k, \bm{F}} T \l( \frac{1}{1 + (V / \sqrt{\log{\log{T}}})^{r}} + \frac{1}{(\log{\log{T}})^{\a_{\bm{F}} + \frac{1}{2}}} \r)
		\exp\l( -\frac{h_{\bm{F}}V^{2}}{\log{\log{T}}} + \frac{C_{1} V^3}{(\log{\log{T}})^2 \log_{3}{T}} \r)                              \\
		 & \ll T\l( \frac{1}{1 + (V / \sqrt{\log{\log{T}}})^{r}} + \frac{1}{(\log{\log{T}})^{\a_{\bm{F}} + \frac{1}{2}}} \r)
		(\log{T})^{C_{1} D_{2}^{3} k^{3} \e(T)}\exp\l( -h_{\bm{F}}\frac{V^{2}}{\log{\log{T}}} \r)
	\end{align*}
	for $(\log{\log{T}})^{2/3} \leq V \leq D_{2}k\log{\log{T}}$.
	Here, $C_{1} = C_{1}(\bm{F})$ is some positive constant.
	Similarly to the proof of \eqref{pNew_MVT2_1_I2} by using this estimate, we obtain
	\begin{align*}
		I_{5}
		\ll_{\bm{F}} T + T(\log{T})^{k^{2}/h_{\bm{F}} + Bk^{3}\e(T)}
		\l( \frac{k \sqrt{\log{\log{T}}}}{1 + (k \sqrt{\log{\log{T}}})^{r}} + \frac{1}{(\log{\log{T}})^{\a_{\bm{F}} + \frac{1}{2}}} \r).
	\end{align*}
	Hence, we obtain \eqref{New_MVT_RH1}.

	For estimate \eqref{New_MVT_RH2}, it holds from the positivity of $\Phi_{\bm{F}}(T, V)$ and equation \eqref{MINEF} that
	\begin{align*}
		\int_{T}^{2T}\exp\l( 2k \phi_{\bm{F}}(t) \r)dt
		\gg_{k} \int_{\frac{k}{h_{\bm{F}}}\log{\log{T}}}^{\frac{k}{h_{\bm{F}}}\log{\log{T}} + \sqrt{\log{\log{T}}}}
		e^{2k V}\Phi_{\bm{F}}(T, V)dV.
	\end{align*}
	When $\theta_i\in[\frac{\pi}{2}, \frac{3\pi}{2}]$,
	assuming $\vartheta_{\bm{F}} < \frac{1}{r + 1}$, we use \eqref{RH_LD_JVD1l} to obtain
	\begin{align*}
		\Phi_{\bm{F}}(T, V)
		 & \gg_{k, \bm{F}} \frac{T}{1 + (V / \sqrt{\log{\log{T}}})^{r}}
		\exp\l( -h_{\bm{F}}\frac{V^{2}}{\log{\log{T}}} - \frac{CV^3}{(\log{\log{T}})^2 \log_{3}{T}} \r) \\
		 & \gg_{\bm{F}} \frac{T (\log{T})^{-C_{2} k^{3} \e(T)}}{1 + (V / \sqrt{\log{\log{T}}})^{r}}
		\exp\l( -h_{\bm{F}}\frac{V^{2}}{\log{\log{T}}} \r)
	\end{align*}
	for $\frac{k}{h_{\bm{F}}}\log{\log{T}} \leq V \leq \frac{k}{h_{\bm{F}}}\log{\log{T}} + \sqrt{\log{\log{T}}}$.
	Here, $C_{2} = C_{2}(\bm{F})$ is a positive constant.
	Similarly to the proof of \eqref{New_MVT2} by using this estimate
	and the bound $\Phi_{\bm{F}}(T, V) \gg_{\bm{F}} T$ for $0 \leq V \leq 1$,
	we can also obtain \eqref{New_MVT_RH2}.
\end{proof}

\section{\textbf{Concluding remarks}}\label{finalremarks}
\subsection{Moments of max/min values}
Under the same conditions as in Theorem \ref{New_MVT}, we can apply Theorem \ref{GJu} to show that
for sufficiently small $k$, there exists some constant $B=B(\bm F)>0$ such that 
\begin{gather}
	\label{MAXLUC}
	\int_{T}^{2T}\l(\max_{1 \leq j \leq r}|F_{j}(\tfrac{1}{2} + it)|\r)^{2k}dt
  \ll_{\bm{F}} T (\log{T})^{n_{\bm{F}}k^2 + B k^{3}}, \\
  \label{MINLUC}
  \int_{T}^{2T}\l(\min_{1 \leq j \leq r}|F_{j}(\tfrac{1}{2} + it)|\r)^{-2k}dt
  \gg_{\bm{F}} T (\log{T})^{n_{\bm{F}}k^2 - B k^{3}}, 
\end{gather}
where $n_{\bm{F}} =\max_{1 \leq j \leq r}n_{F_{j}}$.
Assuming the Riemann Hypothesis for these $L$-functions, we can replace the term $Bk^3$ in the exponent by any $\epsilon>0$ and the results hold for any $k>0$. 
In fact, consider
\begin{align*}
	S_{j}
	= \set{t \in [T, 2T]}{\Re e^{-i\theta_{j}} \log{F_{j}}(\tfrac{1}{2} + it) > V}.
\end{align*}
Using the inclusion-exclusion principle, we see that
\begin{align}
  \begin{aligned}
    \label{IEPS_j}
    &\sum_{j = 1}^{r}\meas S_{j} - \sum_{1 \leq \ell < m \leq r}\meas\l( S_{\ell} \cap S_{m} \r) \\
    &\leq \meas\l(\bigcup_{j = 1}^{r}S_{j}\r)
    = \meas\set{t \in [T, 2T]}{\max_{1 \leq j \leq r}\Re e^{-i\theta_{j}} \log{F_{j}}(\tfrac{1}{2} + it) > V}  \\
    & \leq \sum_{j = 1}^{r}\meas S_{j}.
  \end{aligned}
\end{align}
Applying the second inequality of \eqref{IEPS_j} and Theorem \ref{GJu} with $r = 1$ and
$V_{1} = \frac{V}{\sqrt{(n_{F_{j}} / 2)\log{\log{T}}}}$ for $j = 1, \dots, r$, we obtain
\begin{align*}
	 & \frac{1}{T}\meas\set{t \in [T, 2T]}{\max_{1 \leq j \leq r}\Re e^{-i\theta_{j}} \log{F_{j}}(\tfrac{1}{2} + it) > V} \\
	 & \ll_{\bm{F}} \frac{1}{1 + V / \sqrt{\log{\log{T}}}}
	\exp\l( -\frac{V^2}{n_{\bm{F}}\log{\log{T}}}\l( 1 + O_{\bm{F}}\l( \frac{V}{\log{\log{T}}} \r) \r) \r)
\end{align*}
for any $\bm{\theta} \in [-\tfrac{\pi}{2}, \tfrac{\pi}{2}]^{r}$ and $0 \leq V \leq a \log{\log{T}}$ with $a = a(\bm{F}) > 0$ suitable small.
Similarly, applying the first inequality of \eqref{IEPS_j} and Theorem \ref{GJu} with $r = 1, 2$,
$V_{1} = \frac{V}{\sqrt{(n_{F_{j}}/2)\log{\log{T}}}}$ for $j = 1, \dots, r$, 
and $(V_{1}, V_{2}) = \l( \frac{V}{\sqrt{(n_{F_{\ell}} / 2)\log{\log{T}}}}, \frac{V}{\sqrt{(n_{F_{m}} / 2)\log{\log{T}}}} \r)$
for $1 \leq \ell < m \leq r$, we obtain
\begin{align*}
  & \frac{1}{T}\meas\set{t \in [T, 2T]}{\max_{1 \leq j \leq r}\Re e^{-i\theta_{j}} \log{F_{j}}(\tfrac{1}{2} + it) > V} \\
  & \gg_{\bm{F}} \frac{1}{1 + V / \sqrt{\log{\log{T}}}}
 \exp\l( -\frac{V^2}{n_{\bm{F}}\log{\log{T}}}\l( 1 + O_{\bm{F}}\l( \frac{V}{\log{\log{T}}} \r) \r) \r)
\end{align*}
for any $\bm{\theta} \in [\tfrac{\pi}{2}, \tfrac{3\pi}{2}]^{r}$ 
and $0 \leq V \leq a_{3}\min\{\log{\log{T}}, (\log \log T)^{\frac{\a_{\bm{F}}}{r} + \frac{1}{2} + \frac{1}{2r}}\}$.
The rest of the argument follows the same way as in Theorems \ref{New_MVT} and Theorem \ref{New_MVT_RH}.
Note that assumptions \eqref{SOC2}, (A2) are not needed for \eqref{MAXLUC}. 
\subsection{Moments of product of $L$-functions}
Our method can be modified to study the moment of product of $L$-functions. Under the same assumptions as in Theorem \ref{New_MVT}, we can also show that for sufficiently small $k_1, \dots, k_r>0$ there exists some constant $B=B(\bm F)$ such that
\begin{align}
  \label{PRDLUC}
&\int_{T}^{2T}\prod_{j = 1}^{r}|F_{j}(\tfrac{1}{2} + it)|^{2k_{j}}dt
	\ll_{\bm{F}} T (\log{T})^{n_{F_{1}}k_{1}^2 + \cdots + n_{F_{r}}k_{r}^2 + B K^{3}},\\
	 \label{PRDLUC2}
	&\int_{T}^{2T}\prod_{j = 1}^{r}|F_{j}(\tfrac{1}{2} + it)|^{-2k_{j}}dt
	\gg_{\bm{F}} T (\log{T})^{n_{F_{1}}k_{1}^2 + \cdots + n_{F_{r}}k_{r}^2 - B K^{3}}, \quad \text{ if }\theta_{\bm F}\leq \frac{1}{r+1}
\end{align}
where $K=\max_{1\leq j\leq r}k_j$.
If we assume the Riemann Hypothesis for these $L$-functions, the term $BK^3$ in the exponent can be replaced by any $\epsilon>0$ and the results hold for any $k_1, \dots, k_r>0$. 

To do this, we need to modify the definition of $\S_{X}(T, \bm{V}; \bm{F}, \bm{\theta})$.
Let $X \geq 3$, $\bm{k} = (k_{1}, \dots, k_{r}) \in (\RR_{\geq 0})^{r}$, 
and let $\bm{F}$ be a $r$-tuple of Dirichlet series and $\bm{\theta} \in \RR^{r}$ satisfying (S4), (S5), (A1), and (A2). 
Let
\begin{align*}
	\s_{\bm{F}}(X, \bm{k})
	= \sqrt{\frac{1}{2}\sum_{j = 1}^{r}k_{j}\sum_{p \leq X} \sum_{\ell = 1}^{\infty}\frac{|b_{F_{j}}(p^{\ell})|^2}{p^{\ell}}}
\end{align*}
and 
\begin{align*}
\S_{X}(T, \bm{V}; \bm{F}, \bm{\theta}, \bm k)=\frac{1}{T}\meas
	\set{t \in [T, 2T]}{\frac{\sum_{j = 1}^{r}k_{j}\Re e^{-i\theta_{j}}P_{F_{j}}(\tfrac{1}{2} + it, X)}{\s_{\bm{F}}(X; \bm{k})} > V}.
\end{align*}
Then, we can show that for any $\bm{k} \in (\RR_{> 0})^{r}$
and $|V| \leq a \s_{\bm{F}}(X; \bm{k})$ with $a = a(\bm{F}, \bm{k}) > 0$ suitably small,
\begin{align*}
	\frac{1}{T}\meas
	\set{t \in [T, 2T]}{\frac{\sum_{j = 1}^{r}k_{j}\Re e^{-i\theta_{j}}P_{F_{j}}(\tfrac{1}{2} + it, X)}{\s_{\bm{F}}(X; \bm{k})} > V}
	\sim \int_{V}^{\infty}e^{-u^2/2}\frac{du}{\sqrt{2\pi}}
\end{align*}
by modifying the proofs in Section \ref{Proof_Props_JVD}.
Based on this asymptotic formula, we can also prove an analogue of Theorem \ref{GJu} and thus \eqref{PRDLUC}.
Similarly, we can also obtain \eqref{PRDLUC2}.

\subsection{Other large deviation results}
We would like to mention that our method also recovers the work of Heuberger-Kropf \cite{HK2018}
for higher dimensional quasi-power theorem, and it is likely our method yields improvement to their work in the direction of large deviations.




\begin{thebibliography}{99}

	\bibitem{AS} M. Avdispahić and L. Smajlović, On the Selberg orthogonality for automorphic $L$-functions,
	\textit{Arch. Math.} \textbf{94} (2010), 147--154.




	\bibitem{BH1995} E. Bombieri and D. A. Hejhal, On the distribution of zeros of linear combinations of Euler products,
	\textit{Duke Math. J.} \textbf{80} no.3 (1995), 821--862.

	\bibitem{CG1993} J. B. Conrey and A. Ghosh, On the Selberg class of Dirichlet series: small degrees,
	\textit{Duke Math. J.} \textbf{72} no.3 (1993), 673--693.

	\bibitem{Conrey05} J.B. Conrey, D.W. Farmer, J.P. Keating, M.O. Rubinstein and N.C. Snaith, Integral moments of $L$-functions,
	\textit{Proc. Lond. Math. Soc.} (3) \textbf{91} (2005), 33--104.

  \bibitem{Cramer} H. Cramér, Sur un nouveau théorème limite de la théorie des probabilités. \textit{Act. Se. Ind.}, \textbf{736} (1938), 5--23.



	\bibitem{F1974} A. Fujii, On the zeros of Dirichlet $L$-functions. I,
	\textit{Trans. Amer. Math. Soc.} \textbf{196} (1974), 225--235.

	\bibitem{FZ} K. Ford and A. Zaharescu, Unnormalized differences between zeros of $L$-functions,
	\textit{Compos. Math.} \textbf{151} no.2 (2015), 230--252.



  \bibitem{G1989} S. M. Gonek, On negative moments of the Riemann zeta-function, \textit{Mathematika} \textbf{36} (1989), 71--88.

	\bibitem{GHK2007} S. M. Gonek, C. P. Hughes, and J. P. Keating,
	A hybrid Euler-Hadamard product for the Riemann zeta-function,
	\textit{Duke Math. J.} no.3 \textbf{136} (2007), 507--549.

	\bibitem{Good} A. Good, The square mean of Dirichlet series associated with cusp forms,
	\textit{Mathematika} \textbf{29} (1982), no. 2, 278–295.

	\bibitem{H2013} A. J. Harper, Sharp conditional bounds for moments of the Riemann zeta function,
	preprint, \texttt{arXiv:1305.4618}.

	\bibitem{Heap13} W. Heap, Moments of the Dedekind zeta function and other non-primitive $L$-functions.
	\textit{Math. Proc. Cambridge Philos. Soc.} 2013, 1--29.


  \bibitem{He2021} W. Heap, On the splitting conjecture in the hybrid model for the Riemann zeta-function,
  preprint, \texttt{arXiv:2102.02092}.

	\bibitem{HRS2019} W. Heap, M. Radziwi\l\l, and K. Soundararajan,
	Sharp upper bounds for fractional moments of the Riemann zeta-function,
	\textit{Quart. J. Math.} \textbf{70} (2019), 1387--1396.
	
	\bibitem{HeapSound}W. Heap and K. Soundararajan, Lower bounds for moments of zeta and $ L $-functions revisited, preprint \texttt{arXiv:2007.13154}.
	
	\bibitem{HK2018} C. Heuberger and S. Kropf, Higher dimensional quasi-power theorem and Berry-Esseen inequality,
	\textit{Monatsh. Math.} \textbf{187} (2018), 293--314.

	\bibitem{HW2020} P.-H. Hsu and P.-J. Wong, On Selberg's central limit theorem for Dirichlet $L$-functions, 
  \textit{Journal de Th\'eorie des Nombres de Bordeaux} \textbf{32} (2020), 685–710.

	\bibitem{HH1996} H. H. Hwang, Large deviations for combinatorial distributions. I: Central limit theorems,
	\textit{Ann. Appl. Probab.} \textbf{6} (1996), no. 1, 297--319.

	\bibitem{II2019} S. Inoue, On the logarithm of the Riemann zeta-function and its iterated integrals,
	preprint, \texttt{arXiv:1909.03643}.

	\bibitem{Ju1983} M. Jutila, On the value distribution of the zeta-function on the critical line,
	\textit{Bull. Lond. Math. Soc.} \textbf{15} (1983), 513--518.

\bibitem{Kim06} H. Kim, A note on Fourier coefficients of cusp forms on $GL_n$, Forum Math., \textbf{18} (1) (2006), 115-119.

\bibitem{KimSarnak}H. Kim and P. Sarnak, Refined estimates towards the Ramanujan and Selberg conjectures,
J. Amer. Math. Soc., \textbf{16} (1) (2003), 175-181.



	\bibitem{LLR2019} Y. Lamzouri, S. Lester, and M. Radziwi\l\l, Discrepancy bounds for the distribution
	of the Riemann zeta-function and applications, \textit{J. Anal. Math.} \textbf{139} (2019), no. 2, 453--494.

	\bibitem{L1987} A. Laurin\v{c}ikas, A limit theorem for the Riemann zeta-function on the critical line II,
	\textit{Litovsk. Mat. Sb.} \textbf{27} (3) (1987), 91--110.

	\bibitem{Liu05} J. Liu, Y. Wang and Y. Ye, A proof of Selberg's orthogonality for automorphic $L$-functions,
	\textit{Manuscripta Math.} \textbf{118} (2005), 135--149.

	\bibitem{LiuYe} J. Liu and Y.B. Ye, Selberg's orthogonality conjecture for automorphic $L$-functions,
	\textit{Amer. J. Math.} (4) \textbf{127} (2005), 837--849.

	\bibitem{LiuYe06} J. Liu and Y.B. Ye, Zeros of automorphic $L$-functions and noncyclic base change
	Number Theory, Dev. Math., \textbf{15}, Springer, New York (2006), 119--152.

	\bibitem{Luo95} W. Luo, Zeros of Hecke $L$-functions associated with cusp forms, \textit{Acta Arith.}, \textbf{71} (1995), 139--158.

  \bibitem{LRS1995} W. Luo, Z. Rudnick, and P. Sarnak, On Selberg's eigenvalue conjecture,
  \textit{Geom. Funct. Anal.} \textbf{5}, no. 2 (1995), 387--401.

	\bibitem{MT}M. B. Milinovich and C. L. Turnage-Butterbaugh. Moments of products of automorphic $L$-functions,
	\textit{J. Number Theory} \textbf{139} (2014), 175--204.

	\bibitem{MV} H. L. Montgomery and R. C. Vaughan, \textit{Multiplicative Number Theory I. Classical Theory},
	Cambridge Studies in Advanced Mathematics, Cambridge University press, 2007.

	\bibitem{M1970} Y. Motohashi, A note on the mean value of the Dedekind zeta function of the quadratic field,
	\textit{Math. Ann.} \textbf{188} (1970), 123--127.

	\bibitem{MP} M. R. Murty, \textit{Problems and Analytic Number Theory},
	Second Edition, Graduate Texts in Mathematics, Vol. 206, Springer, 2008.

	\bibitem{Naj} J. Najnudel, Exponential moments of the argument of the Riemann zeta function on the critical line,
	\textit{Mathematika} \textbf{66} (2020), 612--621.

	\bibitem{Ra2011} M. Radziwi\l\l, Large deviations in Selberg's central limit theorem, \texttt{arXiv:1108.5092}.

	\bibitem{RS2017} M. Radziwi\l\l \ and K. Soundararajan, Selberg's central limit theorem for $\log{|\zeta(\tfrac{1}{2} + it)|}$,
	\textit{Enseign. Math.} (2) \textbf{63} (2017), 1--19.

	\bibitem{RS1996} Z. Rudnick and P. Sarnak, Zeros of principal $L$-functions and random matrix theory,
	\textit{Duke Math. J.} no.2 \textbf{81} (1996), 269--322.

	\bibitem{SSmaass} A. Sankaranarayanan and J. Sengupta.  Zero-density estimate of $L$-functions attached to Maa\ss \ forms,
	\textit{Acta Arith.} (3) \textbf{127} (2007), 273--284.



	\bibitem{SCR} A. Selberg, Contributions to the theory of the Riemann zeta-function,
	Avhandl. Norske Vid.-Akad. Olso I. Mat.-Naturv. Kl., no.1;
	Collected Papers, Vol. 1, New York: Springer Verlag. 1989, 214--280.


	\bibitem{S1992} A. Selberg, Old and new conjectures and results about a class of Dirichlet series,
	\textit{Proceedings of the Amalfi Conference on Analytic Number Theory}, 367--385, 1992.

	\bibitem{SM2009} K. Soundararajan, Moments of the Riemann zeta-function,
	\textit{Ann. of Math.} (2) \textbf{170} (2009), 981--993.


	\bibitem{Te1988} G. Tenenbaum, Sur la distribution conjointe des deux fonctions ``nombre de facteurs premiers," 
	\textit{Aeq. Math.} \textbf{35} (1988), 55--68.

  

	\bibitem{Top} B. Topacogullari, The fourth moment of individual Dirichlet $L$-functions on the critical line,
	\texttt{arXiv.1909.11517}, to appear in \textit{Math. Z}.


	\bibitem{KTDT} K. M. Tsang, \textit{The distribution of the values of the Riemann zeta function}, PhD thesis,
	Princeton University, Princeton, NJ, 1984.



	\bibitem{Zhang1} Q. Zhang, Integral mean values of modular $L$-functions, \textit{J. Number Theory} \textbf{115} (2005), no. 1, 100–122.


	\bibitem{Zhang2} Q. Zhang, Integral mean values of Maass $L$-functions, \textit{Int. Math. Res. Not.} \textbf{2006}, Art. ID 41417, 19 pp.

\end{thebibliography}
\end{document}